\documentclass{amsart}
\usepackage[dvips,final]{graphics}
\usepackage{array}
\usepackage{arydshln}
\usepackage[makeroom]{cancel}
 \usepackage[all]{xy}
 \usepackage{url}
\usepackage{multirow, blkarray}
\usepackage{booktabs}
\usepackage{textcomp}
 \usepackage[final]{epsfig}
 \usepackage{color}
\usepackage[T1]{fontenc}      
\usepackage[english,french]{babel}
\usepackage[utf8]{inputenc}
\usepackage{blindtext}

\usepackage{amsfonts,amscd,array, mathdots, epigraph}
\usepackage{amsmath}
\usepackage{amssymb}
\usepackage{amsthm}
\usepackage{mathrsfs}
\usepackage{stmaryrd}
\usepackage{slashbox}
\usepackage{diagbox}
\usepackage{enumitem}

\usepackage{ulem}
\usepackage{tikz}
\usepackage{xcolor}
\usepackage{multicol}
\definecolor{ufogreen}{rgb}{0.24, 0.82, 0.44}

\begin{document}


\newtheorem{theorem}{Théorème}[section]
\newtheorem{theore}{Théorème}
\newtheorem{definition}[theorem]{Définition}
\newtheorem{proposition}[theorem]{Proposition}
\newtheorem{corollary}[theorem]{Corollaire}
\newtheorem*{con}{Conjecture}
\newtheorem*{remark}{Remarque}
\newtheorem*{remarks}{Remarques}
\newtheorem*{pro}{Problème}
\newtheorem*{examples}{Exemples}
\newtheorem*{example}{Exemple}
\newtheorem{lemma}[theorem]{Lemme}


\title[Taille et irréductibilité des solutions monomiales minimales dans $SL_{2}(\mathbb{Z}/N\mathbb{Z})$]{Étude des liens entre la taille et l'irréductibilité des solutions monomiales minimales dans $SL_{2}(\mathbb{Z}/N\mathbb{Z})$}

\author{Flavien Mabilat}

\date{}


\address{
}
\def\emailaddrname{{\itshape Courriel}}
\email{flavien.mabilat@univ-reims.fr}

\maketitle

\selectlanguage{french}
\begin{abstract}
Cet article a pour objet l'étude de certains $n$-uplets d'éléments d'un anneau $\mathbb{Z}/N\mathbb{Z}$ liés à la combinatoire des sous-groupes de congruence du groupe modulaire. Plus précisément, on va s'intéresser ici à la notion de solutions monomiales minimales. Celles-ci sont les solutions d'une équation matricielle (apparaissant également lors de l'étude des frises de Coxeter), modulo un entier $N$, dont toutes les composantes sont identiques et minimales pour cette propriété. Notre objectif ici est d'étudier les liens entre la taille des solutions monomiales minimales et une certaine propriété d'irréductibilité qui est centrale dans l'étude de la combinatoire du groupe modulaire. En particulier, on obtiendra une majoration de la taille des solutions monomiales irréductibles et on démontrera que certaines tailles entraînent automatiquement l'irréductibilité.
\\
\end{abstract}

\selectlanguage{english}
\begin{abstract}
This article aims to study some $n$-tuples of elements belonging to a ring $\mathbb{Z}/N\mathbb{Z}$ related to the combinatorics of congruence subgroups of the modular group. More precisely, we will focus here on the notion of minimal monomial solutions. These are the solutions of a matrix equation (also appearing during the study of Coxeter's friezes), modulo an integer $N$, all of whose components are identical and minimal for this property. Our objective here is to study the links between the size of minimal monomial solutions and a property of irreducibility which is central in the study of the combinatorics of the modular group. In particular, we will obtain an upper bound of the size of irreducible monomial solutions and we will prove that some sizes automatically lead to irreducibility.

\end{abstract}

\selectlanguage{french}

\thispagestyle{empty}

\noindent {\bf Mots clés:} groupe modulaire; solution monomiale minimale; irréductibilité; classes modulo $N$  
\\
\begin{flushright}
\og \textit{Une idée que j'ai, il faut que je la nie: c'est ma manière de l'essayer.} \fg
\\Alain, \textit{Histoire de mes pensées.}
\end{flushright}

\section{Introduction}

Depuis plus de deux siècles, le groupe modulaire \[SL_{2}(\mathbb{Z})=
\left\{
\begin{pmatrix}
a & b \\
c & d
   \end{pmatrix}
 \;\vert\;a,b,c,d \in \mathbb{Z},\;
 ad-bc=1
\right\}\] et ses nombreux sous-groupes ont été l'objet de multiples travaux et ont attiré l'attention de très nombreux mathématiciens. Grâce à ces derniers, on dispose d'une connaissance remarquable de ce groupe ainsi que d'une myriade de résultats fascinants sur ses propriétés. Parmi ceux-ci, l'un des plus notables, et des plus anciens, est la possibilité d'écrire $SL_{2}(\mathbb{Z})$ comme un groupe engendré par deux éléments. On peut envisager plusieurs façons de générer le groupe modulaire de cette manière. Cela dit, dans ce qui va suivre, on va s'intéresser uniquement aux deux matrices suivantes :
\[T=\begin{pmatrix}
 1 & 1 \\[2pt]
    0    & 1 
   \end{pmatrix}, S=\begin{pmatrix}
   0 & -1 \\[2pt]
    1    & 0 
   \end{pmatrix}.
 \] 

\noindent En effet, une fois ces deux générateurs choisis, on peut montrer (voir par exemple l'introduction de \cite{M}) que, pour toute matrice $A$ du groupe modulaire, il existe un entier strictement positif $n$ et des entiers strictement positifs $a_{1},\ldots,a_{n}$ tels que : \[A=T^{a_{n}}ST^{a_{n-1}}S\cdots T^{a_{1}}S=\begin{pmatrix}
   a_{n} & -1 \\[4pt]
    1    & 0 
   \end{pmatrix}
\begin{pmatrix}
   a_{n-1} & -1 \\[4pt]
    1    & 0 
   \end{pmatrix}
   \cdots
   \begin{pmatrix}
   a_{1} & -1 \\[4pt]
    1    & 0 
    \end{pmatrix}:=M_{n}(a_{1},\ldots,a_{n}).\]
		
\noindent Toutefois, il convient de remarquer que ce type d'écriture n'est pas unique. En effet, on dispose, par exemple, de l'égalité suivante : \[-Id=M_{3}(1,1,1)=M_{4}(1,2,1,2).\]

Bien que l'attrait pour le groupe modulaire ne se soit jamais tari et que la possibilité d'écrire les éléments de ce dernier sous la forme $M_{n}(a_{1},\ldots,a_{n})$ (avec $a_{1},\ldots,a_{n}$ des entiers strictement positifs) soit connu depuis longtemps, l'existence de cette écriture avait jusqu'à récemment été considéré comme une curiosité et n'avait pas été réellement exploité. Toutefois, on observe depuis quelque années un regain notable d'intérêt pour celle-ci, du fait notamment de l'apparition des matrices $M_{n}(a_{1},\ldots,a_{n})$ dans de nombreux domaines mathématiques. Elles interviennent en effet dans l'étude d'un très grand nombre d'objets, depuis l'expression des réduites des fractions continues de Hirzebruch-Jung jusqu'à la construction des frises de Coxeter en passant par l'écriture matricielle des équations de Sturm-Liouville discrètes (voir par exemple l'introduction de \cite{O}). En particulier, l'une des questions principales soulevées par l'existence de cette écriture est la recherche des diverses expressions associées à une même matrice et en particulier à $\pm Id$. En d'autres termes, on s'intéresse  à la résolution sur $\mathbb{N}^{*}$ de l'équation suivante (parfois appelée équation de Conway-Coxeter) : \begin{equation}
\label{a}
\tag{$E$}
M_{n}(a_1,\ldots,a_n)=\pm Id.
\end{equation} 

\noindent V.\ Ovsienko (voir \cite{O} Théorèmes 1 et 2) a donné une construction récursive des solutions de \eqref{a} sur les entiers naturels non nuls et a fourni une description combinatoire des solutions en terme de découpages de polygones (généralisant ainsi un théorème de Conway-Coxeter, voir \cite{CoCo} et \cite{MO} Théorème 3.3). On dispose également de résultats analogues pour les solutions des équations $M_{n}(a_1,\ldots,a_n)=\pm S$ et $M_{n}(a_1,\ldots,a_n)=\pm T$ (voir \cite{M} Théorèmes 2.2 et 2.4).
\\
\\ \indent À la lueur de ces résultats, plusieurs pistes de généralisations sont envisageables. Une de celles-ci, liée avec la construction des frises de Coxeter, consiste à résoudre \eqref{a} sur d'autres ensembles (voir par exemple \cite{C,CH}). C'est cette piste que l'on va exploiter ici en se plaçant sur les anneaux $\mathbb{Z}/N\mathbb{Z}$, c'est-à-dire que l'on va s'intéresser aux solutions sur $\mathbb{Z}/N\mathbb{Z}$ de l'équation :
\begin{equation}
\label{p}
\tag{$E_{N}$}
M_{n}(a_1,\ldots,a_n)=\pm Id.
\end{equation} 
\noindent On dira, en particulier, qu'une solution de \eqref{p} est de taille $n$ si cette solution est un $n$-uplet d'éléments de $\mathbb{Z}/N\mathbb{Z}$. En effet, outre l'intérêt intrinsèque de l'étude de \eqref{a} sur ces ensembles, les solutions de cette dernière permettent d'avoir des informations sur les $n$-uplets d'entiers strictement positifs pour lesquels les matrices $M_{n}(a_{1},\ldots,a_{n})$ appartiennent aux sous-groupes de congruence ci-dessous :
\[\hat{\Gamma}(N):=\{A \in SL_{2}(\mathbb{Z})~{\rm tel~que}~A= \pm Id~( {\rm mod}~N)\}.\] 
\noindent Notons que si on adopte ce point de vue, le théorème d'Ovsienko s'intéresse aux $n$-uplets d'entiers strictement positifs pour lesquels $M_{n}(a_{1},\ldots,a_{n})$ appartient au centre de $SL_{2}(\mathbb{Z})$.
\\
\\ \indent On dispose déjà de nombreux résultats sur l'équation \eqref{p}, comme par exemple le nombre de solutions de taille fixée lorsque $N$ est premier ou des descriptions combinatoires des solutions pour les petites valeurs de $N$. Cela dit, la plupart des éléments obtenus reposent sur une notion d'irréductibilité basée sur une opération entre uplets d'un même ensemble (voir la section suivante). Cette dernière, qui a été introduite à l'origine pour l'étude des frises de Coxeter (voir \cite{C}), nous permet de restreindre notre étude à la recherche d'informations sur les solutions irréductibles. Celles-ci sont d'autant plus intéressante à considérer que leur nombre est fini (voir \cite{M4} Théorème 1.1). Cela amène alors naturellement à s'intéresser à des classes particulières de solutions. Parmi celles-ci, on peut notamment citer les solutions $\overline{k}$-monomiales minimales qui sont les solutions pour lesquels toutes les composantes sont identiques et égales à $\overline{k}$ et dont la taille est la plus petite possible. On connaît en effet de nombreuses propriétés de ces solutions (voir notamment la section suivante) et elles sont à la source de plusieurs résultats intéressants d'irréductibilité. L'objectif ici est d'étudier les rapports entre l'irréductibilité des solutions monomiales minimales et leur taille. Après avoir redonné les définitions essentielles dans la section \ref{RP}, on montrera en particulier dans la section \ref{trois} que les solutions monomiales minimales irréductibles de \eqref{p} sont de taille inférieure à $N$, sauf si $N=2$ ou si $N=3m$ avec $m$ impair premier avec 3. Ensuite, dans la section \ref{quatre}, on établira que, pour $N \neq 2m$ avec $m$ impair, les solutions monomiales minimales de \eqref{p} de taille impaire sont irréductibles tandis qu'on montrera dans la section \ref{cinq} que les solutions monomiales minimales de \eqref{p} de taille $N$ sont irréductibles, sauf éventuellement si $N$ est égal à 2 ou est de la forme $2 \times 3^{a} \times b$ avec $a \geq 1$, $b>1$ et $b$ impair non divisible par 3.

\section{Définitions et résultats principaux}
\label{RP}    

Dans cette section, on va énoncer les résultats principaux, évoqués dans l'introduction, qui seront démontrés dans les parties suivantes. Cela dit, avant de rentrer dans les détails de ces derniers, on a besoin d'un certain nombres de définitions et de notations qui nous seront très utiles pour la suite. Dans tout cet article, $N$ est un entier naturel supérieur ou égal à 2. On aura régulièrement besoin de considérer des classes d'entiers modulo $N$ et modulo des diviseurs de $N$. Aussi, s'il n'y a pas d'ambiguïté sur $N$, on utilisera la notation $\overline{a}:=a+N\mathbb{Z}$ (avec $a \in \mathbb{Z}$) pour désigner les classes modulo $N$ et on écrira $a+\frac{N}{k}\mathbb{Z}$ pour représenter les classes modulo un diviseur de $N$. $\mathbb{P}$ désigne l'ensemble des nombres premiers. Dans la suite, les pgcd et les ppcm d'entiers strictement positifs, qui sont définis au signe près, seront pris dans $\mathbb{N}^{*}$. De plus, si $(x_{i})_{i \in I}$ est une famille d'entiers, on pose $\prod_{j \in \emptyset} x_{j}:=1$.
\\
\\Avant d'introduire la notion centrale de solutions irréductibles, on a besoin de deux définitions.

\begin{definition}[\cite{C}, lemme 2.7]
\label{21}

Soient $(\overline{a_{1}},\ldots,\overline{a_{n}}) \in (\mathbb{Z}/N \mathbb{Z})^{n}$ et $(\overline{b_{1}},\ldots,\overline{b_{m}}) \in (\mathbb{Z}/N \mathbb{Z})^{m}$. On définit l'opération ci-dessous: \[(\overline{a_{1}},\ldots,\overline{a_{n}}) \oplus (\overline{b_{1}},\ldots,\overline{b_{m}})= (\overline{a_{1}+b_{m}},\overline{a_{2}},\ldots,\overline{a_{n-1}},\overline{a_{n}+b_{1}},\overline{b_{2}},\ldots,\overline{b_{m-1}}).\] Le $(n+m-2)$-uplet obtenu est appelé la somme de $(\overline{a_{1}},\ldots,\overline{a_{n}})$ avec $(\overline{b_{1}},\ldots,\overline{b_{m}})$.

\end{definition}

\begin{examples}

{\rm On donne ci-dessous quelques exemples de sommes :
\begin{itemize}
\item $(\overline{1},\overline{1},\overline{3}) \oplus (\overline{-2},\overline{0},\overline{2})=(\overline{3},\overline{1},\overline{1},\overline{0})$;
\item $(\overline{2},\overline{2},\overline{1},\overline{0}) \oplus (\overline{1},\overline{-1},\overline{1})=(\overline{3},\overline{2},\overline{1},\overline{1},\overline{-1})$;
\item $(\overline{1},\overline{0},\overline{2},\overline{3}) \oplus (\overline{2},\overline{4},\overline{1},\overline{-1},\overline{5})=(\overline{6},\overline{0},\overline{2},\overline{5},\overline{4},\overline{1},\overline{-1})$;
\item $n \geq 2$, $(\overline{a_{1}},\ldots,\overline{a_{n}}) \oplus (\overline{0},\overline{0}) = (\overline{0},\overline{0}) \oplus (\overline{a_{1}},\ldots,\overline{a_{n}})=(\overline{a_{1}},\ldots,\overline{a_{n}})$.
\end{itemize}
}
\end{examples}

Cette notion de somme est particulièrement utile pour l'étude des solutions de l'équation \eqref{p} car elle possède la propriété remarquable suivante : si $(\overline{b_{1}},\ldots,\overline{b_{m}})$ est une solution de \eqref{p} alors la somme $(\overline{a_{1}},\ldots,\overline{a_{n}}) \oplus (\overline{b_{1}},\ldots,\overline{b_{m}})$ est une solution de \eqref{p} si et seulement si $(\overline{a_{1}},\ldots,\overline{a_{n}})$ est une solution de \eqref{p} (voir \cite{C,WZ} et \cite{M0} proposition 3.7). En revanche, $\oplus$ n'est ni commutative ni associative (voir \cite{WZ} exemple 2.1) et les uplets d'éléments de $\mathbb{Z}/N \mathbb{Z}$ différents de $(\overline{0},\overline{0})$ n'ont pas d'inverse pour $\oplus$.

\begin{definition}[\cite{C}, définition 2.5]
\label{22}

Soient $(\overline{a_{1}},\ldots,\overline{a_{n}}) \in (\mathbb{Z}/N \mathbb{Z})^{n}$ et $(\overline{b_{1}},\ldots,\overline{b_{n}}) \in (\mathbb{Z}/N \mathbb{Z})^{n}$. On dit que $(\overline{a_{1}},\ldots,\overline{a_{n}}) \sim (\overline{b_{1}},\ldots,\overline{b_{n}})$ si $(\overline{b_{1}},\ldots,\overline{b_{n}})$ est obtenu par permutations circulaires de $(\overline{a_{1}},\ldots,\overline{a_{n}})$ ou de $(\overline{a_{n}},\ldots,\overline{a_{1}})$.

\end{definition}

On constate assez rapidement que $\sim$ est une relation d'équivalence sur les $n$-uplets d'éléments de $\mathbb{Z}/N \mathbb{Z}$ (voir \cite{WZ} lemme 1.7). De plus, $\sim$ conserve la propriété d'être solution de \eqref{p}, c'est-à-dire que si $(\overline{a_{1}},\ldots,\overline{a_{n}}) \sim (\overline{b_{1}},\ldots,\overline{b_{n}})$ alors $(\overline{a_{1}},\ldots,\overline{a_{n}})$ est solution de \eqref{p} si et seulement si $(\overline{b_{1}},\ldots,\overline{b_{n}})$ l'est aussi (voir \cite{C} proposition 2.6).
\\
\\Muni de ces deux définitions, on peut maintenant définir la notion d'irréductibilité annoncée.

\begin{definition}[\cite{C}, définition 2.9]
\label{23}

Une solution $(\overline{c_{1}},\ldots,\overline{c_{n}})$ avec $n \geq 3$ de \eqref{p} est dite réductible s'il existe une solution de \eqref{p} $(\overline{b_{1}},\ldots,\overline{b_{l}})$ et un $m$-uplet $(\overline{a_{1}},\ldots,\overline{a_{m}})$ d'éléments de $\mathbb{Z}/N \mathbb{Z}$ tels que \begin{itemize}
\item $(\overline{c_{1}},\ldots,\overline{c_{n}}) \sim (\overline{a_{1}},\ldots,\overline{a_{m}}) \oplus (\overline{b_{1}},\ldots,\overline{b_{l}})$;
\item $m \geq 3$ et $l \geq 3$.
\end{itemize}
Une solution est dite irréductible si elle n'est pas réductible.

\end{definition}

\begin{remark} 

{\rm $(\overline{0},\overline{0})$ n'est pas considérée comme une solution irréductible de \eqref{p}.}

\end{remark}

\indent Une fois ce concept introduit, la recherche d'informations sur les solutions irréductibles devient naturellement un des principaux axes d'étude des solutions de \eqref{p}. Cela conduit notamment à s'intéresser à certaines classes particulières de solutions possédant de bonnes propriétés d'irréductibilité. En particulier, on a défini les solutions ci-dessous :

\begin{definition}[\cite{M1}, définition 3.9]
\label{24}

i)~Soient $n \in \mathbb{N}^{*}$ et $\overline{k} \in \mathbb{Z}/N\mathbb{Z}$. On appelle solution $(n,\overline{k})$-monomiale un $n$-uplet d'éléments de $\mathbb{Z}/ N \mathbb{Z}$ constitué uniquement de $\overline{k}$ et solution de \eqref{p}.
\\
\\ ii)~On appelle solution monomiale une solution pour laquelle il existe $m \in \mathbb{N}^{*}$ et $\overline{l} \in \mathbb{Z}/N\mathbb{Z}$ tels qu'elle est $(m,\overline{l})$-monomiale.
\\
\\ iii)~On appelle solution $\overline{k}$-monomiale minimale une solution $(n,\overline{k})$-monomiale avec $n$ le plus petit entier pour lequel il existe une solution $(n,\overline{k})$-monomiale.
\\
\\ iv)~On appelle solution monomiale minimale une solution $\overline{k}$-monomiale minimale pour un $\overline{k} \in \mathbb{Z}/N\mathbb{Z}$.

\end{definition}

Notons que pour tout $\overline{k} \in \mathbb{Z}/N\mathbb{Z}$ la solution $\overline{k}$-monomiale minimale de \eqref{p} existe toujours puisque $M_{1}(\overline{k})$ est d'ordre fini dans ${\rm PSL}_{2}(\mathbb{Z}/N\mathbb{Z})$. 
\\
\\ \indent On dispose d'un certain nombre de résultats d'irréductibilité pour ces solutions ainsi que de plusieurs théorèmes donnant des indications sur la taille de ces dernières (voir \cite{M0,M1,M2} et la section \ref{pre}). Ici, notre objectif est d'obtenir des résultats qui mêlent les deux types d'information. À la lueur de cet objectif, deux directions émergent, d'une part trouver des informations générales sur la taille des solutions monomiales minimales irréductibles, et, d'autre part, chercher des tailles entraînant automatiquement l'irréductibilité ou la réductibilité des solutions monomiales minimales. 
\\
\\ \indent On commencera par s'intéresser à la première piste d'étude en démontrant dans la section \ref{trois} le résultat suivant :

\begin{theorem}
\label{241}

Soit $N$ un entier différent de 2 qui n'est pas de la forme $3m$ avec $m$ premier avec 6. Les solutions monomiales minimales irréductibles de \eqref{p} sont de taille inférieure ou égale à $N$.

\end{theorem}

Puis, en utilisant le cas mis de côté par ce dernier, on résoudra la conjecture 2 de \cite{M0}. Dans cet article, on avait en effet émis deux conjectures portant sur les solutions irréductibles de \eqref{p}. La première indiquait que le nombre de solutions irréductibles sur $\mathbb{Z}N\mathbb{Z}$ était fini et a été résolue positivement dans \cite{M4}. La deuxième s'intéressait à l'existence d'une constante $K$ vérifiant la propriété suivante : pour tout $N \geq 2$, $\ell_{N} \leq N+K$, où $\ell_{N}$ est la taille maximale des solutions irréductibles de $(E_{N})$. C'est la question de l'existence de cette constante que l'on va trancher ici. Plus précisément, on effectuera la preuve du résultat général suivant :

\begin{theorem}
\label{25}

Il n'existe pas de constante $K \in \mathbb{N}^{*}$ tel que pour tout $N \geq 2$ les solutions irréductibles de \eqref{p} soient de taille inférieure à $N+K$.

\end{theorem}

\noindent Ensuite, on s'intéressera à certaines tailles particulières.

\begin{theorem}
\label{26}

Soit $N>2$. 
\\
\\i) On suppose que $N$ n'est pas de la forme $2m$ avec $m$ impair. Si \eqref{p} possède des solutions monomiales minimales de taille impaire alors elles sont irréductibles.
\\
\\ii) On suppose $N=2m$ avec $m$ impair. Si \eqref{p} possède des solutions monomiales minimales de taille $l$ impaire divisible par 9 alors elles sont irréductibles.

\end{theorem}

\begin{theorem}
\label{27}

Soit $N>2$. 
\\
\\i) Si $N$ n'est pas de la forme $2 \times 3^{a} \times b$, avec $a \geq 1$ et $b>1$ impair non divisible par 3, alors toutes les solutions monomiales minimales non nulles de \eqref{p} de taille $N$ sont irréductibles.
\\
\\ii) Si $N=2 \times 3^{a} \times b$, avec $a \geq 1$ et $b>1$ impair non divisible par 3, alors il existe des solutions monomiales minimales non nulles de \eqref{p} de taille $N$ réductibles.

\end{theorem}

Le théorème \ref{26} sera prouvé dans la section \ref{quatre} tandis que la preuve du théorème \ref{27} est donnée dans la section \ref{cinq}. 

\section{Majoration de la taille des solutions monomiales minimales irréductibles}
\label{trois}

L'objectif de cette section est de démontrer les théorèmes \ref{241} et \ref{25} et de donner un certain nombre de résultats intermédiaires qui nous seront utiles dans tout l'article.

\subsection{Résultats préliminaires}
\label{pre}

Le but de cette sous-partie est d'énoncer plusieurs résultats préliminaires concernant la taille des solutions monomiales minimales. Le premier d'entre eux concerne le cas de l'équation $(E_{p})$ avec $p$ premier et est issu de la modification d'un résultat sur les ordres des éléments de $SL_{2}(\mathbb{Z}/N\mathbb{Z})$ donné dans \cite{CGL} (page 216).

\begin{theorem}[\cite{M1} Théorème 3.4]
\label{31}

Soient $p \in \mathbb{P}$ impair et $\overline{k} \in \mathbb{Z}/p\mathbb{Z}$. 
\begin{itemize}
\item Si $\overline{k}=\pm \overline{2}$ alors la taille de la solution $\overline{k}$-monomiale minimale de $(E_{p})$ est égale à $p$.
\item Si $\overline{k}\neq \pm \overline{2}$ alors la taille de la solution $\overline{k}$-monomiale minimale de $(E_{p})$ divise $\frac{p+1}{2}$ ou $\frac{p-1}{2}$.
\end{itemize}

\end{theorem}

\noindent On dispose également du résultat général suivant : 

\begin{theorem}[\cite{M0} Théorème 2.6 et corollaire 3.22]
\label{30}

Soit $N \geq 3$. La solution $\pm \overline{2}$-monomiale minimale de \eqref{p} est irréductible de taille $N$. De plus, $M_{N}(\overline{2},\ldots,\overline{2})=Id$ et $M_{N}(\overline{-2},\ldots,\overline{-2})=(-1)^{N}Id$.

\end{theorem}

On va maintenant considérer plusieurs résultats abordant les liens entre la taille d'une solution et les diviseurs de $N$.

\begin{theorem}[\cite{M1} Théorème 2.5]
\label{32}
Soit $N$ un entier pair supérieur à 4. On a deux cas :
\begin{itemize}
\item Si $4$ divise $N$ alors la solution $\overline{\frac{N}{2}}$-monomiale minimale de \eqref{p} est de taille 4 et \[M_{4}\left(\overline{\frac{N}{2}},\overline{\frac{N}{2}},\overline{\frac{N}{2}},\overline{\frac{N}{2}}\right)=Id;\]
\item Si 4 ne divise pas $N$ alors la solution $\overline{\frac{N}{2}}$-monomiale minimale de \eqref{p} est de taille 6 et \[M_{6}\left(\overline{\frac{N}{2}},\ldots,\overline{\frac{N}{2}}\right)=-Id.\]
\end{itemize}
\noindent De plus, dans les deux cas, la solution est irréductible.

\end{theorem}

\begin{lemma}[\cite{M2} lemme 3.1]
\label{33}

Soient $N$ un entier supérieur à 2 et $\overline{k}$ un élément de $\mathbb{Z}/N\mathbb{Z}$. Soit $d$ un diviseur de $N$. La taille de la solution $\overline{k}$-monomiale minimale de \eqref{p} est un multiple de la taille de la solution $(k+d\mathbb{Z})$-monomiale minimale de $(E_{d})$.

\end{lemma}

\begin{lemma}[\cite{M2} lemme 5.3]
\label{34}

Soient $p$ un nombre premier, $k$ un entier et $n \in \mathbb{N}^{*}$. Soit $r$ la taille de la solution $\overline{k}$-monomiale minimale de $(E_{p^{n}})$.
\\
\\i) La taille de la solution $\overline{k}$-monomiale minimale de $(E_{p^{n+1}})$ est égale à $r$ ou à $pr$.
\\
\\ii) Si $p$ est impair et si $r$ est pair alors $M_{r}(\overline{k},\ldots,\overline{k})=-Id$.

\end{lemma}

\begin{lemma}
\label{35}

Soient $k$ un entier et $n \geq 2$. Soit $r$ la taille de la solution $\overline{k}$-monomiale minimale de $(E_{2^{n}})$. Si $\overline{k} \neq \overline{0}$ et si $r$ est pair alors $M_{r}(\overline{k},\ldots,\overline{k})=Id$.

\end{lemma}

\begin{proof}

On raisonne par récurrence sur $n$. 
\\
\\Si $n=2$ alors les hypothèses impliquent $k+4\mathbb{Z}=2+4\mathbb{Z}$ et $r=4$ et on a $M_{4}(k+4\mathbb{Z},\ldots,k+4\mathbb{Z})=Id$ (Théorème \ref{32}). 
\\
\\Supposons qu'il existe un $n \geq 2$ tel que toutes les solutions non nulles de taille paire de $(E_{2^{n}})$  vérifient la propriété souhaitée. Soit $k$ un entier non divisible par $N=2^{n+1}$ tel que la solution $(k+2^{n+1}\mathbb{Z})$-monomiale minimale de $(E_{2^{n+1}})$ est de taille paire égale à $l$. Il existe $\epsilon \in \{-1,1\}$ tel que 
\[M_{l}(k+2^{n+1}\mathbb{Z},\ldots,k+2^{n+1}\mathbb{Z})=(\epsilon+2^{n+1}\mathbb{Z})Id.\]

\noindent Si $k+2^{n}\mathbb{Z}=0+2^{n}\mathbb{Z}$ alors on a $k+2^{n+1}\mathbb{Z}=\frac{N}{2}+2^{n+1}\mathbb{Z}$. On a alors $l=4$ (voir Théorème \ref{32}) et $M_{l}(k+2^{n+1}\mathbb{Z},\ldots,k+2^{n+1}\mathbb{Z})=Id$. On suppose donc maintenant $k+2^{n}\mathbb{Z}\neq 0+2^{n}\mathbb{Z}$. 
\\
\\Soit $r$ la taille de la solution $(k+2^{n}\mathbb{Z})$-monomiale minimale de $(E_{2^{n}})$. Par le lemme \ref{34}, $l=r$ ou $l=2r$. On distingue les deux cas :
\begin{itemize}
\item Si $l=r$ alors $r$ est pair. Par l'hypothèse de récurrence, $M_{l}(k+2^{n}\mathbb{Z},\ldots,k+2^{n}\mathbb{Z})=(1+2^{n}\mathbb{Z})Id$. De plus, comme $2^{n}$ divise $2^{n+1}$, $M_{l}(k+2^{n}\mathbb{Z},\ldots,k+2^{n}\mathbb{Z})=(\epsilon+2^{n}\mathbb{Z})Id$. Donc, ($1+2^{n}\mathbb{Z})=(\epsilon+2^{n}\mathbb{Z})$. Comme $2^{n} \geq 3$, on a $\epsilon=1$.
\item Si $l=2r$. $M_{l}(k+2^{n}\mathbb{Z},\ldots,k+2^{n}\mathbb{Z})=(M_{r}(k+2^{n}\mathbb{Z},\ldots,k+2^{n}\mathbb{Z}))^{2}=(1+2^{n}\mathbb{Z})Id$. De plus, comme $2^{n}$ divise $2^{n+1}$, $M_{l}(k+2^{n}\mathbb{Z},\ldots,k+2^{n}\mathbb{Z})=(\epsilon+2^{n}\mathbb{Z})Id$. Donc, ($1+2^{n}\mathbb{Z})=(\epsilon+2^{n}\mathbb{Z})$. Comme $2^{n} \geq 3$, on a $\epsilon=1$.
\end{itemize}

\noindent Par récurrence, le résultat est démontré.

\end{proof}

\begin{lemma}
\label{36}

Soient $p$ un nombre premier impair, $k$ un entier et $n \in \mathbb{N}^{*}$. Soit $r$ la taille de la solution $\overline{k}$-monomiale minimale de $(E_{p^{n}})$. Si $r>\frac{p-1}{2}p^{n-1}$ alors $r \in \{p^{n}, \frac{p+1}{2}p^{n-1}\}$.
\\
\\En particulier, on a :
\begin{itemize}
\item $r \leq p^{n}$;
\item si $r>\frac{p+1}{2}p^{n-1}$ alors $r=p^{n}$;
\item si $r<p^{n}$ alors $r \leq \frac{p+1}{2}p^{n-1}$.
\end{itemize}

\end{lemma}

\begin{proof}

On raisonne par récurrence sur $n$. Supposons que $n=1$. Par le théorème \ref{31}, $r=p$ ou $r$ divise $\frac{p+1}{2}$ ou $\frac{p-1}{2}$. Donc, si $r> \frac{p-1}{2}$ alors $r \in \{p, \frac{p+1}{2}\}$.
\\
\\Supposons qu'il existe un $n$ dans $\mathbb{N}^{*}$ tel que, pour tout entier $k$, la propriété suivante suivante soit vérifiée : si $r$ est la taille de la solution $(k+p^{n}\mathbb{Z})$-monomiale minimale de $(E_{p^{n}})$ et si $r> \frac{p-1}{2}p^{n-1}$ alors $r \in \{p^{n}, \frac{p+1}{2}p^{n-1}\}$.
\\
\\Soit $k$ un entier et $l$ la taille de la solution $(k+p^{n+1}\mathbb{Z})$-monomiale minimale de $(E_{p^{n+1}})$. On suppose que $l> \frac{p-1}{2}p^{n}$ (notons que l'on peut toujours choisir un tel $k$ en prenant par exemple $k=2$). Notons $r$ la taille de la solution $(k+p^{n}\mathbb{Z})$-monomiale minimale de $(E_{p^{n}})$. Par le lemme \ref{34}, $l=r$ ou $l=pr$. Ainsi, $r > \frac{p-1}{2}p^{n-1}$. Par hypothèse de récurrence, $r \in \{p^{n}, \frac{p+1}{2}p^{n-1}\}$. Donc, $l \in \{p^{n}, \frac{p+1}{2}p^{n-1}, p^{n+1}, \frac{p+1}{2}p^{n}\}$. Or, $\frac{p-1}{2}\geq 1$ (puisque $p \geq 3$) et donc $l> p^{n}> \frac{p+1}{2}p^{n-1}$. Par conséquent, $l \in \{p^{n+1}, \frac{p+1}{2}p^{n}\}$.
\\
\\Par récurrence, le résultat est démontré.
\\
\\Soient $p$ un nombre premier impair, $k$ un entier et $n \in \mathbb{N}^{*}$. Soit $r$ la taille de la solution $\overline{k}$-monomiale minimale de $(E_{p^{n}})$. 
\begin{itemize}
\item On a deux cas : soit $r \leq \frac{p-1}{2}p^{n-1} \leq p^{n}$ soit $r > \frac{p-1}{2}p^{n-1}$ et alors par ce qui précède $r \leq p^{n}$.
\item Si $r>\frac{p+1}{2}p^{n-1}$ alors, comme $\frac{p+1}{2}p^{n-1} > \frac{p-1}{2}p^{n-1}$, on a par ce qui précède $r=p^{n}$;
\item Si $r \neq p^{n}$ alors soit $r \leq \frac{p-1}{2}p^{n-1}<\frac{p+1}{2}p^{n-1}$ soit $r > \frac{p-1}{2}p^{n-1}$ et alors par ce qui précède $r=\frac{p+1}{2}p^{n-1}$.
\end{itemize}

\end{proof}

\begin{lemma}
\label{361}

Soient $p=2$, $k$ un entier et $n \geq 2$. Soit $r$ la taille de la solution $\overline{k}$-monomiale minimale de $(E_{p^{n}})$. Si $r>p^{n-1}$ alors $r \in \{p^{n}, \frac{p+1}{2}p^{n-1}\}$. En particulier, on a $r \leq p^{n}$.

\end{lemma}

\begin{proof}

On raisonne par récurrence sur $n$. Si $n=2$, on a $p^{n}=4$. Les solutions monomiales minimales de $(E_{4})$ sont de taille 2, 3 ou 4. Donc, si la taille de la solution $\overline{k}$-monomiale minimale de $(E_{4})$ est strictement supérieure à 2, elle est égale à $3=\frac{p+1}{2}p^{n-1}$ ou $4=p^{n}$.
\\
\\Supposons qu'il existe un entier $n$ supérieur à 2 tel que, pour tout entier $k$, la propriété suivante suivante soit vérifiée : si $r$ est la taille de la solution $(k+p^{n}\mathbb{Z})$-monomiale minimale de $(E_{p^{n}})$ et si $r> p^{n-1}$ alors $r \in \{p^{n}, \frac{p+1}{2}p^{n-1}\}$.
\\
\\Soit $k$ un entier et $l$ la taille de la solution $(k+p^{n+1}\mathbb{Z})$-monomiale minimale de $(E_{p^{n+1}})$. On suppose que $l> p^{n}$ (notons que l'on peut toujours choisir un tel $k$ en prenant par exemple $k=2$). Notons $r$ la taille de la solution $(k+p^{n}\mathbb{Z})$-monomiale minimale de $(E_{p^{n}})$. Par le lemme \ref{34}, $l=r$ ou $l=pr$. Ainsi, $r > p^{n-1}$. Par hypothèse de récurrence, $r \in \{p^{n}, \frac{p+1}{2}p^{n-1}\}$. Donc, $l \in \{p^{n}, \frac{p+1}{2}p^{n-1}, p^{n+1}, \frac{p+1}{2}p^{n}\}$. Or, $l> p^{n}> \frac{p+1}{2}p^{n-1}$. Par conséquent, $l \in \{p^{n+1}, \frac{p+1}{2}p^{n}\}$.
\\
\\Par récurrence, le résultat est démontré.

\end{proof}

\begin{remark}

{\rm Si $N=2$ et si $k$ est impair alors la taille de la solution $\overline{k}$-monomiale minimale de $(E_{2})$ est égale à 3. Ce cas particulier, en apparence insignifiant, aura des conséquences importantes dans la suite.
}

\end{remark}

On va maintenant donner quelques informations sur la taille maximale des solutions monomiales minimales. Jusqu'ici on ne disposait que de la borne générale $3N$ qui correspond à l'ordre maximal des éléments de $SL_{2}(\mathbb{Z}/N\mathbb{Z})$ (voir \cite{CGL} page 216). Dans ce qui suit, on va obtenir des éléments plus précis qui permettront notamment de redémontrer et même d'améliorer cette borne.

\begin{proposition}
\label{37}

Soit $N=p_{1}^{\alpha_{1}}\ldots p_{r}^{\alpha_{r}}$, avec, pour tout $i$, $p_{i}$ premier, $\alpha_{i} \geq 1$ et $p_{i} \neq p_{j}$ si $i \neq j$. Soient $\overline{k} \in \mathbb{Z}/N\mathbb{Z}$ et $l_{i}$ la taille de la solution $(k+p_{i}^{\alpha_{i}}\mathbb{Z})$-monomiale minimale de $(E_{p_{i}^{\alpha_{i}}})$. La taille $l$ de la solution $\overline{k}$-monomiale minimale de \eqref{p} est égale à $ppcm(l_{i},~1 \leq i \leq r)$ ou à $2 \times ppcm(l_{i},~1 \leq i \leq r)$. En particulier, si $N$ n'est pas de la forme $2u$ avec $u$ impair, alors $l \leq 2N$.

\end{proposition}

\begin{proof}

Pour tout $1 \leq i \leq r$, $p_{i}^{\alpha_{i}}$ divise $N$. Donc, par le lemme \ref{33}, $l$ est un multiple de $l_{i}$. Ainsi, $l$ est un multiple de $m={\rm ppcm}(l_{i},~1 \leq i \leq r)$. 
\\
\\Soit $1 \leq i \leq r$. Il existe $a_{i} \in \mathbb{N}$ tel que $m=a_{i}l_{i}$ et il existe $\epsilon_{i} \in \{-1, 1\}$ tel que 
\[M_{l_{i}}(k+p_{i}^{\alpha_{i}}\mathbb{Z},\ldots,k+p_{i}^{\alpha_{i}}\mathbb{Z})=(\epsilon_{i}+p_{i}^{\alpha_{i}}\mathbb{Z})Id.\]
\noindent De plus,
\begin{eqnarray*}
M &=& M_{2m}(k+p_{i}^{\alpha_{i}}\mathbb{Z},\ldots,k+p_{i}^{\alpha_{i}}\mathbb{Z}) \\
  &=& M_{2a_{i}l_{i}}(k+p_{i}^{\alpha_{i}}\mathbb{Z},\ldots,k+p_{i}^{\alpha_{i}}\mathbb{Z}) \\
	&=& M_{l_{i}}(k+p_{i}^{\alpha_{i}}\mathbb{Z},\ldots,k+p_{i}^{\alpha_{i}}\mathbb{Z})^{2a_{i}} \\
	&=& (\epsilon_{i}+p_{i}^{\alpha_{i}}\mathbb{Z})^{2a_{i}}Id \\
	&=& (1+p_{i}^{\alpha_{i}}\mathbb{Z})Id.
\end{eqnarray*}

\noindent Ainsi, par le lemme chinois, $M_{2m}(\overline{k},\ldots,\overline{k})=Id$. On en déduit que $l$ divise $2m$. Or, $l$ est également un multiple de $m$. Donc, $l=m$ ou $l=2m$.
\\
\\De plus, si $m$ n'est pas de la forme $2u$ avec $u$ impair, $m \leq l_{1}\ldots l_{r}$ et, pour tout $1\leq i \leq r$, $l_{i} \leq p_{i}^{\alpha_{i}}$ (lemmes \ref{36} et \ref{361}). Ainsi, $m \leq N$ et $l \leq 2N$.

\end{proof}

\begin{proposition}
\label{38}

Soient $N=2u$ avec $u \geq 3$ impair et $k$ un entier. Soient $l_{1}$ la taille de la solution $(k+2\mathbb{Z})$-monomiale minimale de $(E_{2})$, $h$ la taille de la solution $(k+u)$-monomiale minimale de $(E_{u})$ et $l$ la taille de la solution $\overline{k}$-monomiale minimale de \eqref{p}. 
\[l=ppcm(l_{1},h).\]
\noindent En particulier, on a les inégalités suivantes :
\begin{itemize}
\item si $k$ est pair $l \leq N$;
\item si $k$ est impair $l \leq 3N$.
\end{itemize}

\end{proposition}

\begin{proof}

Il existe $\epsilon \in \{-1,1\}$ tel que $M_{h}(k+u\mathbb{Z},\ldots,k+u\mathbb{Z})=(\epsilon+u\mathbb{Z}) Id$. 2 et $u$ divisent $N$. Donc, par le lemme \ref{33}, $l$ est un multiple de $l_{1}$ et de $h$, c'est-à-dire $l$ est un multiple de $m={\rm ppcm}(l_{1},h)$. Or, 

\[M_{m}(k+u\mathbb{Z},\ldots,k+u\mathbb{Z})=M_{h}(k+u\mathbb{Z},\ldots,k+u\mathbb{Z})^{\frac{m}{h}}=(\epsilon+u\mathbb{Z})^{\frac{m}{h}}Id=(\epsilon^{\frac{m}{h}}+u\mathbb{Z})Id.\]
\noindent et 
\[M_{m}(k+2\mathbb{Z},\ldots,k+2\mathbb{Z})=M_{l_{1}}(k+2\mathbb{Z},k+2\mathbb{Z})^{\frac{m}{l_{1}}}=Id=-Id=(\epsilon^{\frac{m}{h}}+2\mathbb{Z})Id.\]
\noindent Par le lemme chinois, $l=m$. 
\\
\\Si $k$ est impair, $l_{1}=3$. Donc, $l={\rm ppcm}(3,h) \leq 3h \leq 3 \times (2u)=3N$ ($h \leq 2u$ par la proposition \ref{37}).
\\
\\Si $k$ est pair, $l_{1}=2$. Donc, $l={\rm ppcm}(2,h)$. Si $h$ est pair alors $l=h$ et, par la proposition \ref{37}, $l \leq 2u=N$. Si $h$ est impair. Comme $u \neq 1$, on peut écrire $u$ sous la forme $u=p_{2}^{\alpha_{2}}\ldots p_{r}^{\alpha_{r}}$, avec, pour tout $i$, $p_{i}$ premier impair, $\alpha_{i} \geq 1$ et $p_{i} \neq p_{j}$ si $i \neq j$. On note $l_{i}$ la taille de la solution $(k+p_{i}^{\alpha_{i}}\mathbb{Z})$-monomiale minimale de $(E_{p_{i}^{\alpha_{i}}})$.  Par le lemme \ref{36}, $l_{i} \leq p_{i}^{\alpha_{i}}$ pour tout $i$ dans $[\![1;r]\!]$. Par la proposition \ref{37}, $h={\rm ppcm}(l_{i},~2 \leq i \leq r)$ ou $h=2 \times {\rm ppcm}(l_{i},~2 \leq i \leq r)$. Comme $h$ est impair, $h={\rm ppcm}(l_{i},~2 \leq i \leq r)$. Ainsi :
\[l={\rm ppcm}(2,h)=2h=2 \times {\rm ppcm}(l_{i},~2 \leq i \leq r) \leq 2\times (p_{2}^{\alpha_{2}}\ldots p_{r}^{\alpha_{r}})=2u=N.\]

\end{proof}

\begin{remark}
{\rm 
Les inégalités des deux résultats précédents sont optimales. En effet, on peut notamment considérer les deux exemples ci-dessous :
\begin{itemize}
\item si $N=35=5 \times 7$ et $k=23$ alors la taille de la solution $\overline{k}$-monomiale minimale de \eqref{p} est égale à $70=2 \times 35$.
\item si $N=70=2 \times (5 \times 7)$ et $k=23$ alors la taille de la solution $\overline{k}$-monomiale minimale de \eqref{p} est égale à $210=3 \times 70$.
\end{itemize}
}
\end{remark}

\begin{lemma}
\label{39}

Soit $N=p_{1}^{\alpha_{1}} \ldots p_{r}^{\alpha_{r}}$ avec $r \geq 2$, $p_{i}$ des nombres premiers deux à deux distincts et $\alpha_{i}$ des entiers naturels non nuls. Soient $k$ un entier naturel et $l_{i}$ la taille de la solution $(k+p_{i}^{\alpha_{i}}\mathbb{Z})$-monomiale minimale de $(E_{p_{i}^{\alpha_{i}}})$. 
\\
\\i) On suppose que $p_{1}=2$. Si $l_{1}$ est pair différent de 2 et si, pour tout $i$ dans $[\![2;r]\!]$, $l_{i}$ est impair alors la taille de la solution $\overline{k}$-monomiale minimale de \eqref{p} est égale à $l=ppcm(l_{i},~1 \leq i \leq r)$.
\\
\\ii) On suppose que, pour tout $i$ dans $[\![1;r]\!]$, $l_{i}$ est impair. La taille de la solution $\overline{k}$-monomiale minimale de \eqref{p} est égale à $l=ppcm(l_{i},~1 \leq i \leq r)$ si et seulement s'il existe $\epsilon \in \{-1,1\}$ tel que, pour tout $i$ dans $[\![1;r]\!]$, $M_{l_{i}}(k+p_{i}^{\alpha_{i}}\mathbb{Z},\ldots,k+p_{i}^{\alpha_{i}}\mathbb{Z})=(\epsilon+p_{i}^{\alpha_{i}}\mathbb{Z})Id$.

\end{lemma}

\begin{proof}

Soit $l$ la taille de la solution $\overline{k}$-monomiale minimale de \eqref{p}. 
\\
\\i) Par la proposition \ref{37}, $l$ est un multiple de $m={\rm ppcm}(l_{i},~1 \leq i \leq r)$. Comme $m$ est un multiple de $l_{1}$, $m$ est pair. Soit $i$ dans $[\![2;r]\!]$. Comme $l_{i}$ est impair, $\frac{m}{l_{i}}$ est pair. Il existe $\epsilon_{i} \in \{-1, 1\}$ tel que 
\[M_{l_{i}}(k+p_{i}^{\alpha_{i}}\mathbb{Z},\ldots,k+p_{i}^{\alpha_{i}}\mathbb{Z})=(\epsilon_{i}+p_{i}^{\alpha_{i}}\mathbb{Z})Id.\]
\noindent On a
\[M_{m}(k+p_{i}^{\alpha_{i}}\mathbb{Z},\ldots,k+p_{i}^{\alpha_{i}}\mathbb{Z})=M_{l_{i}}(k+p_{i}^{\alpha_{i}}\mathbb{Z},\ldots,k+p_{i}^{\alpha_{i}}\mathbb{Z})^{\frac{m}{l_{i}}}=(\epsilon_{i}+p_{i}^{\alpha_{i}}\mathbb{Z})^{\frac{m}{l_{i}}}Id=(1+p_{i}^{\alpha_{i}}\mathbb{Z})Id.\]

\noindent De plus, en utilisant le lemme \ref{35} ($l_{1} \neq 2$ pair donc $\alpha_{1} \geq 2$ et $k+p_{1}^{\alpha_{1}}\mathbb{Z} \neq 0+p_{1}^{\alpha_{1}}\mathbb{Z}$), on a 
\[M_{m}(k+p_{1}^{\alpha_{1}}\mathbb{Z},\ldots,k+p_{1}^{\alpha_{1}}\mathbb{Z})=M_{l_{1}}(k+p_{1}^{\alpha_{1}}\mathbb{Z},\ldots,k+p_{1}^{\alpha_{1}}\mathbb{Z})^{\frac{m}{l_{1}}}=(1+p_{1}^{\alpha_{1}}\mathbb{Z})^{\frac{m}{l_{1}}}Id=(1+p_{1}^{\alpha_{1}}\mathbb{Z})Id.\]

\noindent Par le lemme chinois, $M_{m}(\overline{k},\ldots,\overline{k})=Id$. En particulier, $l=m$.
\\
\\ii) S'il existe $\epsilon \in \{-1,1\}$ tel que, pour tout $i$ dans $[\![1;r]\!]$, $M_{l_{i}}(k+p_{i}^{\alpha_{i}}\mathbb{Z},\ldots,k+p_{i}^{\alpha_{i}}\mathbb{Z})=(\epsilon+p_{i}^{\alpha_{i}}\mathbb{Z})Id$. Soit $m={\rm ppcm}(l_{i},~1 \leq i \leq r)$. Comme tous les $l_{i}$ sont impairs, $m$ est impair.
\\
\\Par la proposition \ref{37}, $l$ est un multiple de $m$. De plus, pour tout $i$ dans $[\![1;r]\!]$, $\frac{m}{l_{i}}$ est impair. Donc, pour tout $i$ dans $[\![1;r]\!]$, $(\epsilon+p_{i}^{\alpha_{i}}\mathbb{Z})^{\frac{m}{l_{i}}}=(\epsilon+p_{i}^{\alpha_{i}}\mathbb{Z})$, et on a
\[M_{m}(k+p_{i}^{\alpha_{i}}\mathbb{Z},\ldots,k+p_{i}^{\alpha_{i}}\mathbb{Z})=M_{l_{i}}(k+p_{i}^{\alpha_{i}}\mathbb{Z},\ldots,k+p_{i}^{\alpha_{i}}\mathbb{Z})^{\frac{m}{l_{i}}}=(\epsilon+p_{i}^{\alpha_{i}}\mathbb{Z})^{\frac{m}{l_{i}}}Id=(\epsilon+p_{i}^{\alpha_{i}}\mathbb{Z})Id.\]

\noindent Par le lemme chinois, $M_{m}(\overline{k},\ldots,\overline{k})=\overline{\epsilon}Id$. En particulier, $l=m$.
\\
\\S'il existe $j$ et $h$ dans $[\![1;r]\!]$ tels que $M_{l_{j}}(k+p_{j}^{\alpha_{j}}\mathbb{Z},\ldots,k+p_{j}^{\alpha_{j}}\mathbb{Z})=(1+p_{j}^{\alpha_{j}}\mathbb{Z})Id \neq (-1+p_{j}^{\alpha_{j}}\mathbb{Z})Id$ et $M_{l_{h}}(k+p_{h}^{\alpha_{h}}\mathbb{Z},\ldots,k+p_{h}^{\alpha_{h}}\mathbb{Z})=(1+p_{h}^{\alpha_{h}}\mathbb{Z})Id \neq (-1+p_{h}^{\alpha_{h}}\mathbb{Z})Id$. Si $l=m={\rm ppcm}(l_{i},~1 \leq i \leq r)$. Il existe $\alpha \in \{-1,1\}$ tel que $M_{m}(\overline{k},\ldots,\overline{k})=\overline{\alpha}Id$. En particulier, puisque $\frac{m}{l_{j}}$ et $\frac{m}{l_{h}}$ sont impairs, on a :
\begin{eqnarray*}
(\alpha+p_{j}^{\alpha_{j}}\mathbb{Z})Id &=& M_{m}(k+p_{j}^{\alpha_{j}}\mathbb{Z},\ldots,k+p_{j}^{\alpha_{j}}\mathbb{Z}) \\
                                        &=& M_{l_{j}}(k+p_{j}^{\alpha_{j}}\mathbb{Z},\ldots,k+p_{j}^{\alpha_{j}}\mathbb{Z})^{\frac{m}{l_{j}}} \\
																				&=& (1+p_{j}^{\alpha_{j}}\mathbb{Z})^{\frac{m}{l_{j}}}Id \\
																				&=& (1+p_{j}^{\alpha_{j}}\mathbb{Z})Id,
\end{eqnarray*}
\noindent et
\begin{eqnarray*}
(\alpha+p_{h}^{\alpha_{h}}\mathbb{Z})Id &=& M_{m}(k+p_{h}^{\alpha_{h}}\mathbb{Z},\ldots,k+p_{h}^{\alpha_{h}}\mathbb{Z}) \\
                                        &=& M_{l_{h}}(k+p_{h}^{\alpha_{h}}\mathbb{Z},\ldots,k+p_{h}^{\alpha_{h}}\mathbb{Z})^{\frac{m}{l_{h}}} \\
																				&=& (-1+p_{h}^{\alpha_{h}}\mathbb{Z})^{\frac{m}{l_{h}}}Id \\
																				&=& (-1+p_{h}^{\alpha_{h}}\mathbb{Z})Id.
\end{eqnarray*}		

\noindent Ainsi, si $\alpha=1$, on a $1+p_{h}^{\alpha_{h}}\mathbb{Z}=-1+p_{h}^{\alpha_{h}}\mathbb{Z}$ et si $\alpha=-1$, on a $-1+p_{j}^{\alpha_{j}}\mathbb{Z}=1+p_{j}^{\alpha_{j}}\mathbb{Z}$. Dans les deux cas, on arrive à une absurdité. Donc, $l \neq m$.													

\end{proof}

\begin{remark}
{\rm Si $l_{1}=2$ dans le cas i) alors le résultat n'est plus vrai en général. Par exemple, prenons $N=12$ et $k=4$. La solution $(k+4\mathbb{Z})$-monomiale minimale de $(E_{4})$ est de taille 2, la solution $(k+3\mathbb{Z})$-monomiale minimale de $(E_{3})$ est de taille 3 et la solution $\overline{k}$-monomiale miniale de \eqref{p} est de taille $12=2\times {\rm ppcm}(2,3)$.
}
\end{remark} 

\begin{lemma}[\cite{M2} lemme 3.2]
\label{plus}

Soient $N$ un entier supérieur à 2 et $\overline{k}$ un élément de $\mathbb{Z}/N\mathbb{Z}$. Soit $n$ la taille de la solution $\overline{k}$-monomiale minimale de $(E_{N})$.
\\
\\A) Soit $m$ un entier naturel.
\\i) Si $(\overline{a},\overline{k},\ldots,\overline{k},\overline{b}) \in (\mathbb{Z}/N\mathbb{Z})^{nm}$ est une solution de $(E_{N})$ alors $\overline{a}=\overline{b}=\overline{k}$.
\\ii) Il n'y a pas de solution de $(E_{N})$ de la forme $(\overline{a},\overline{k},\ldots,\overline{k},\overline{b})$ de taille $nm+1$.
\\iii) Si $(\overline{a},\overline{k},\ldots,\overline{k},\overline{b}) \in (\mathbb{Z}/N\mathbb{Z})^{nm+2}$ est une solution de $(E_{N})$ alors $\overline{a}=\overline{b}=\overline{0}$.
\\
\\B)  On suppose que la solution $\overline{k}$-monomiale minimale de $(E_{N})$ est irréductible. $(\overline{a},\overline{k},\ldots,\overline{k},\overline{b}) \in (\mathbb{Z}/N\mathbb{Z})^{l}$ est une solution de $(E_{N})$ si et seulement si une deux conditions suivantes est vérifiée :
\begin{itemize}
\item $l \equiv 0 [n]$ et $\overline{a}=\overline{b}=\overline{k}$;
\item $l \equiv 2 [n]$ et $\overline{a}=\overline{b}=\overline{0}$.

\end{itemize}

\end{lemma}

On conclut cette séquence de résultats intermédiaires par le rappel des deux résultats importants ci-dessous.

\begin{theorem}[\cite{M3} Théorème 2.6]
\label{311b}

Soient $p$ un nombre premier, $n$ un entier naturel non nul et $N=p^{n}$. Soit $k \in \mathbb{Z}$.
\\
\\i) Si $p \neq 2$. La solution $\overline{k}$-monomiale minimale de $(E_{N})$ est irréductible si et seulement si $p$ ne divise pas $k$.
\\
\\ii) Si $p=2$. La solution $\overline{k}$-monomiale minimale de $(E_{N})$ est irréductible si et seulement si  une des conditions suivantes est vérifiée :
\begin{itemize}
\item $k$ est impair; 
\item $\overline{k}=\overline{2^{n-1}}$;
\item $n \geq 2$ et il existe un entier impair $a$ tel que $k=2a$.
\end{itemize}

\end{theorem}

\begin{proposition}[\cite{M3} proposition 4.1]
\label{311}

Soit $N=p_{1}^{\alpha_{1}} \ldots p_{r}^{\alpha_{r}}$ avec $r \geq 1$, $p_{i}$ des nombres premiers deux à deux distincts et $\alpha_{i}$ des entiers naturels non nuls. Soit $k=ap_{1}^{\beta_{1}} \ldots p_{r}^{\beta_{r}}$ avec, pour tout $i$, $1 \leq \beta_{i} \leq \alpha_{i}$ et $a$ un entier qui n'est divisible par aucun des $p_{i}$. La solution $\overline{k}$-monomiale minimale de \eqref{p} est de taille $2p_{1}^{\alpha_{1}-\beta_{1}} \ldots p_{r}^{\alpha_{r}-\beta_{r}}$.

\end{proposition}

Notons que cette sous-partie regroupe tous les résultats connus jusqu'à maintenant sur la taille des solutions monomiales minimales.

\subsection{Démonstration du théorème principal}
\label{preuve}

Maintenant que tous ces résultats sont fixés, on peut revenir à notre objectif principal. Avant de commencer la preuve, on donne un dernier petit lemme.

\begin{lemma}
\label{312}

$f:x \longmapsto \frac{x+1}{x}$ est strictement décroissante sur $\mathbb{R}^{+*}$.

\end{lemma}

\begin{proof}

Pour tout $x \in \mathbb{R}^{+*}$, $f'(x)=\frac{x-(x+1)}{x^{2}}=\frac{-1}{x^{2}}<0$. Donc, $f$ est strictement décroissante sur $\mathbb{R}^{+*}$.

\end{proof}

\begin{theorem}
\label{prin}

Soit $N$ un entier différent de 2 qui n'est pas de la forme $3m$ avec $m$ premier avec 6. Les solutions monomiales minimales irréductibles de \eqref{p} sont de taille inférieure ou égale à $N$.

\end{theorem}

Pour démontrer le théorème \ref{prin}, on va considérer deux cas qui seront chacun détaillé dans une proposition dédiée. On commence par le résultat suivant :

\begin{proposition}
\label{313}

Soit $N$ un entier naturel qui n'est pas de la forme $2u$ avec $u$ impair ou $3v$ avec $v$ impair non divisible par 3, c'est-à-dire un entier $N$ qui vérifie une des condition suivantes :
\begin{itemize}
\item $N$ impair non divisible par 3;
\item $N$ impair divisible par 9;
\item $N$ pair divisible par 4.
\end{itemize}

\noindent Les solutions monomiales minimales irréductibles de \eqref{p} sont de taille inférieure ou égale à $N$.

\end{proposition}

\begin{proof}

On raisonne par contraposée. Soit $N=p_{1}^{\alpha_{1}}\ldots p_{r}^{\alpha_{r}}$, avec pour tout $i$ dans $[\![1;r]\!]$ $p_{i} \in \mathbb{P}$, $\alpha_{i} \geq 1$ et $p_{i} \neq p_{j}$ si $i \neq j$, un entier naturel qui n'est pas de la forme $2u$ avec $u$ impair ou $3v$ avec $v$ impair non divisible par 3. On suppose qu'il existe un entier $k$ tel que la taille de la solution $\overline{k}$-monomiale minimale est de taille $l$ strictement supérieure à $N$ (notons que l'existence d'un tel $k$ n'est pas garantie, par exemple pour $N=12$). Par les lemmes \ref{36} et \ref{361}, $r \geq 2$. On note, pour tout $i$ dans $[\![1;r]\!]$, $l_{i}$ la taille de la solution $(k+p_{i}^{\alpha_{i}}\mathbb{Z})$-monomiale minimale de $(E_{p_{i}^{\alpha_{i}}})$. On scinde la preuve en deux étapes.
\\
\\\uwave{$1^{{\rm \grave{e}re}}$ étape :} On étudie les conséquences de l'hypothèse $l>N$ sur les $l_{i}$.
\\
\\Par la proposition \ref{37}, $l={\rm ppcm}(l_{i},~1 \leq i \leq r)$ ou $l=2 \times {\rm ppcm}(l_{i},~1 \leq i \leq r)$. Si $l={\rm ppcm}(l_{i},~1 \leq i \leq r)$ alors, par les lemmes \ref{36} et \ref{361} :
\[l \leq \prod_{i=1}^{r} l_{i} \leq \prod_{i=1}^{r} p_{i}^{\alpha_{i}}=N.\]

\noindent Ainsi, $l=2 \times {\rm ppcm}(l_{i},~1 \leq i \leq r)$. 
\\
\\Supposons par l'absurde qu'il existe $j$ dans $[\![1;r]\!]$ tel que $l_{j} \leq p_{j}^{\alpha_{j}-1}$. Comme $2 p_{j}^{\alpha_{j}-1}\leq p_{j}^{\alpha_{j}}$, on a :
\[l=2 \times {\rm ppcm}(l_{i},~1 \leq i \leq r) \leq 2 \prod_{i=1}^{r} l_{i} \leq 2p_{j}^{\alpha_{j}-1} \prod_{i\neq j} p_{i}^{\alpha_{i}} \leq N.\]
\noindent Donc, pour tout $i \in [\![1;r]\!]$, $l_{i} > p_{i}^{\alpha_{i}-1}$.
\\
\\Supposons par l'absurde qu'il existe $j$ dans $[\![1;r]\!]$ tel que $p_{j}$ est impair et $l_{j} \leq \frac{p_{j}-1}{2}p_{j}^{\alpha_{j}-1}$. Comme $(p_{j}-1)p_{j}^{\alpha_{j}-1}<p_{j}^{\alpha_{j}}$, on a :
\[l=2 \times {\rm ppcm}(l_{i},~1 \leq i \leq r) \leq 2 \prod_{i=1}^{r} l_{i} \leq (p_{j}-1)p_{j}^{\alpha_{j}-1} \prod_{i\neq j} p_{i}^{\alpha_{i}}<\prod_{i=1}^{r} p_{i}^{\alpha_{i}}=N.\]
\noindent Donc, pour tout $i \in [\![1;r]\!]$ tel que $p_{i}$ est impair on a $l_{i} > \frac{p_{i}-1}{2}p_{i}^{\alpha_{i}-1}$.
\\
\\Ainsi, par les lemmes \ref{36} et \ref{361}, on a, pour tout $i$ dans $[\![1;r]\!]$, $l_{i}=\frac{p_{i}+1}{2}p_{i}^{\alpha_{i}-1}$ ou $l_{i}=p_{i}^{\alpha_{i}}$. 
\\
\\Supposons par l'absurde qu'il existe $j$ et $h$ dans $[\![1;r]\!]$, $j<h$, tel que $l_{j}=\frac{p_{j}+1}{2}p_{j}^{\alpha_{j}-1}$ et $l_{h}=\frac{p_{h}+1}{2}p_{h}^{\alpha_{h}-1}$. On a :
\begin{eqnarray*}
l &=& 2 \times {\rm ppcm}(l_{i},~1 \leq i \leq r) \\
  &\leq& 2\prod_{i=1}^{r} l_{i} \\
	&\leq& 2\left(\frac{p_{j}+1}{2}p_{j}^{\alpha_{j}-1}\right)\left(\frac{p_{h}+1}{2}p_{h}^{\alpha_{h}-1}\right)\prod_{i \neq j,h} p_{i}^{\alpha_{i}} \\
	&=& \frac{1}{2} \times \frac{p_{j}+1}{p_{j}} \times \frac{p_{h}+1}{p_{h}}\prod_{i=1}^{r} p_{i}^{\alpha_{i}} \\
	&=& \frac{1}{2} \times \frac{p_{j}+1}{p_{j}} \times \frac{p_{h}+1}{p_{h}} \times N \\
	&\leq& \frac{1}{2} \times \frac{3}{2} \times \frac{4}{3} \times N~({\rm lemme~\ref{312}}~{\rm avec}~p_{j}\geq 2~{\rm et}~p_{h}\geq 3) \\
	&=& N.
\end{eqnarray*}

\noindent Ainsi, on a deux cas :
\begin{itemize}
\item pour tout $i$ dans $[\![1;r]\!]$ $l_{i}=p_{i}^{\alpha_{i}}$;
\item il existe $j_{0}$ dans $[\![1;r]\!]$ tel que $l_{j_{0}}=\frac{p_{j_{0}}+1}{2}p_{j_{0}}^{\alpha_{j_{0}}-1}$ et, pour tout $i$ dans $[\![1;r]\!]$ différent de $j_{0}$, $l_{i}=p_{i}^{\alpha_{i}}$.
\\
\end{itemize}

\noindent Soient $i$ et $j$ dans $[\![1;r]\!]$. Supposons que $l_{i}$ et $l_{j}$ ne sont pas premiers entre eux. Il existe un nombre premier $p$ qui divise $l_{i}$ et $l_{j}$ et donc : 
\[l=2 \times {\rm ppcm}(l_{k},~1 \leq k \leq r) \leq \frac{2}{p} \prod_{k=1}^{r} l_{k} \leq \prod_{k=1}^{r} l_{k} \leq \prod_{k=1}^{r} p_{k}^{\alpha_{k}}=N.\]

\noindent Donc, les $l_{i}$ sont deux à deux premiers entre eux.
\\
\\ \uwave{$2^{{\rm \grave{e}me}}$ étape :} On considère deux cas.
\\
\\A) Si pour tout $i$ dans $[\![1;r]\!]$ $l_{i}$ est impair.
\\
\\Pour tout $i \neq j$ , $l_{i}$ et $l_{j}$ sont premiers entre eux et $l=2\prod_{i=1}^{r} l_{i}$. Notons :
\begin{itemize}
\item $I:=\{i \in [\![1;r]\!], M_{l_{i}}(k+p_{i}^{\alpha_{i}}\mathbb{Z},\ldots,k+p_{i}^{\alpha_{i}}\mathbb{Z})=(1+p_{i}^{\alpha_{i}}\mathbb{Z})Id\}$;
\item $J:=\{j \in [\![1;r]\!], M_{l_{j}}(k+p_{j}^{\alpha_{j}}\mathbb{Z},\ldots,k+p_{j}^{\alpha_{j}}\mathbb{Z})=(-1+p_{j}^{\alpha_{j}}\mathbb{Z})Id\}$.
\end{itemize}
\noindent Par le lemme \ref{39} ii), $I$ et $J$ sont non vides. De plus, $I \cup J=[\![1;r]\!]$ et $I \cap J=\emptyset$ (puisque, pour tout $i$, $p_{i}^{\alpha_{i}} \neq 2$).
\\
\\Comme $\prod_{i \in I} l_{i}$ est inversible modulo $\prod_{j \in J} l_{j}$, il existe un entier $m$ compris entre $0$ et $\left(\prod_{j \in J} l_{j}\right)-1$ vérifiant 
\[\left(\prod_{i \in I} l_{i}\right)m \equiv 2 \left[\prod_{j \in J} l_{j}\right].\]

\noindent Si $m$ est impair, on pose $s=\left(\prod_{i \in I} l_{i}\right)m$. Si $m$ est pair, on pose $s=\left(\prod_{i \in I} l_{i}\right)m+\left(\prod_{i=1}^{r} l_{i}\right)$. Dans les deux cas, $s$ est impair. Puisque $l_{i}$ est impair, $l_{i}>2$ pour tout $i$. En particulier, on a $m \neq 0$. Donc, $s \geq 3$, puisque $I$ est non vide et $l_{i} \geq 3$ pour tout $i$. De plus, 
\begin{eqnarray*}
s &\leq& \left(\prod_{i \in I} l_{i}\right)\left(\left(\prod_{j \in J} l_{j}\right)-1\right)+\left(\prod_{i=1}^{r} l_{i}\right) \\
  &=&\left(\prod_{i \in I} l_{i}\right)\left(\prod_{j \in J} l_{j}\right)+\left(\prod_{i=1}^{r} l_{i}\right)-\left(\prod_{i \in I} l_{i}\right) \\
	&=&l-\left(\prod_{i \in I} l_{i}\right) \\
	&<& l.
\end{eqnarray*}

\noindent Pour tout $i \in I$, on a, puisque $l_{i}$ divise $s$ : \[M_{s}(k+p_{i}^{\alpha_{i}}\mathbb{Z},\ldots,k+p_{i}^{\alpha_{i}}\mathbb{Z})=(M_{l_{i}}(k+p_{i}^{\alpha_{i}}\mathbb{Z},\ldots,k+p_{i}^{\alpha_{i}}\mathbb{Z}))^{\frac{s}{l_{i}}}=(1+p_{i}^{\alpha_{i}}\mathbb{Z})^{\frac{s}{l_{i}}} Id=(1+p_{i}^{\alpha_{i}}\mathbb{Z})Id.\]

\noindent Pour tout $j \in J$, on a, puisque $l_{j}$ divise $s-2$ : 
\begin{eqnarray*}
M &=& M_{s}(0+p_{j}^{\alpha_{j}}\mathbb{Z},k+p_{j}^{\alpha_{j}}\mathbb{Z},\ldots,k+p_{j}^{\alpha_{j}}\mathbb{Z},0+p_{j}^{\alpha_{j}}\mathbb{Z}) \\
  &=& M_{1}(0+p_{j}^{\alpha_{j}}\mathbb{Z})M_{l_{j}}(k+p_{j}^{\alpha_{j}}\mathbb{Z},\ldots,k+p_{j}^{\alpha_{j}}\mathbb{Z})^{\frac{s-2}{l_{j}}}M_{1}(0+p_{j}^{\alpha_{j}}\mathbb{Z}) \\
	&=& M_{1}(0+p_{j}^{\alpha_{j}}\mathbb{Z})((-1+p_{j}^{\alpha_{j}})Id)^{\frac{s-2}{l_{j}}}M_{1}(0+p_{j}^{\alpha_{j}}\mathbb{Z}) \\
	&=&-M_{2}(0+p_{j}^{\alpha_{j}}\mathbb{Z},0+p_{j}^{\alpha_{j}}\mathbb{Z})~\left({\rm car}~\frac{s-2}{l_{j}}~{\rm impair}\right) \\
	&=& (1+p_{j}^{\alpha_{j}}\mathbb{Z})Id.
\end{eqnarray*}

\noindent Notons $a$ l'unique entier compris entre 0 et $N-1$ tel que $a \equiv k [p_{i}^{\alpha_{i}}]$ pour tout $i \in I$ et $a \equiv 0 [p_{j}^{\alpha_{j}}]$ pour tout $j \in J$. Par le lemme chinois, on a $M_{s}(\overline{a},\overline{k},\ldots,\overline{k},\overline{a})=Id$. Comme $3 \leq s \leq l-1$, on peut utiliser cette solution pour réduire la solution $\overline{k}$-monomiale minimale de \eqref{p} qui est donc réductible.
\\
\\B) S'il existe un $j_{0}$ dans $[\![1;r]\!]$ tel que $l_{j_{0}}$ est pair. 
\\
\\Comme les $l_{i}$ sont deux à deux premiers entre eux, $l_{i}$ est impair pour tout $i$ dans $[\![1;r]\!]$ différent de $j_{0}$. Supposons que $p_{j_{0}}=2$. On a $\alpha_{j_{0}} \geq 2$ et $l_{j_{0}}=2^{\alpha_{j_{0}}} \geq 4$ ou $l_{j_{0}}=3\times 2^{\alpha_{j_{0}}-1} \geq 6$ (par la première étape). En particulier, $l_{j_{0}} \neq 2$. Ainsi, par le lemme \ref{39} i), on a $l={\rm ppcm}(l_{j},~1 \leq j \leq r)$. Donc, $p_{j_{0}} \neq 2$. De plus, par la première étape, on a nécessairement $l_{j_{0}}=\frac{p_{j_{0}}+1}{2}p_{j_{0}}^{\alpha_{j_{0}}-1}$ (puisque $p_{j_{0}}$ est impair). 
\\
\\Montrons que $l_{j_{0}} \neq 2$. Pour cela, on suppose par l'absurde que $l_{j_{0}}=2$. Comme $\frac{p_{j_{0}}+1}{2} \geq 2$, $l_{j_{0}} \geq 2 \times p_{j_{0}}^{\alpha_{j_{0}}-1}$. Donc, si $\alpha_{j_{0}} \geq 2$, $l_{j_{0}}\geq 6$. On a donc $\alpha_{j_{0}}=1$ et $2=l_{j_{0}}=\frac{p_{j_{0}}+1}{2}$. Ainsi, $p_{j_{0}}=3$ et $\alpha_{j_{0}}=1$. L'hypothèse de départ sur $N$ implique donc que 4 divise $N$. Donc, il existe $h$ dans $[\![1;r]\!]$ tel que $p_{h}=2$. Par la première étape, on a $l_{h}=2^{\alpha_{h}} \neq 1$. Par conséquent, $l_{h}$ et $l_{j_{0}}$ ne sont pas premiers entre eux, ce qui est absurde. Ainsi, $l_{j_{0}} \neq 2$.
\\
\\Par le lemme \ref{34} ii), $M_{l_{j_{0}}}(k+p_{j_{0}}^{\alpha_{j_{0}}} \mathbb{Z},\ldots,k+p_{j_{0}}^{\alpha_{j_{0}}} \mathbb{Z})=(-1+p_{j_{0}}^{\alpha_{j_{0}}} \mathbb{Z})Id$. Pour tout $i \neq j_{0}$, il existe $\epsilon_{i} \in \{-1,1\}$ tel que $M_{l_{i}}(k+p_{i}^{\alpha_{i}} \mathbb{Z},\ldots,k+p_{i}^{\alpha_{i}} \mathbb{Z})=(\epsilon_{i}+p_{i}^{\alpha_{i}} \mathbb{Z})Id$ (notons que contrairement au cas A il est tout fait possible que tous les $\epsilon_{i}$ soient égaux à -1).
\\
\\$l_{j_{0}}$ est inversible modulo $\prod_{i \neq j_{0}} l_{i}$ Donc, il existe un entier $m$ compris entre 0 et $\left(\prod_{i \neq j_{0}} l_{i}\right)-1$ tel que 
\[l_{j_{0}}m \equiv 2 \left[\prod_{i \neq j_{0}} l_{i}\right].\]

\noindent Si $m$ est impair, on pose $s=l_{j_{0}}m$. Si $m$ est pair, on pose $s=l_{j_{0}}\left(m+\prod_{i\neq j_{0}} l_{i}\right)$. Puisque $\prod_{i \neq j_{0}} l_{i}>2$, on a $m \neq 0$. Donc, $s \geq 3$, puisque $l_{j_{0}} \geq 3$. De plus, 
\[s\leq l_{j_{0}}\left(\left(\prod_{i \neq j_{0}} l_{i}\right)-1+\prod_{i\neq j_{0}} l_{i}\right)=2\left(\prod_{i=1}^{r} l_{i}\right)-l_{j_{0}}<l.\]

\noindent Comme $\frac{s}{l_{j_{0}}}$ est impair, on a :
\[M_{s}(k+p_{j_{0}}^{\alpha_{j_{0}}}\mathbb{Z},\ldots,k+p_{j_{0}}^{\alpha_{j_{0}}}\mathbb{Z})=(M_{l_{j_{0}}}(k+p_{j_{0}}^{\alpha_{j_{0}}}\mathbb{Z},\ldots,k+p_{j_{0}}^{\alpha_{j_{0}}}\mathbb{Z}))^{\frac{s}{l_{j_{0}}}}=(-Id)^{\frac{s}{l_{j_{0}}}}=(-1+p_{j_{0}}^{\alpha_{j_{0}}}\mathbb{Z})Id.\]

\noindent De plus, pour tout $i \neq j_{0}$, $\frac{s-2}{l_{i}}$ est pair et donc :
\begin{eqnarray*}
M &=& M_{s}(0+p_{i}^{\alpha_{i}}\mathbb{Z},k+p_{i}^{\alpha_{i}}\mathbb{Z},\ldots,k+p_{i}^{\alpha_{i}}\mathbb{Z},0+p_{i}^{\alpha_{i}}\mathbb{Z}) \\
  &=& M_{1}(0+p_{i}^{\alpha_{i}}\mathbb{Z})M_{l_{i}}(k+p_{i}^{\alpha_{i}}\mathbb{Z},\ldots,k+p_{i}^{\alpha_{i}}\mathbb{Z})^{\frac{s-2}{l_{i}}}M_{1}(0+p_{i}^{\alpha_{i}}\mathbb{Z}) \\
  &=& M_{1}(0+p_{i}^{\alpha_{i}}\mathbb{Z})((\epsilon_{i}+p_{i}^{\alpha_{i}} \mathbb{Z})Id)^{\frac{s-2}{l_{i}}}M_{1}(0+p_{i}^{\alpha_{i}}\mathbb{Z}) \\
	&=& M_{2}(0+p_{i}^{\alpha_{i}}\mathbb{Z},0+p_{i}^{\alpha_{i}}\mathbb{Z}) \\
	&=& (-1+p_{i}^{\alpha_{i}}\mathbb{Z})Id.
\end{eqnarray*}

\noindent Notons $a$ l'unique entier compris entre 0 et $N-1$ tel que $a \equiv 0 [p_{i}^{\alpha_{i}}]$ pour tout $i \neq j_{0}$ et $a \equiv k [p_{j_{0}}^{\alpha_{j_{0}}}]$. Par le lemme chinois, on a $M_{s}(\overline{a},\overline{k},\ldots,\overline{k},\overline{a})=-Id$. Comme $3 \leq s \leq l-1$, on peut utiliser cette solution pour réduire la solution $\overline{k}$-monomiale minimale de \eqref{p} qui est donc réductible.

\end{proof}

\begin{remarks}
{\rm i) Dans le cas B de la preuve ci-dessus, il est possible, comme on l'a indiqué, qu'on ait, pour tout $i$ dans $[\![1;r]\!]$, $M_{l_{i}}(k+p_{i}^{\alpha_{i}}\mathbb{Z},\ldots,k+p_{i}^{\alpha_{i}}\mathbb{Z})=(-1+p_{i}^{\alpha_{i}}\mathbb{Z}) Id$. Par exemple, si $N=35=5 \times 7$ et $k=3$. La solution $\overline{k}$-monomiale minimale de \eqref{p} est de taille 40, la solution $(k+7\mathbb{Z})$-monomiale minimale de $(E_{7})$ est de taille 4 et la solution $(k+5\mathbb{Z})$-monomiale minimale de $(E_{5})$ est de taille 5. On est donc bien dans le cas B de la preuve et on a : 
\[M_{4}(k+7\mathbb{Z},\ldots,k+7\mathbb{Z})=(-1+7\mathbb{Z})Id,~~~~~~~~M_{5}(k+5\mathbb{Z},\ldots,k+5\mathbb{Z})=(-1+5\mathbb{Z})Id.\]

\noindent ii) Dans la preuve précédente, la solution $(k+p_{i}^{\alpha_{i}}\mathbb{Z})$-monomiale minimale de $(E_{p_{i}}^{\alpha_{i}})$ peut être irréductible pour tout $i$ dans $[\![1;r]\!]$. Par exemple, si $N=35=5 \times 7$ et $k=23$. Par la proposition \ref{313}, la solution $\overline{k}$-monomiale minimale de \eqref{p} est réductible. En revanche, comme $k+5\mathbb{Z}=-2+5\mathbb{Z}$, la solution $(k+5\mathbb{Z})$-monomiale minimale de $(E_{5})$ est irréductible (Théorème \ref{30}). De même, la solution $(k+7\mathbb{Z})$-monomiale minimale de $(E_{7})$ est irréductible puisque $k+7\mathbb{Z}=2+7\mathbb{Z}$ (Théorème \ref{30}).
\\
\\iii) Si $N$ n'est pas de la forme $2u$ avec $u$ impair ou $3v$ avec $v$ impair non divisible par 3 alors les solutions de \eqref{p} de taille strictement supérieure à $N$ sont de taille paire.
}
\end{remarks}

Avant de poursuivre la démonstration du théorème \ref{prin}, on va utiliser quelques éléments présentés dans la preuve de la proposition précédente pour obtenir un résultat qui va affiner la proposition \ref{37}.

\begin{proposition}

Si $N$ est un entier naturel divisible par 8 alors les solutions monomiales minimales de \eqref{p} sont de taille inférieure ou égale à $N$.

\end{proposition}

\begin{proof}

Soit $N=p_{1}^{\alpha_{1}}\ldots p_{r}^{\alpha_{r}}$ avec $p_{1}=2$ et $\alpha_{1} \geq 3$. Pour tout $i$ dans $[\![1;r]\!]$, on note $l_{i}$ la taille de la solution $(k+p_{i}^{\alpha_{i}}\mathbb{Z})$-monomiale minimale de $(E_{p_{i}^{\alpha_{i}}})$ et $l$ la taille de la solution $\overline{k}$-monomiale minimale de \eqref{p}. 
\\
\\On suppose par l'absurde qu'il existe un entier $k$ tel que la taille de la solution $\overline{k}$-monomiale minimale de \eqref{p} est strictement supérieure à $N$.
\\
\\Par la première étape de la démonstration de la proposition précédente, $l=2\times {\rm ppcm}(l_{i},~1\leq i \leq r)$ et on a deux cas possibles : 
\begin{itemize}
\item pour tout $i$ dans $[\![1;r]\!]$ $l_{i}=p_{i}^{\alpha_{i}}$;
\item il existe $j_{0}$ dans $[\![1;r]\!]$ tel que $l_{j_{0}}=\frac{p_{j_{0}}+1}{2}p_{j_{0}}^{\alpha_{j_{0}}-1}$ et, pour tout $i$ dans $[\![1;r]\!]$ différent de $j_{0}$, $l_{i}=p_{i}^{\alpha_{i}}$.
\\
\end{itemize}

\noindent Dans les deux cas, on a également, pour tout $i\neq j$, $l_{i}$ et $l_{j}$ premiers entre eux.
\\
\\Si pour tout $i$ dans $[\![1;r]\!]$ $l_{i}=p_{i}^{\alpha_{i}}$ alors $l_{1}$ est pair et différent de 2 et pour tout $i \geq 2$ $l_{i}$ est impair. Par le lemme \ref{39} i), $l={\rm ppcm}(l_{i},~1\leq i \leq r)$, ce qui absurde. 
\\
\\Donc, il existe $j_{0}$ dans $[\![1;r]\!]$ tel que $l_{j_{0}}=\frac{p_{j_{0}}+1}{2}p_{j_{0}}^{\alpha_{j_{0}}-1}$ et, pour tout $i$ dans $[\![1;r]\!]$ différent de $j_{0}$, $l_{i}=p_{i}^{\alpha_{i}}$. Si $j_{0}=1$ alors, comme $\alpha_{1} \geq 3$, $l_{1}$ est pair  et différent de 2 et pour tout $i \geq 2$ $l_{i}$ est impair. Par le lemme \ref{39} i), $l={\rm ppcm}(l_{i},~1\leq i \leq r)$, ce qui absurde. Ainsi, $j_{0} \geq 2$. Donc, $l_{j_{0}}$ est premier avec $l_{1}$ et en particulier $l_{j_{0}}$ est impair. Ainsi, $l_{1}$ est pair et différent de 2 et pour tout $i \geq 2$ $l_{i}$ est impair. Par le lemme \ref{39} i), $l={\rm ppcm}(l_{i},~1\leq i \leq r)$, ce qui absurde. 

\end{proof}

\noindent On considère maintenant le cas manquant :

\begin{proposition}
\label{314}

Soit $N=2m$ avec $m$ impair différent de 1. Les solutions monomiales minimales irréductibles de \eqref{p} sont de taille inférieure ou égale à $N$.

\end{proposition}

\begin{proof}

On raisonne par contraposée. Soit $N=2\times p_{2}^{\alpha_{2}}\ldots p_{r}^{\alpha_{r}}$, avec $r\geq 2$ et, pour tout $i$ dans $[\![2;r]\!]$, $p_{i} \in \mathbb{P}$ impair, $\alpha_{i} \geq 1$ et $p_{i} \neq p_{j}$ si $i \neq j$. On suppose qu'il existe un entier $k$ tel que la taille de la solution $\overline{k}$-monomiale minimale est de taille $l$ strictement supérieure à $N$ (notons que l'existence d'un tel $k$ n'est pas garantie, par exemple pour $N=6$). On note $l_{1}$ la taille de la solution $(k+2\mathbb{Z})$-monomiale minimale de $(E_{2})$, et, pour tout $i$ dans $[\![2;r]\!]$, $l_{i}$ la taille de la solution $(k+p_{i}^{\alpha_{i}}\mathbb{Z})$-monomiale minimale de $(E_{p_{i}^{\alpha_{i}}})$. On a $l_{1}=2$ ou $l_{1}=3$. Si $l_{1}=2$, on a, par la proposition \ref{38}, $l \leq N$. Donc, $l_{1}=3$.
\\
\\Par la proposition \ref{37}, on a $l={\rm ppcm}(l_{i},~1 \leq i \leq r)$ ou $l=2 \times {\rm ppcm}(l_{i},~1 \leq i \leq r)$. On considère chaque cas séparément.
\\
\\ \uwave{$1^{{\rm er}}$ cas :} On suppose que $l={\rm ppcm}(l_{i},~1 \leq i \leq r)$.
\\
\\S'il existe $j$ dans $[\![2;r]\!]$ tel que $l_{j} \leq \frac{p_{j}+1}{2}p_{j}^{\alpha_{j}-1}$. $p_{j} \geq 3$ donc, par le lemme \ref{312}, $\frac{p_{j}+1}{p_{j}} \leq \frac{4}{3}$. On a : 
\[l \leq \prod_{i=1}^{r} l_{i} \leq 3 \times \frac{p_{j}+1}{2}p_{j}^{\alpha_{j}-1} \prod_{i \neq j} p_{i}^{\alpha_{i}} = \frac{3}{4} \times \frac{p_{j}+1}{p_{j}} \times N \leq N.\]

\noindent Ainsi, pour tout $i$ dans $[\![2;r]\!]$, $l_{i} > \frac{p_{i}+1}{2}p_{i}^{\alpha_{i}-1}$. Donc, par le lemme \ref{36}, $l_{i}=p_{i}^{\alpha_{i}}$ pour tout $i$ dans $[\![2;r]\!]$. S'il existe $j$ dans $[\![2;r]\!]$ tel que $l_{j}=3^{\alpha_{j}}$, on a 
\[l \leq 3^{\alpha_{j}} \times \prod_{i \neq 1,j} l_{i}=m < N.\]

\noindent Comme $l={\rm ppcm}(l_{i},~1 \leq i \leq r)$, il existe $\epsilon \in \{-1,1\}$ tel que $M_{l_{i}}(k+p_{i}^{\alpha_{i}}\mathbb{Z},\ldots,k+p_{i}^{\alpha_{i}}\mathbb{Z})=(\epsilon+p_{i}^{\alpha_{i}}\mathbb{Z})Id$ pour tout $i$ dans $[\![2;r]\!]$, puisque $l_{i}$ est impair pour tout $i$ dans $[\![1;r]\!]$ (lemme \ref{39} ii)). 
\\
\\Par ce qui précède, 3 ne divise pas $\prod_{i=2}^{r} l_{i}$. On a donc $\prod_{i=2}^{r} l_{i} \equiv 1,2 [3]$. On distingue les deux cas :
\\
\\i) $\prod_{i=2}^{r} l_{i} \equiv 1 [3]$. On pose $s=\left(\prod_{i=2}^{r} l_{i}\right)+2$. On a $s \geq 3$ et $s \leq m+2<2m=N <l$. On a :
\[M_{s}(k+2\mathbb{Z},\ldots,k+2\mathbb{Z})=(M_{3}(k+2\mathbb{Z},\ldots,k+2\mathbb{Z}))^{\frac{s}{3}}=Id=(-Id)=(-\epsilon+2\mathbb{Z})Id.\]

\noindent De plus, pour tout $i \neq 1$, $\frac{s-2}{l_{i}}$ est impair et donc :
\begin{eqnarray*}
M &=& M_{s}(0+p_{i}^{\alpha_{i}}\mathbb{Z},k+p_{i}^{\alpha_{i}}\mathbb{Z},\ldots,k+p_{i}^{\alpha_{i}}\mathbb{Z},0+p_{i}^{\alpha_{i}}\mathbb{Z}) \\
  &=& M_{1}(0+p_{i}^{\alpha_{i}}\mathbb{Z})M_{l_{i}}(k+p_{i}^{\alpha_{i}}\mathbb{Z},\ldots,k+p_{i}^{\alpha_{i}}\mathbb{Z})^{\frac{s-2}{l_{i}}}M_{1}(0+p_{i}^{\alpha_{i}}\mathbb{Z}) \\
  &=& M_{1}(0+p_{i}^{\alpha_{i}}\mathbb{Z})((\epsilon+p_{i}^{\alpha_{i}} \mathbb{Z})Id)^{\frac{s-2}{l_{i}}}M_{1}(0+p_{i}^{\alpha_{i}}\mathbb{Z}) \\
	&=& (\epsilon+p_{i}^{\alpha_{i}} \mathbb{Z})M_{2}(0+p_{i}^{\alpha_{i}}\mathbb{Z},0+p_{i}^{\alpha_{i}}\mathbb{Z}) \\
	&=&(-\epsilon+p_{i}^{\alpha_{i}} \mathbb{Z})Id.
\end{eqnarray*}

\noindent Notons $a$ l'unique entier compris entre 0 et $N-1$ tel que $a \equiv 0 [p_{i}^{\alpha_{i}}]$ pour tout $2 \leq i \leq r$ et $a \equiv k [2]$. Par le lemme chinois, on a $M_{s}(\overline{a},\overline{k},\ldots,\overline{k},\overline{a})=\overline{-\epsilon}Id$. Comme $3 \leq s \leq l-1$, on peut utiliser cette solution pour réduire la solution $\overline{k}$-monomiale minimale de \eqref{p} qui est donc réductible.
\\
\\ii) $\prod_{i=2}^{r} l_{i} \equiv 2 [3]$. On pose $s=\prod_{i=2}^{r} l_{i}$. On a $s \geq 3$ et $s \leq m<2m=N <l$. On a :
\begin{eqnarray*}
M &=& M_{s}(0+2\mathbb{Z},k+2\mathbb{Z},\ldots,k+2\mathbb{Z},0+2\mathbb{Z}) \\
  &=& M_{1}(0+2\mathbb{Z})(M_{3}(k+2\mathbb{Z},\ldots,k+2\mathbb{Z}))^{\frac{s-2}{3}}M_{1}(0+2\mathbb{Z}) \\
	&=& Id \\
	&=& -Id \\
	&=&(\epsilon+2\mathbb{Z})Id.
\end{eqnarray*}

\noindent De plus, pour tout $i \neq 1$, $\frac{s}{l_{i}}$ est impair et donc :
\[M_{s}(k+p_{i}^{\alpha_{i}}\mathbb{Z},\ldots,k+p_{i}^{\alpha_{i}}\mathbb{Z})=M_{l_{i}}(k+p_{i}^{\alpha_{i}}\mathbb{Z},\ldots,k+p_{i}^{\alpha_{i}}\mathbb{Z})^{\frac{s}{l_{i}}}=((\epsilon+p_{i}^{\alpha_{i}} \mathbb{Z})Id)^{\frac{s}{l_{i}}}=(\epsilon+p_{i}^{\alpha_{i}} \mathbb{Z})Id.\]

\noindent Notons $a$ l'unique entier compris entre 0 et $N-1$ tel que $a \equiv k [p_{i}^{\alpha_{i}}]$ pour tout $2 \leq i \leq r$ et $a \equiv 0 [2]$. Par le lemme chinois, on a $M_{s}(\overline{a},\overline{k},\ldots,\overline{k},\overline{a})=\overline{\epsilon}Id$. Comme $3 \leq s \leq l-1$, on peut utiliser cette solution pour réduire la solution $\overline{k}$-monomiale minimale de \eqref{p} qui est donc réductible.
\\
\\ \uwave{$2^{{\rm \grave{e}me}}$ cas :} On suppose que $l=2\times{\rm ppcm}(l_{i},~1 \leq i \leq r)$. Par la proposition \ref{38}, $r \geq 3$. On étudie les conséquences de l'hypothèse $l>N$ sur les $l_{i}$.
\\
\\Soient $j$ et $h$ dans $[\![1;r]\!]$ avec $j<h$. Si $l_{j}$ et $l_{h}$ ne sont pas premiers entre eux. Il existe un nombre premier $p$ qui divise $l_{j}$ et $l_{h}$. Si $p$ est impair, on a $p \geq 3$ et donc : 
\[l \leq 2 \times \frac{3}{p} \times \prod_{i=2}^{r} l_{i} \leq 2 \times \frac{3}{p} \prod_{i=2}^{r} p_{i}^{\alpha_{i}}= \frac{3}{p}\times N \leq N.\]

\noindent Si $p$ est pair. On a nécessairement $j,h \neq 1$ et donc $l_{j} \leq p_{j}^{\alpha_{j}}$ et $l_{h} \leq p_{h}^{\alpha_{h}}$. De plus, $l_{j} \neq p_{j}^{\alpha_{j}}$ et $l_{h} \neq p_{h}^{\alpha_{h}}$ Donc, $l_{j} \leq \frac{p_{j}+1}{2}p_{j}^{\alpha_{j}-1}$ et $l_{h} \leq \frac{p_{h}+1}{2}p_{h}^{\alpha_{h}-1}$. De plus, puisque 2 divise $l_{j}$ et $l_{h}$, on a 
\begin{eqnarray*}
l &\leq& 2 \times \frac{3}{2} \times \prod_{i=2}^{r} l_{i} \\
	&\leq& 2 \times \frac{3}{2} \times \frac{p_{j}+1}{2}p_{j}^{\alpha_{j}-1} \times \frac{p_{h}+1}{2}p_{h}^{\alpha_{h}-1} \times \prod_{i \neq j,h} p_{i}^{\alpha_{i}} \\
	&=& \frac{3}{8} \times \frac{p_{j}+1}{p_{j}} \times \frac{p_{h}+1}{p_{h}} \times N \\
	&\leq& \frac{3}{8} \times \frac{4}{3} \times \frac{6}{5} \times N~({\rm lemme~\ref{312}}~{\rm avec}~p_{j}\geq 3~{\rm et}~p_{h} \geq 5) \\
	&=& \frac{3}{5}N \\
	&<& N.
\end{eqnarray*}

\noindent Ainsi, pour tout $i \neq j$, $l_{i}$ et $l_{j}$ sont premiers entre eux et $l=2\prod_{i=1}^{r} l_{i}$.
\\
\\On distingue deux cas :
\\
\\A) Si tous les $l_{i}$ sont impairs. La méthode est globalement la même que celle utilisée dans le cas A) de la démonstration de la proposition \ref{313}. Il suffit simplement d'y ajouter une petite subtilité. En effet, si on reprend les définitions de $I$ et $J$ données dans cette preuve alors on a $I \cap J=\{1\}$. Donc, on retire 1 de $I$. $I-\{1\} \neq \emptyset$ car sinon on aurait $J=[\![1;r]\!]$ et $l={\rm ppcm}(l_{i},~1 \leq i \leq r)$, par le lemme \ref{39} ii).
\\
\\B) S'il existe $j_{0}$ dans $[\![1;r]\!]$  tel que $l_{j_{0}}$ est pair. Nécessairement $j_{0} \neq 1$ puisque $l_{1}=3$. De plus, comme les $l_{i}$ sont deux à deux premiers entre eux, on a, pour tout $i \neq j_{0}$, $l_{i}$ impair.
\\
\\Par le lemme \ref{34} ii), $M_{l_{j_{0}}}(k+p_{j_{0}}^{\alpha_{j_{0}}} \mathbb{Z},\ldots,k+p_{j_{0}}^{\alpha_{j_{0}}} \mathbb{Z})=(-1+p_{j_{0}}^{\alpha_{j_{0}}} \mathbb{Z})Id$. Pour tout $i \neq j_{0}$, il existe $\epsilon_{i} \in \{-1,1\}$ tel que $M_{l_{i}}(k+p_{i}^{\alpha_{i}} \mathbb{Z},\ldots,k+p_{i}^{\alpha_{i}} \mathbb{Z})=(\epsilon_{i}+p_{i}^{\alpha_{i}} \mathbb{Z})Id$.
\\
\\$3l_{j_{0}}$ est inversible modulo $\prod_{i \neq 1,j_{0}} l_{i}$ ($\prod_{i \neq 1,j_{0}} l_{i}>1$ car $r \geq 3$). Par conséquent, il existe un entier $m$ compris entre 0 et $\left(\prod_{i \neq 1,j_{0}} l_{i}\right)-1$ tel que 
\[3l_{j_{0}}m \equiv 2 \left[\prod_{i \neq 1, j_{0}} l_{i}\right].\]

\noindent Si $m$ est impair, on pose $s=3l_{j_{0}}m$. Si $m$ est pair, on pose $s=3l_{j_{0}}\left(m+\prod_{i\neq 1,j_{0}} l_{i}\right)$. $r \geq 3$ et $l_{i}$ impair, pour $i \neq 1, j_{0}$. Donc, $\prod_{i \neq 1,j_{0}} l_{i}>2$ et on a $m \neq 0$. Ainsi, $s \geq 3$. De plus, 
\[s\leq 3l_{j_{0}}\left(\left(\prod_{i \neq 1,j_{0}} l_{i}\right)-1+\prod_{i\neq 1,j_{0}} l_{i}\right)=2\left(\prod_{i=1}^{r} l_{i}\right)-3l_{j_{0}}<l.\]

\noindent Comme $\frac{s}{l_{j_{0}}}$ est impair, on a :
\[M_{s}(k+p_{j_{0}}^{\alpha_{j_{0}}}\mathbb{Z},\ldots,k+p_{j_{0}}^{\alpha_{j_{0}}}\mathbb{Z})=(M_{l_{j_{0}}}(k+p_{j_{0}}^{\alpha_{j_{0}}}\mathbb{Z},\ldots,k+p_{j_{0}}^{\alpha_{j_{0}}}\mathbb{Z}))^{\frac{s}{l_{j_{0}}}}=(-Id)^{\frac{s}{l_{j_{0}}}}=(-1+p_{j_{0}}^{\alpha_{j_{0}}}\mathbb{Z})Id.\]

\noindent Comme $s$ est un multiple de 3, on a
\[M_{s}(k+2\mathbb{Z},\ldots,k+2\mathbb{Z})=(M_{3}(k+2\mathbb{Z},\ldots,k+2\mathbb{Z}))^{\frac{s}{3}}=(1+2\mathbb{Z})Id=(-1+2\mathbb{Z})Id.\]

\noindent De plus, pour tout $i \neq 1,j_{0}$, $\frac{s-2}{l_{i}}$ est pair et donc :
\begin{eqnarray*}
M &=& M_{s}(0+p_{i}^{\alpha_{i}}\mathbb{Z},k+p_{i}^{\alpha_{i}}\mathbb{Z},\ldots,k+p_{i}^{\alpha_{i}}\mathbb{Z},0+p_{i}^{\alpha_{i}}\mathbb{Z}) \\
  &=& M_{1}(0+p_{i}^{\alpha_{i}}\mathbb{Z})M_{l_{i}}(k+p_{i}^{\alpha_{i}}\mathbb{Z},\ldots,k+p_{i}^{\alpha_{i}}\mathbb{Z})^{\frac{s-2}{l_{i}}}M_{1}(0+p_{i}^{\alpha_{i}}\mathbb{Z}) \\
  &=& M_{1}(0+p_{i}^{\alpha_{i}}\mathbb{Z})((\epsilon_{i}+p_{i}^{\alpha_{i}} \mathbb{Z})Id)^{\frac{s-2}{l_{i}}}M_{1}(0+p_{i}^{\alpha_{i}}\mathbb{Z}) \\
	&=& M_{2}(0+p_{i}^{\alpha_{i}}\mathbb{Z},0+p_{i}^{\alpha_{i}}\mathbb{Z}) \\
	&=& (-1+p_{i}^{\alpha_{i}}\mathbb{Z})Id.
\end{eqnarray*}

\noindent Notons $a$ l'unique entier compris entre 0 et $N-1$ tel que $a \equiv 0 [p_{i}^{\alpha_{i}}]$ pour tout $i \neq 1,j_{0}$, $a \equiv k [p_{j_{0}}^{\alpha_{j_{0}}}]$ et $a \equiv k [2]$. Par le lemme chinois, on a $M_{s}(\overline{a},\overline{k},\ldots,\overline{k},\overline{a})=-Id$. Comme $3 \leq s \leq l-1$, on peut utiliser cette solution pour réduire la solution $\overline{k}$-monomiale minimale de \eqref{p} qui est donc réductible.

\end{proof}

\begin{remark}
{\rm À la lueur du théorème \ref{30}, on voit que la borne du théorème \ref{prin} est optimale.
}
\end{remark}

\begin{examples}
{\rm On donne ci-dessous quelques exemples simples d'applications du théorème \ref{prin} :
\begin{itemize}
\item La solution $\overline{3}$-monomiale minimale de $(E_{70})$ est de taille 120 donc elle est réductible;
\item La solution $\overline{7}$-monomiale minimale de $(E_{75})$ est de taille 150 donc elle est réductible;
\item La solution $\overline{17}$-monomiale minimale de $(E_{100})$ est de taille 150 donc elle est réductible;
\item La solution $\overline{38}$-monomiale minimale de $(E_{175})$ est de taille 200 donc elle est réductible.
\end{itemize}
}
\end{examples}

La réciproque du théorème \ref{prin} est fausse en général, même si on exclut le cas de la solution $\overline{0}$-monomiale minimale de \eqref{p}. Par exemple, pour $N=9$ et $k=3$. La solution $\overline{3}$-monomiale minimale de \eqref{p} est de taille 6 et on peut la réduire avec la solution $(\overline{6},\overline{3},\overline{3},\overline{6})$. Cependant, on peut essayer d'aller plus loin et tenter de savoir s'il y a beaucoup d'entier $N$ pour lesquels la réciproque est fausse. Pour cela, on peut constater que ce problème est étroitement lié à la recherche des entiers monomialement réductibles, qui sont les entiers $N$ pour lesquels toutes les solutions monomiales minimales non nulles de \eqref{p} sont irréductibles. Plus précisément, on dispose des deux résultat ci-dessous qui ont servi à établir la classification de ces entiers :

\begin{proposition}[\cite{M2} proposition 3.7]
\label{315}

Soient $N=nm$ avec $n$ et $m$ des entiers naturels premiers entre eux différents de 1. On suppose que $m$ est impair et non divisible par 3. Il existe un entier $1 \leq k \leq N-1$ tel que $k$ est premier avec $N$ et tel que la solution $\overline{k}$-monomiale minimale de \eqref{p} est réductible, de taille $6m$ si $n>2$, et $3m$ sinon.

\end{proposition}

\begin{proposition}[\cite{M2} proposition 3.5]
\label{316}

Soit $N \geq 2$. 
\begin{itemize}
\item Si $N$ est divisible par 16 alors la solution $\overline{\frac{N}{4}}$-monomiale minimale de \eqref{p} est réductible de taille 8.
\item Si $N$ est divisible par le carré d'un nombre premier $p$ alors la solution $\overline{\frac{N}{p}}$-monomiale minimale de \eqref{p} est réductible de taille $2p$.
\end{itemize}

\end{proposition}

Ces résultats permettent de constater qu'il existe de nombreux entiers $N$ pour lesquels il existe une solution monomiale minimale non nulle de \eqref{p} réductible de taille inférieure à $N$.

\subsection{Le cas des entiers de la forme $3m$ avec $m$ premier avec 6}
\label{preuve2}

On va ici considérer le cas laissé en suspens dans le théorème \ref{prin} en démontrant le résultat ci-dessous :

\begin{proposition}
\label{317}

Soit $N=3m$ avec $m$ impair non divisible par 3. Si une solution monomiale minimale de \eqref{p} est de taille $l$ strictement supérieure à $N$ et si $l \neq N+\frac{N}{3}$ alors la solution est réductible.

\end{proposition}

\begin{proof}

Soit $N=3m$ avec $m$ impair non divisible par 3. On a $N=3p_{2}^{\alpha_{2}}\ldots p_{r}^{\alpha_{r}}$ avec $r \geq 2$ et, pour tout $i$ dans $[\![2;r]\!]$, $p_{i}$ des nombres premiers impairs différents de $p_{1}=3$. On note, pour tout $i$ dans $[\![1;r]\!]$, $l_{i}$ la taille de la solution $(k+p_{i}^{\alpha_{i}}\mathbb{Z})$-monomiale minimale de $(E_{p_{i}^{\alpha_{i}}})$. Supposons qu'il existe un entier $k$ tel que \eqref{p} possède une solution $\overline{k}$-monomiale minimale de \eqref{p} de taille $l$ strictement supérieure à $N$ mais différente de $N+\frac{N}{3}$. On va montrer que cette solution est réductible.
\\
\\La preuve est globalement la même que celle de la proposition \ref{313}. En effet, les éléments développés dans la première étape demeure. Ainsi, $l=2\times {\rm ppcm}(l_{i},~1 \leq i \leq r)$, tous les $l_{i}$ sont premiers entre eux et on a deux cas :
\begin{itemize}
\item pour tout $i$ dans $[\![1;r]\!]$, $l_{i}=p_{i}^{\alpha_{i}}$;
\item il existe $j_{0}$ dans $[\![1;r]\!]$ tel que $l_{j_{0}}=\frac{p_{j_{0}}+1}{2}p_{j_{0}}^{\alpha_{j_{0}}-1}$ et pour tout $i \neq j_{0}$ $l_{i}=p_{i}^{\alpha_{i}}$.
\\
\end{itemize}

\noindent Les éléments de la deuxième étape sont également toujours valides. En effet, si tous les $l_{i}$ sont impairs alors on peut procéder comme dans la partie A) de la preuve de la proposition \ref{313}. Supposons maintenant qu'un des $l_{i}$ est pair. Il y a en un seul et c'est nécessairement $l_{j_{0}}$ (puisque tous les $p_{i}$ sont impairs). Si $l_{j_{0}}=2$ alors $\alpha_{j_{0}}=1$, $p_{j_{0}}=p_{1}=3$ et $l=2\prod_{i=1}^{r} l_{i}=4\prod_{i=2}^{r} p_{i}^{\alpha_{i}}=N+\frac{N}{3}$. Donc, $l_{j_{0}} \neq 2$ et on peut reprendre les éléments développés dans la partie B) de la preuve de la proposition \ref{313}.

\end{proof}

Ce résultat laisse entrevoir la possibilité de démontrer le théorème \ref{25}. On va donc donc exploiter cette piste en effectuant la preuve de la proposition ci-dessous :

\begin{proposition}
\label{d1}

Soient $p$ un nombre premier impair différent de 3 et $N=3p$. Il existe une solution irréductible de \eqref{p} de taille $4p$.

\end{proposition}

\begin{proof}

Comme $p$ est premier différent de 3, on a $p \equiv 1 [3]$ ou $p \equiv 2 [3]$. On va distinguer les deux cas :
\\
\\i) Si $p \equiv 1 [3]$. On pose $k=p+2$ et $l$ la taille de la solution $\overline{k}$-monomiale minimale de \eqref{p}. 
\\
\\$k \equiv 0 [3]$ donc la solution $(k+3\mathbb{Z})$-monomiale minimale de $(E_{3})$ est de taille 2 et $k \equiv 2 [p]$ donc la solution $(k+p\mathbb{Z})$-monomiale minimale de $(E_{p})$ est de taille $p$. Ainsi, $l={\rm ppcm}(2,p)=2p$ ou $l=2 \times {\rm ppcm}(2,p)=4p$ (proposition \ref{37}). Si $l=2p$ alors il existe $\epsilon \in \{-1,1\}$ tel que $M_{2p}(\overline{k},\ldots,\overline{k})=\overline{\epsilon} Id$. Or, on a :
\[M_{2p}(k+3\mathbb{Z},\ldots,k+3\mathbb{Z})=(-1+3\mathbb{Z})^{p}Id=(-1+3\mathbb{Z})Id;\]
\[M_{2p}(k+p\mathbb{Z},\ldots,k+p\mathbb{Z})=(1+p\mathbb{Z})^{2}Id=(1+p\mathbb{Z})Id.\]

\noindent Si $\epsilon=1$ alors $\epsilon+3\mathbb{Z}=1+3\mathbb{Z}$. Or, par ce qui précède, $\epsilon+3\mathbb{Z}=-1+3\mathbb{Z}$. Si $\epsilon=-1$ alors $\epsilon+p\mathbb{Z}=-1+p\mathbb{Z}$. Or, par ce qui précède, $\epsilon+p\mathbb{Z}=1+p\mathbb{Z} \neq -1+p\mathbb{Z}$ (car $p \neq 2$). Donc, $l \neq 2p$. Ainsi, $l=4p$. 
\\
\\Supposons par l'absurde que la solution $\overline{k}$-monomiale minimale de \eqref{p} est réductible. Il existe une solution $(\overline{x},\overline{k},\ldots,\overline{k},\overline{y})$ de \eqref{p} de taille $l' \leq 2p+1$.
\\
\\Comme $p$ divise $N$, $(x+p\mathbb{Z},k+p\mathbb{Z},\ldots,k+p\mathbb{Z},y+p\mathbb{Z})$ est une solution de $(E_{p})$. Or, la solution $(2+p\mathbb{Z})$-monomiale minimale de $(E_{p})$ est irréductible de taille $p$. Donc, par le lemme \ref{plus}, $l' \in \{p, p+2, 2p\}$.
\\
\\De plus, 3 divise $N$. Donc, $(x+3\mathbb{Z},k+3\mathbb{Z},\ldots,k+3\mathbb{Z},y+3\mathbb{Z})$ est une solution de $(E_{3})$. Comme la solution $(0+3\mathbb{Z})$-monomiale minimale de $(E_{3})$ est de taille $2$, on a, par le lemme \ref{plus}, $l'$ pair, c'est-à-dire, $l'=2p$.
\\
\\Ainsi, par le lemme \ref{plus}, on a $x+p\mathbb{Z}=y+p\mathbb{Z}=k+p\mathbb{Z}$ et $x+3\mathbb{Z}=y+3\mathbb{Z}=0+3\mathbb{Z}$. Donc,
\[M_{2p}(x+p\mathbb{Z},k+p\mathbb{Z},\ldots,k+p\mathbb{Z},y+p\mathbb{Z})=(1+p\mathbb{Z})Id,\]
\[M_{2p}(x+3\mathbb{Z},k+3\mathbb{Z},\ldots,k+3\mathbb{Z},y+3\mathbb{Z})=(-1+3\mathbb{Z})Id.\]

\noindent Ceci est absurde. Donc, la solution $\overline{k}$-monomiale minimale de \eqref{p} est irréductible.
\\
\\ii) Si $p \equiv 2 [3]$. On pose $k=p-2$ et $l$ la taille de la solution $\overline{k}$-monomiale minimale de \eqref{p}. 
\\
\\$k \equiv 0 [3]$ donc la solution $(k+3\mathbb{Z})$-monomiale minimale de $(E_{3})$ est de taille 2 et $k \equiv -2 [p]$ donc la solution $(k+p\mathbb{Z})$-monomiale minimale de $(E_{p})$ est de taille $p$. Ainsi, $l={\rm ppcm}(2,p)=2p$ ou $l=2 \times {\rm ppcm}(2,p)=4p$ (proposition \ref{37}). Si $l=2p$ alors il existe $\epsilon \in \{-1,1\}$ tel que $M_{2p}(\overline{k},\ldots,\overline{k})=\overline{\epsilon} Id$. Or, on a :
\[M_{2p}(k+3\mathbb{Z},\ldots,k+3\mathbb{Z})=(-1+3\mathbb{Z})^{p}Id=(-1+3\mathbb{Z})Id;\]
\[M_{2p}(k+p\mathbb{Z},\ldots,k+p\mathbb{Z})=(-1+p\mathbb{Z})^{2}Id=(1+p\mathbb{Z})Id.\]

\noindent Si $\epsilon=1$ alors $\epsilon+3\mathbb{Z}=1+3\mathbb{Z}$. Or, par ce qui précède, $\epsilon+3\mathbb{Z}=-1+3\mathbb{Z}$. Si $\epsilon=-1$ alors $\epsilon+p\mathbb{Z}=-1+p\mathbb{Z}$. Or, par ce qui précède, $\epsilon+p\mathbb{Z}=1+p\mathbb{Z} \neq -1+p\mathbb{Z}$ (car $p \neq 2$). Donc, $l \neq 2p$. Ainsi, $l=4p$.
\\
\\Supposons par l'absurde que la solution $\overline{k}$-monomiale minimale de \eqref{p} est réductible. Il existe une solution $(\overline{x},\overline{k},\ldots,\overline{k},\overline{y})$ de \eqref{p} de taille $l' \leq 2p+1$.
\\
\\Comme $p$ divise $N$, $(x+p\mathbb{Z},k+p\mathbb{Z},\ldots,k+p\mathbb{Z},y+p\mathbb{Z})$ est une solution de $(E_{p})$. Or, la solution $(-2+p\mathbb{Z})$-monomiale minimale de $(E_{p})$ est irréductible de taille $p$. Ainsi, par le lemme \ref{plus}, $l' \in \{p, p+2, 2p\}$.
\\
\\De plus, 3 divise $N$. Donc, $(x+3\mathbb{Z},k+3\mathbb{Z},\ldots,k+3\mathbb{Z},y+3\mathbb{Z})$ est une solution de $(E_{3})$. Comme la solution $(0+3\mathbb{Z})$-monomiale minimale de $(E_{3})$ est de taille $2$, on a, par le lemme \ref{plus}, $l'$ pair, c'est-à-dire, $l'=2p$.
\\
\\Ainsi, par le lemme \ref{plus}, on a $x+p\mathbb{Z}=y+p\mathbb{Z}=k+p\mathbb{Z}$ et $x+3\mathbb{Z}=y+3\mathbb{Z}=0+3\mathbb{Z}$. Donc,
\[M_{2p}(x+p\mathbb{Z},k+p\mathbb{Z},\ldots,k+p\mathbb{Z},y+p\mathbb{Z})=(1+p\mathbb{Z})Id,\]
\[M_{2p}(x+3\mathbb{Z},k+3\mathbb{Z},\ldots,k+3\mathbb{Z},y+3\mathbb{Z})=(-1+3\mathbb{Z})Id.\]

\noindent Ceci est absurde. Donc, la solution $\overline{k}$-monomiale minimale de \eqref{p} est irréductible.

\end{proof}

\noindent Grâce à cette proposition, on peut démontrer le résultat souhaité.

\begin{theorem}
\label{25bis}

Il n'existe pas de constante $K \in \mathbb{N}^{*}$ tel que pour tout $N \geq 2$ les solutions irréductibles de \eqref{p} soient de taille inférieure à $N+K$.

\end{theorem}

\begin{proof}

Supposons par l'absurde qu'il existe une constante $K \in \mathbb{N}^{*}$ tel que pour tout $N \in \mathbb{N}^{*}$, $N \geq 2$ les solutions irréductibles de  \eqref{p} sont de taille inférieure à $N+K$.
\\
\\Par la proposition \ref{d1}, on a pour tout nombre premier $p$ impair différent de 3, $3p+p \leq 3p+K$. Donc, $p \leq K$. Ainsi, la suite des nombres premiers $p$ impairs différents de 3 est une suite d'entiers strictement positifs bornée. Donc, il n'y a qu'un nombre fini de nombres premiers, ce qui est absurde.

\end{proof}

\begin{remark}

{\rm On pourrait en fait se contenter dans la proposition \ref{d1} du cas $p \equiv 1 [3]$ et utiliser ensuite le théorème faible de la progression arithmétique de Dirichlet (voir \cite{G} proposition VII.13). Cependant, ce dernier étant beaucoup plus difficile à démontrer que l'infinité des nombres premiers, il semble que la preuve ci-dessus est plus intéressante car plus élémentaire.
}

\end{remark}

\section{Solutions monomiales minimales de taille impaire}
\label{quatre}

Le théorème \ref{prin} montre que, dans la plupart des cas, les solutions de taille strictement supérieure à $N$ sont réductibles. On peut également se demander si certaines tailles conduisent systématiquement à la réductibilité ou à l'irréductibilité d'une solution monomiale minimale. 
\\
\\ \indent Commençons par remarquer que le théorème \ref{30} implique qu'il n'existe pas d'entier $m>2$ qui entraîne automatiquement la réductibilité d'une solution monomiale minimale de taille $m$. On va donc s'intéresser à la potentielle existence d'entiers entraînant l'irréductibilité. Le cas pair est assez rapide. En effet, si $l \geq 3$ alors la solution $\overline{l}$-monomiale minimale de $(E_{l^{2}})$ est réductible, puisqu'on peut la réduire avec $(-\overline{l},\overline{l},\overline{l},-\overline{l})$ ($\overline{l} \neq -\overline{l}$ puisque $l \geq 3$), et elle est de taille $2l$, par la proposition \ref{311}. On sait également qu'une solution monomiale minimale de taille 2 est réductible, puisqu'elle est nécessairement égale à $(\overline{0},\overline{0})$. En revanche, pour tout $N$, les solutions $\overline{k}$-monomiales minimales de $(E_{N})$ de taille 4 sont irréductibles, puisqu'une solution de $(E_{N})$ de taille 4 est réductible si et seulement si elle contient $\pm \overline{1}$, ce qui n'est jamais le cas pour une solution monomiale minimale de taille 4. Il ne reste donc plus qu'à considérer le cas des entiers impairs.

\begin{theorem}
\label{318}

Soit $N>2$. 
\\
\\i) On suppose que $N$ n'est pas de la forme $2m$ avec $m$ impair. Si \eqref{p} possède des solutions monomiales minimales de taille impaire alors elles sont irréductibles.
\\
\\ii) On suppose $N=2m$ avec $m$ impair différent de 1. Si \eqref{p} possède des solutions monomiales minimales de taille $l$ impaire divisible par 9 alors elles sont irréductibles.

\end{theorem}

\begin{proof}

i) On raisonne par récurrence sur le nombre $r$ de facteurs premiers de $N$. 
\\
\\Soit $N \neq 2$ un entier égal à la puissance d'un nombre premier. On suppose qu'il existe un $k \in \mathbb{N}$ tel que la solution $\overline{k}$-monomiale minimale de $(E_{N})$ est de taille impaire $l$. Comme $l$ est impair, $k$ est premier avec $N$ (proposition \ref{311}) et donc la solution $\overline{k}$-monomiale minimale de $(E_{N})$ est irréductible (théorème \ref{311b}).
\\
\\On suppose qu'il existe un $r$ dans $\mathbb{N}^{*}$ tel que, pour tout entier $N$ qui possède $r$ facteurs premiers et qui n'est pas de la forme $2m$ avec $m$ impair, toutes les solutions monomiales minimales de $(E_{N})$ de taille impaire sont irréductibles. Soit $N \geq 2$ un entier qui n'est de la forme $2m$ avec $m$ impair et qui possède $r+1$ facteurs premiers. $N=p_{1}^{\alpha_{1}}\ldots p_{r+1}^{\alpha_{r+1}}$, avec pour tout $i$ dans $[\![1;r+1]\!]$ $p_{i} \in \mathbb{P}$, $\alpha_{i} \geq 1$ et $p_{i} \neq p_{j}$ si $i \neq j$. Soit $k \in \mathbb{N}$ tel que la solution $\overline{k}$-monomiale minimale de $(E_{N})$ est de taille impaire $l$ (notons qu'un tel $k$ existe toujours, par exemple $k=1$). Notons $l_{1}$ la taille de la solution $(k+p_{1}^{\alpha_{1}}\mathbb{Z})$-monomiale minmale de $(E_{p_{1}^{\alpha_{1}}})$ et $l_{2}$ la taille de la solution $\left(k+\frac{N}{p_{1}^{\alpha_{1}}}\mathbb{Z}\right)$-monomiale minmale de $\left(E_{\frac{N}{p_{1}^{\alpha_{1}}}}\right)$.
\\
\\En procédant comme dans les preuves précédentes, on montre que $l={\rm ppcm}(l_{1},l_{2})$ ou $l=2 \times {\rm ppcm}(l_{1},l_{2})$. Comme $l$ est impair, on a nécessairement $l={\rm ppcm}(l_{1},l_{2})$ et $l_{1}, l_{2}$ impairs. En reprenant les éléments donnés dans la preuve du lemme \ref{39}, on montre qu'il existe $\epsilon \in \{-1,1\}$ tel que $M_{l_{1}}(k+p_{1}^{\alpha_{1}}\mathbb{Z},\ldots,k+p_{1}^{\alpha_{1}}\mathbb{Z})=(\epsilon+p_{1}^{\alpha_{1}}\mathbb{Z})Id$ et $M_{l_{2}}\left(k+\frac{N}{p_{1}^{\alpha_{1}}}\mathbb{Z},\ldots,k+\frac{N}{p_{1}^{\alpha_{1}}}\mathbb{Z}\right)=\left(\epsilon+\frac{N}{p_{1}^{\alpha_{1}}}\mathbb{Z}\right)Id$.
\\
\\Supposons par l'absurde que la solution $\overline{k}$-monomiale minimale de $(E_{N})$ est réductible. Il existe une solution $(\overline{x},\overline{k},\ldots,\overline{k},\overline{y})$ de $(E_{N})$ de taille $3 \leq l' \leq l-1$. Il existe $\beta \in \{-1,1\}$ tel que $M_{l'}(\overline{x},\overline{k},\ldots,\overline{k},\overline{y})=\overline{\beta} Id$.
\\
\\Comme $p_{1}^{\alpha_{1}}$ divise $N$, $(x+p_{1}^{\alpha_{1}}\mathbb{Z},k+p_{1}^{\alpha_{1}}\mathbb{Z},\ldots,k+p_{1}^{\alpha_{1}}\mathbb{Z},y+p_{1}^{\alpha_{1}}\mathbb{Z})$ est une solution de $(E_{p_{1}^{\alpha_{1}}})$. Or, par l'initialisation de la récurrence, la solution $(k+p_{1}^{\alpha_{1}}\mathbb{Z})$-monomiale minimale est irréductible, puisque $l_{1}$ est impair. Donc, par le lemme \ref{plus}, on a $l' \equiv 0 [l_{1}]$ et $x+p_{1}^{\alpha_{1}}\mathbb{Z}=y+p_{1}^{\alpha_{1}}\mathbb{Z}=k+p_{1}^{\alpha_{1}}\mathbb{Z}$ ou bien $l' \equiv 2 [l_{1}]$ et $x+p_{1}^{\alpha_{1}}\mathbb{Z}=y+p_{1}^{\alpha_{1}}\mathbb{Z}=0+p_{1}^{\alpha_{1}}\mathbb{Z}$.
\\
\\Comme $\frac{N}{p_{1}^{\alpha_{1}}}$ divise $N$, $\left(x+\frac{N}{p_{1}^{\alpha_{1}}}\mathbb{Z},k+\frac{N}{p_{1}^{\alpha_{1}}}\mathbb{Z},\ldots,k+\frac{N}{p_{1}^{\alpha_{1}}}\mathbb{Z},y+\frac{N}{p_{1}^{\alpha_{1}}}\mathbb{Z}\right)$ est une solution de $\left(E_{\frac{N}{p_{1}^{\alpha_{1}}}}\right)$. Or, par l'hypothèse de récurrence, la solution $\left(k+\frac{N}{p_{1}^{\alpha_{1}}}\mathbb{Z}\right)$-monomiale minimale est irréductible, puisque $\frac{N}{p_{1}^{\alpha_{1}}}$ n'a que $r$ facteurs premiers, $\frac{N}{p_{1}^{\alpha_{1}}}$ n'est pas de la forme $2m$ avec $m$ impair et $l_{2}$ est impair. Donc, par le lemme \ref{plus}, on a $l' \equiv 0 [l_{2}]$ et $x+\frac{N}{p_{1}^{\alpha_{1}}}\mathbb{Z}=y+\frac{N}{p_{1}^{\alpha_{1}}}\mathbb{Z}=k+\frac{N}{p_{1}^{\alpha_{1}}}\mathbb{Z}$ ou bien $l' \equiv 2 [l_{2}]$ et $x+\frac{N}{p_{1}^{\alpha_{1}}}\mathbb{Z}=y+\frac{N}{p_{1}^{\alpha_{1}}}\mathbb{Z}=0+\frac{N}{p_{1}^{\alpha_{1}}}\mathbb{Z}$.
\\
\\Si $l' \equiv 0 [l_{1}]$ et $l'\equiv 0 [l_{2}]$ alors $l' \equiv 0 [{\rm ppcm}(l_{1},l_{2})]$, ce qui est impossible puisque $3 \leq l' \leq l-1$. Si $l' \equiv 2 [l_{1}]$ et $l'\equiv 2 [l_{2}]$ alors $l' \equiv 2 [{\rm ppcm}(l_{1},l_{2})]$, ce qui est impossible puisque $3 \leq l' \leq l-1$.  Ainsi, on a deux cas : $l' \equiv 0 [l_{1}]$ et $l'\equiv 2 [l_{2}]$ ou bien $l' \equiv 2 [l_{1}]$ et $l'\equiv 0 [l_{2}]$.
\\
\\Supposons que $l' \equiv 0 [l_{1}]$ et $l'\equiv 2 [l_{2}]$. Il existe $(a,b) \in (\mathbb{N}^{*})^{2}$ tel que $l'=al_{1}$ et $l'=bl_{2}+2$. On a $x+p_{1}^{\alpha_{1}}\mathbb{Z}=y+p_{1}^{\alpha_{1}}\mathbb{Z}=k+p_{1}^{\alpha_{1}}\mathbb{Z}$ et $x+\frac{N}{p_{1}^{\alpha_{1}}}\mathbb{Z}=y+\frac{N}{p_{1}^{\alpha_{1}}}\mathbb{Z}=0+\frac{N}{p_{1}^{\alpha_{1}}}\mathbb{Z}$. Supposons que $a$ est pair. On a $l'$ pair et, comme $l_{2}$ est impair, $b$ est pair. On a :
\begin{eqnarray*}
M &=& M_{l'}(x+p_{1}^{\alpha_{1}}\mathbb{Z},k+p_{1}^{\alpha_{1}}\mathbb{Z},\ldots,k+p_{1}^{\alpha_{1}}\mathbb{Z},y+p_{1}^{\alpha_{1}}\mathbb{Z}) \\
  &=& (M_{l_{1}}(k+p_{1}^{\alpha_{1}}\mathbb{Z},\ldots,k+p_{1}^{\alpha_{1}}\mathbb{Z}))^{a} \\
	&=& ((\epsilon+p_{1}^{\alpha_{1}}\mathbb{Z})Id)^{a} \\
	&=& (1+p_{1}^{\alpha_{1}}\mathbb{Z})Id.
\end{eqnarray*}
\noindent De plus, on a :
\begin{eqnarray*}
N &=& M_{l'}\left(x+\frac{N}{p_{1}^{\alpha_{1}}}\mathbb{Z},k+\frac{N}{p_{1}^{\alpha_{1}}}\mathbb{Z},\ldots,k+\frac{N}{p_{1}^{\alpha_{1}}}\mathbb{Z},y+\frac{N}{p_{1}^{\alpha_{1}}}\mathbb{Z}\right) \\
\end{eqnarray*}
\begin{eqnarray*}
N  &=& M_{1}\left(0+\frac{N}{p_{1}^{\alpha_{1}}}\mathbb{Z}\right)\left(M_{l_{2}}\left(k+\frac{N}{p_{1}^{\alpha_{1}}}\mathbb{Z},\ldots,k+\frac{N}{p_{1}^{\alpha_{1}}}\mathbb{Z}\right)\right)^{b}M_{1}\left(0+\frac{N}{p_{1}^{\alpha_{1}}}\mathbb{Z}\right) \\
	&=& M_{1}\left(0+\frac{N}{p_{1}^{\alpha_{1}}}\mathbb{Z}\right)\left(\left(\epsilon+\frac{N}{p_{1}^{\alpha_{1}}}\mathbb{Z}\right)Id\right)^{b}M_{1}\left(0+\frac{N}{p_{1}^{\alpha_{1}}}\mathbb{Z}\right) \\
	&=& M_{1}\left(0+\frac{N}{p_{1}^{\alpha_{1}}}\mathbb{Z}\right)\left(\left(1+\frac{N}{p_{1}^{\alpha_{1}}}\right)Id\right)M_{1}\left(0+\frac{N}{p_{1}^{\alpha_{1}}}\mathbb{Z}\right) \\
	&=& \left(-1+\frac{N}{p_{1}^{\alpha_{1}}}\mathbb{Z}\right)Id.
\end{eqnarray*}

\noindent Si $\beta=1$ alors $\beta+\frac{N}{p_{1}^{\alpha_{1}}}\mathbb{Z}=1+\frac{N}{p_{1}^{\alpha_{1}}}\mathbb{Z}$. Or, par ce qui précède, $\beta+\frac{N}{p_{1}^{\alpha_{1}}}\mathbb{Z}=-1+\frac{N}{p_{1}^{\alpha_{1}}}\mathbb{Z} \neq 1+\frac{N}{p_{1}^{\alpha_{1}}}\mathbb{Z}$ puisque $\frac{N}{p_{1}^{\alpha_{1}}} \neq 2$. Si $\beta=-1$ alors $\beta+p_{1}^{\alpha_{1}}\mathbb{Z}=-1+p_{1}^{\alpha_{1}}\mathbb{Z}$. Or, par ce qui précède, $\beta+p_{1}^{\alpha_{1}}\mathbb{Z}=1+p_{1}^{\alpha_{1}}\mathbb{Z} \neq -1+p_{1}^{\alpha_{1}}\mathbb{Z}$ puisque $p_{1}^{\alpha_{1}} \neq 2$. 
\\
\\Supposons que $a$ est impair. On a $l'$ impair et, comme $l_{2}$ est impair, $b$ est impair. On a :
\begin{eqnarray*}
M &=& M_{l'}(x+p_{1}^{\alpha_{1}}\mathbb{Z},k+p_{1}^{\alpha_{1}}\mathbb{Z},\ldots,k+p_{1}^{\alpha_{1}}\mathbb{Z},y+p_{1}^{\alpha_{1}}\mathbb{Z}) \\
  &=& (M_{l_{1}}(k+p_{1}^{\alpha_{1}}\mathbb{Z},\ldots,k+p_{1}^{\alpha_{1}}\mathbb{Z}))^{a} \\
	&=& ((\epsilon+p_{1}^{\alpha_{1}}\mathbb{Z})Id)^{a} \\
	&=& (\epsilon+p_{1}^{\alpha_{1}}\mathbb{Z})Id.
\end{eqnarray*}
\noindent De plus, on a :
\begin{eqnarray*}
N &=& M_{l'}\left(x+\frac{N}{p_{1}^{\alpha_{1}}}\mathbb{Z},k+\frac{N}{p_{1}^{\alpha_{1}}}\mathbb{Z},\ldots,k+\frac{N}{p_{1}^{\alpha_{1}}}\mathbb{Z},y+\frac{N}{p_{1}^{\alpha_{1}}}\mathbb{Z}\right) \\
  &=& M_{1}\left(0+\frac{N}{p_{1}^{\alpha_{1}}}\mathbb{Z}\right)\left(M_{l_{2}}\left(k+\frac{N}{p_{1}^{\alpha_{1}}}\mathbb{Z},\ldots,k+\frac{N}{p_{1}^{\alpha_{1}}}\mathbb{Z}\right)\right)^{b}M_{1}\left(0+\frac{N}{p_{1}^{\alpha_{1}}}\mathbb{Z}\right) \\
	&=& M_{1}\left(0+\frac{N}{p_{1}^{\alpha_{1}}}\mathbb{Z}\right)\left(\left(\epsilon+\frac{N}{p_{1}^{\alpha_{1}}}\mathbb{Z}\right)Id\right)^{b}M_{1}\left(0+\frac{N}{p_{1}^{\alpha_{1}}}\mathbb{Z}\right) \\
	&=& M_{1}\left(0+\frac{N}{p_{1}^{\alpha_{1}}}\mathbb{Z}\right)\left(\left(\epsilon+\frac{N}{p_{1}^{\alpha_{1}}}\mathbb{Z}\right)Id\right) M_{1}\left(0+\frac{N}{p_{1}^{\alpha_{1}}}\mathbb{Z}\right) \\
	&=& \left(-\epsilon+\frac{N}{p_{1}^{\alpha_{1}}}\mathbb{Z}\right)Id.
\end{eqnarray*}

\noindent Si $\beta=\epsilon$ alors $\beta+\frac{N}{p_{1}^{\alpha_{1}}}\mathbb{Z}=\epsilon+\frac{N}{p_{1}^{\alpha_{1}}}\mathbb{Z}$. Or, par ce qui précède, $\beta+\frac{N}{p_{1}^{\alpha_{1}}}\mathbb{Z}=-\epsilon+\frac{N}{p_{1}^{\alpha_{1}}}\mathbb{Z} \neq \epsilon+\frac{N}{p_{1}^{\alpha_{1}}}\mathbb{Z}$ puisque $\frac{N}{p_{1}^{\alpha_{1}}} \neq 2$. Si $\beta=-\epsilon$ alors $\beta+p_{1}^{\alpha_{1}}\mathbb{Z}=-\epsilon+p_{1}^{\alpha_{1}}\mathbb{Z}$. Or, par ce qui précède, $\beta+p_{1}^{\alpha_{1}}\mathbb{Z}=\epsilon+p_{1}^{\alpha_{1}}\mathbb{Z} \neq -\epsilon+p_{1}^{\alpha_{1}}\mathbb{Z}$ puisque $p_{1}^{\alpha_{1}} \neq 2$.
\\
\\Si $l' \equiv 2 [l_{1}]$ et $l'\equiv 0 [l_{2}]$ alors on procède d'une façon analogue. Ainsi, on arrive à une absurdité dans les deux cas. Donc, la solution $\overline{k}$-monomiale minimale de $(E_{N})$ est irréductible. 
\\
\\Par récurrence, le résultat est démontré.
\\
\\ii) On suppose que $N=2m$ avec $m \neq 1$ impair. On suppose qu'il existe un $k \in \mathbb{N}$ tel que la solution $\overline{k}$-monomiale minimale de $(E_{N})$ est de taille $l$ impaire divisible par 9. Soit $u$ la taille de la solution $(k+2\mathbb{Z})$-monomiale minimale de $(E_{2})$ et $v$ la taille de la solution $(k+m\mathbb{Z})$-monomiale minimale de $(E_{m})$.
\\
\\Par la proposition \ref{38}, $l={\rm ppcm}(u,v)$. Comme $l$ est impair, $u$ est impair. Comme $u \in \{2,3\}$, on a $u=3$. Puisque $l$ est divisible par $9$, $v$ est nécessairement divisible par 9. En particulier, $l=v$.
\\
\\Supposons par l'absurde que la solution $\overline{k}$-monomiale minimale de $(E_{N})$ est réductible. Il existe une solution $(\overline{x},\overline{k},\ldots,\overline{k},\overline{y})$ de $(E_{N})$ de taille $l'$ avec $3 \leq l' \leq v-1$. Comme $m$ divise $N$, le $l'$-uplet $(x+m\mathbb{Z},k+m\mathbb{Z},\ldots,k+m\mathbb{Z},y+m\mathbb{Z})$ est une solution de $(E_{m})$. Donc, on peut l'utiliser pour réduire la solution $(k+m\mathbb{Z})$-monomiale minimale de $(E_{m})$. Or, puisque $m$ et $v$ sont impairs, cette dernière est irréductible, par le point i). On arrive donc à une contradiction.
\\
\\Ainsi, la solution $\overline{k}$-monomiale minimale de $(E_{N})$ est irréductible. 

\end{proof}

\begin{examples}
{\rm \begin{itemize}
\item La solution $\overline{5}$-monomiale minimale de $(E_{52})$ est de taille 21. Par le théorème précédent, elle est irréductible.
\item La solution $\overline{3}$-monomiale minimale de $(E_{65})$ est de taille 35. Par le théorème précédent, elle est irréductible.
\end{itemize}
}
\end{examples}

À la lueur du résultat précédent, on constate que les solutions monomiales minimales dont la taille est impaire non divisible par 3 ou divisible par 9 sont irréductibles, puisque les solutions monomiales minimales de taille impaire de $(E_{2m})$ ($m$ impair) ont une taille divisible par 3. Pour répondre à la question initiale, il reste donc à considérer le cas des solutions de taille $3h$ avec $h$ impair non divisible par 3. Si $h=1$ alors les solutions monomiales minimales de taille $3h$ de $(E_{N})$ sont les solutions $\pm \overline{1}$-monomiales minimales qui sont irréductibles. Si $h \neq 1$ alors la proposition \ref{315} nous donne l'existence d'une solution monomiale minimale réductible de $(E_{2h})$ de taille $3h$.

\begin{remark}
{\rm Il n'est en réalité pas étonnant que des solutions monomiales minimales réductibles de taille impaire puisse exister modulo $2m$ puisqu'il est possible d'avoir des solutions monomiales minimales de $(E_{2m})$ de taille impaire strictement supérieure à $2m$, qui sont donc nécessairement réductible par le théorème \ref{prin}.
}
\end{remark}

Les solutions de taille $3h$, avec $h>1$ impair non divisible par 3, de $(E_{2m})$ ne sont pas nécessairement réductibles. Par exemple, pour $N=62=2 \times 31$, la solution $\overline{3}$-monomiale minimale de $(E_{62})$ est irréductible de taille 15. On peut donc légitimement se demander si on peut trouver des conditions permettant de déterminer les cas dans lesquels il y a réductibilité. Le résultat ci-dessous apporte quelques éléments susceptibles d'éclairer notre lanterne :

\begin{proposition}

Soit $m$ un entier naturel impair. S'il existe une solution $\overline{k}$-monomiale minimale de $(E_{2m})$ de taille $3h$ avec $h \neq 1$ impair non divisible par 3 alors elle est irréductible si et seulement si la taille de la solution $(k+m\mathbb{Z})$-monomiale minimale de $(E_{m})$ est divisible par 3.

\end{proposition}

\begin{proof}

Soient $N=2m$, $l_{1}$ la taille de la solution $(k+2\mathbb{Z})$-monomiale minimale de $(E_{2})$ et $l_{2}$ la taille de la solution $(k+m\mathbb{Z})$-monomiale minimale de $(E_{m})$. Par la proposition \ref{38}, on a $3h={\rm ppcm}(l_{1},l_{2})$. $l_{1}=2$ ou $l_{1}=3$. Si $l_{1}=2$ alors $3h$ serait pair. Donc, $l_{1}=3$.
\\
\\Si $l_{2}$ est divisible par 3. On a $3h={\rm ppcm}(l_{1},l_{2})=l_{2}$. Supposons par l'absurde que la solution $\overline{k}$-monomiale minimale de $(E_{2m})$ est réductible. Il existe une solution $(\overline{x},\overline{k},\ldots,\overline{k},\overline{y})$ de $(E_{N})$ de taille $3 \leq l' \leq 3h-1$. Comme $m$ divise $N$, $(x+m\mathbb{Z},k+m\mathbb{Z},\ldots,k+m\mathbb{Z},y+m\mathbb{Z})$ est une solution de $(E_{m})$. On peut l'utiliser pour réduire la solution $(k+m\mathbb{Z})$-monomiale minimale de $(E_{m})$, ce qui est absurde. Ainsi, la solution $\overline{k}$-monomiale minimale de $(E_{N})$ est irréductible.
\\
\\Si $l_{2}$ n'est pas divisible par 3. On $3h=3l_{2}$, c'est-à-dire $l_{2}=h$. Il existe $\epsilon \in \{-1,1\}$ tel que 
\[M_{h}(k+m\mathbb{Z},\ldots,k+m\mathbb{Z})=(\epsilon+m\mathbb{Z}) Id.\]
\noindent Comme $h$ n'est pas divisible par 3, on a $h \equiv 1 [3]$ ou $h \equiv 2 [3]$. On distingue les deux cas. 
\\
\\i) Si $h \equiv 1 [3]$, on pose $x$ l'unique élément de $[\![0;N-1]\!]$ tel que $x \equiv 0 [m]$ et $x \equiv k [2]$. On a :
\begin{eqnarray*}
M &=& M_{h+2}(x+m\mathbb{Z},k+m\mathbb{Z},\ldots,k+m\mathbb{Z},x+m\mathbb{Z}) \\
  &=& M_{1}(0+m\mathbb{Z})M_{h}(k+m\mathbb{Z},\ldots,k+m\mathbb{Z})M_{1}(0+m\mathbb{Z}) \\
  &=& (-\epsilon+m\mathbb{Z})Id.
\end{eqnarray*}

\noindent On a également :
\begin{eqnarray*}
M_{h+2}(x+2\mathbb{Z},k+2\mathbb{Z},\ldots,k+2\mathbb{Z},x+2\mathbb{Z}) &=& (M_{3}(k+2\mathbb{Z},\ldots,k+2\mathbb{Z}))^{\frac{h+2}{3}} \\
                                                                        &=&Id \\
																																				&=&-Id\\
																																				&=&(-\epsilon+2\mathbb{Z})Id.
\end{eqnarray*}

\noindent Par le lemme chinois, on a $M_{h+2}(\overline{x},\overline{k},\ldots,\overline{k},\overline{x})=-\overline{\epsilon} Id$. De plus, on a $3 \leq h+2 \leq 3h-1$. Ainsi, la solution $\overline{k}$-monomiale minimale de $(E_{N})$ est réductible.
\\
\\ii) Si $h \equiv 2 [3]$, on pose $x$ l'unique élément de $[\![0;N-1]\!]$ tel que $x \equiv k [m]$ et $x \equiv 0 [2]$. On a :
\[M_{h}(x+m\mathbb{Z},k+m\mathbb{Z},\ldots,k+m\mathbb{Z},x+m\mathbb{Z})=M_{h}(k+m\mathbb{Z},\ldots,k+m\mathbb{Z})=(\epsilon+m\mathbb{Z})Id.\]

\noindent On a également :
\begin{eqnarray*}
M &=& M_{h}(x+2\mathbb{Z},k+2\mathbb{Z},\ldots,k+2\mathbb{Z},x+2\mathbb{Z}) \\
  &=& M_{1}(0+2\mathbb{Z})(M_{3}(k+2\mathbb{Z},\ldots,k+2\mathbb{Z}))^{\frac{h-2}{3}}M_{1}(0+2\mathbb{Z}) \\
	&=& Id \\
	&=&-Id \\
	&=& (\epsilon+2\mathbb{Z})Id.
\end{eqnarray*}

\noindent Par le lemme chinois, on a $M_{h}(\overline{x},\overline{k},\ldots,\overline{k},\overline{x})=\overline{\epsilon} Id$. De plus, on a $3 \leq h \leq 3h-1$. Ainsi, la solution $\overline{k}$-monomiale minimale de $(E_{N})$ est réductible.
\\
\\Le résultat est donc démontré.

\end{proof}

\section{Étude de certaines tailles dépendant de $N$}
\label{cinq}

Dans cette section, on va s'intéresser à l'irréductibilité des solutions monomiales minimales de \eqref{p} dont la taille vérifie certaines propriétés liées à $N$.

\subsection{Solutions monomiales minimales de \eqref{p} taille $N$}

Dans le théorème \ref{prin}, on a obtenu la borne $N$ comme taille maximale des solutions monomiales minimales irréductibles lorsque $N$est différent de 2 et n'est pas de la forme $3m$ avec $m$ impair non divisible par 3. En utilisant le théorème \ref{30}, on a constaté que cette borne était optimale. Toutefois, une question reste en suspens : que peut-on dire des solutions de $(E_{N})$ de taille $N$ ? Plus exactement, l'interrogation principale est de savoir si ces dernières sont systématiquement irréductibles ou non. Pour répondre à cette question, on va démontrer le résultat ci-dessous :

\begin{theorem}
\label{320}

Soit $N>2$. 
\\
\\i) Si $N$ n'est pas de la forme $2 \times 3^{a} \times b$, avec $a \geq 1$ et $b>1$ impair non divisible par 3, alors toutes les solutions monomiales minimales de $(E_{N})$ de taille $N$ sont irréductibles.
\\
\\ii) Si $N=2 \times 3^{a} \times b$, avec $a \geq 1$ et $b>1$ impair non divisible par 3, alors il existe des solutions monomiales minimales de $(E_{N})$ de taille $N$ réductibles.

\end{theorem}

\noindent Comme pour le théorème \ref{prin}, on divise la preuve en plusieurs morceaux.

\begin{proposition}
\label{321}

Soit $N$ un entier naturel qui n'est pas de la forme $2m$, avec $m$ impair. Toutes les solutions monomiales minimales de $(E_{N})$ de taille $N$ sont irréductibles.

\end{proposition}

\begin{proof}

Soit $N=p_{1}^{\alpha_{1}}\ldots p_{r}^{\alpha_{r}}$. Soit $k$ un entier naturel tel que la solution $\overline{k}$-monomiale minimale de $(E_{N})$ est de taille $N$ (notons que, par le théorème \ref{30}, un tel $k$ existe toujours). Pour tout $i$ dans $[\![1;r]\!]$, on note $l_{i}$ la taille de la solution $(k+p_{i}^{\alpha_{i}}\mathbb{Z})$-monomiale minimale de $(E_{p_{i}^{\alpha_{i}}})$ et $l$ la taille de la solution $\overline{k}$-monomiale minimale de $(E_{N})$. Par hypothèse, $l=N$. On procède par étapes.
\\
\\\uwave{$1^{{\rm \grave{e}re}}$ étape :} On montre qu'on peut supposer $N$ pair et $r \geq 2$.
\\
\\Si $N$ est impair alors la solution $\overline{k}$-monomiale minimale de $(E_{N})$ est irréductible par le théorème \ref{318}. On suppose maintenant que $N$ est pair. Si $r=1$ alors, par la proposition \ref{311}, $k$ est impair ou $k=2a$ avec $a$ impair. Par le théorème \ref{311b}, la solution $\overline{k}$-monomiale minimale de $(E_{N})$ est irréductible.
\\
\\On peut donc supposer que $N$ est pair et que $r \geq 2$. Quitte à modifier l'ordre des facteurs premiers, on suppose que $p_{1}=2$. Par hypothèse, $\alpha_{1} \geq 2$. On suppose pour le moment $l_{1} \neq 2$.
\\
\\\uwave{$2^{{\rm \grave{e}me}}$ étape :} On montre que pour tout $i$, $l_{i}=p_{i}^{\alpha_{i}}$.
\\
\\Par la proposition \ref{37}, $l={\rm ppcm}(l_{i},~1 \leq i \leq r)$ ou $l=2\times {\rm ppcm}(l_{i},~1 \leq i \leq r)$. Montrons que $l={\rm ppcm}(l_{i},~1 \leq i \leq r)$. Pour cela, supposons par l'absurde que $l=2\times {\rm ppcm}(l_{i},~1 \leq i \leq r)$. 
\\
\\Supposons par l'absurde qu'il existe $j$ dans $[\![2;r]\!]$ tel que $l_{j} \leq \frac{p_{j}-1}{2}p_{j}^{\alpha_{j}-1}$. Comme $(p_{j}-1)p_{j}^{\alpha_{j}-1}< p_{j}^{\alpha_{j}}$, on a :
\[l=2 \times {\rm ppcm}(l_{i},~1 \leq i \leq r) \leq 2 \prod_{i=1}^{r} l_{i} \leq (p_{j}-1)p_{j}^{\alpha_{j}-1} \prod_{i\neq j} p_{i}^{\alpha_{i}}<\prod_{i=1}^{r} p_{i}^{\alpha_{i}}=N.\]
\noindent Donc, pour tout $i \in [\![2;r]\!]$, on a $l_{i} > \frac{p_{i}-1}{2}p_{i}^{\alpha_{i}-1}$ et, par le lemme \ref{36}, on a $l_{i}=\frac{p_{i}+1}{2}p_{i}^{\alpha_{i}-1}$ ou $l_{i}=p_{i}^{\alpha_{i}}$. 
\\
\\Supposons par l'absurde que $r \geq 3$ et qu'il existe $j$ et $h$ dans $[\![2;r]\!]$, $j<h$, tel que $l_{j}=\frac{p_{j}+1}{2}p_{j}^{\alpha_{j}-1}$ et $l_{h}=\frac{p_{h}+1}{2}p_{h}^{\alpha_{h}-1}$. Par le lemme \ref{312}, on a, puisque $p_{j}\geq 3$ et $p_{h} \geq 5$ :
\[l \leq 2\left(\frac{p_{j}+1}{2}p_{j}^{\alpha_{j}-1}\right)\left(\frac{p_{h}+1}{2}p_{h}^{\alpha_{h}-1}\right)\prod_{i \neq j,h} p_{i}^{\alpha_{i}}=\frac{1}{2} \times \frac{p_{j}+1}{p_{j}} \times \frac{p_{h}+1}{p_{h}} \times N<\frac{1}{2} \times \frac{4}{3} \times \frac{6}{5} \times N=\frac{4}{5}N<N.\]

\noindent Ainsi, on a deux cas :
\begin{itemize}
\item pour tout $i$ dans $[\![2;r]\!]$ $l_{i}=p_{i}^{\alpha_{i}}$;
\item il existe $j_{0}$ dans $[\![2;r]\!]$ tel que $l_{j_{0}}=\frac{p_{j_{0}}+1}{2}p_{j_{0}}^{\alpha_{j_{0}}-1}$ et, pour tout $i$ dans $[\![2;r]\!]$ différent de $j_{0}$, $l_{i}=p_{i}^{\alpha_{i}}$.
\\
\end{itemize}

\noindent De plus, les $l_{i}$ n'ont pas de diviseur impair en commun. En effet, supposons par l'absurde qu'il existe $j$ et $h$ dans $[\![1;r]\!]$ et un nombre premier $p$ impair tel que $p$ divise $l_{j}$ et $l_{h}$. On a :
\[l \leq \frac{2}{p} \prod_{i=1}^{r} l_{i} \leq \frac{2}{3} \prod_{i=1}^{r} p_{i}^{\alpha_{i}}=\frac{2N}{3}<N.\]
\noindent En particulier, si, pour tout $j$ dans $[\![2;r]\!]$, $l_{j}$ est impair alors les $l_{i}$ sont deux à deux premiers entre eux.
\\
\\Concentrons-nous maintenant sur les valeurs possibles de $l_{1}$. Si $l_{1}<2^{\alpha_{1}-1}$ alors \[l\leq 2 \prod_{i=1}^{r} l_{i} < 2 \times 2^{\alpha_{1}-1} \prod_{i=2}^{r} p_{i}^{\alpha_{i}} =N.\]
\noindent Donc, $l_{1} \geq 2^{\alpha_{1}-1}$. On suppose que $l_{1}=2^{\alpha_{1}-1}$. Pour tout $i$ dans $[\![2;r]\!]$, on a $l_{i}=p_{i}^{\alpha_{i}}$ (puisque sinon il existe un $j$ tel que $l_{j} < p_{j}^{\alpha_{j}}$ et, en procédant comme ci-dessus, on obtient $l<N$). Donc, $l_{1}$ est pair différent de 2 et pour tout $i$ dans $[\![2;r]\!]$ $l_{i}$ est impair. Ainsi, par le lemme \ref{39}, on a $l={\rm ppcm}(l_{i},~1 \leq i \leq r)$, ce qui contredit l'hypothèse de départ. Par conséquent, on a $l_{1}=2^{\alpha_{1}}$ ou $l_{1}=3\times 2^{\alpha_{1}-2}$ (lemme \ref{361}).
\\
\\Supposons par l'absurde que $l_{1}=2^{\alpha_{1}}$. Si pour tout $i \in [\![2;r]\!]$ $l_{i}=p_{i}^{\alpha_{i}}$ alors $l=2N$. Donc, il existe $j_{0}$ dans $[\![2;r]\!]$ tel que $l_{j_{0}}=\frac{p_{j_{0}}+1}{2}p_{j_{0}}^{\alpha_{j_{0}}-1}$ et, pour tout $i$ dans $[\![2;r]\!]$ différent de $j_{0}$, $l_{i}=p_{i}^{\alpha_{i}}$. Si $l_{j_{0}}$ est pair, on a par le lemme \ref{312} (avec $p_{j_{0}} \geq 3$)
\[l \leq \frac{2}{2} \frac{p_{j_{0}}+1}{2}p_{j_{0}}^{\alpha_{j_{0}}-1} \prod_{i\neq j_{0}} p_{i}^{\alpha_{i}}=\frac{p_{j_{0}}+1}{2p_{j_{0}}} N \leq \frac{2N}{3}<N.\] 
\noindent Si $l_{j_{0}}$ est impair, on a $l=2\times \frac{p_{j_{0}}+1}{2}p_{j_{0}}^{\alpha_{j_{0}}-1} \prod_{i\neq j_{0}} p_{i}^{\alpha_{i}}=\frac{p_{j_{0}}+1}{p_{j_{0}}} N \neq N$.
\\
\\Donc, $l_{1}=3\times 2^{\alpha_{1}-2}$. Supposons que pour tout $i$ dans $[\![2;r]\!]$ $l_{i}=p_{i}^{\alpha_{i}}$. Comme pour tout $i$ $p_{i} \neq 3$ (puisque les $l_{i}$ n'ont pas de diviseur impair en commun), on a $l=2\times 3\times 2^{\alpha_{1}-2} \prod_{i=2}^{r} p_{i}^{\alpha_{i}}=\frac{3N}{2}$. Donc, il existe $j_{0}$ dans $[\![2;r]\!]$ tel que $l_{j_{0}}=\frac{p_{j_{0}}+1}{2}p_{j_{0}}^{\alpha_{j_{0}}-1}$ et, pour tout $i$ dans $[\![2;r]\!]$ différent de $j_{0}$, $l_{i}=p_{i}^{\alpha_{i}}$. Si $l_{j_{0}}$ est pair, on a par le lemme \ref{312} (avec $p_{j_{0}} \geq 3$)
\[l \leq \frac{2}{2} \frac{p_{j_{0}}+1}{2}p_{j_{0}}^{\alpha_{j_{0}}-1} \prod_{i\neq j_{0}} l_{i}=\frac{3}{4}\frac{p_{j_{0}}+1}{2p_{j_{0}}} N \leq \frac{N}{2}<N.\]
\noindent Si $l_{j_{0}}$ est impair, on a $l=2\times 3\times 2^{\alpha_{1}-2} \times \frac{p_{j_{0}}+1}{2}p_{j_{0}}^{\alpha_{j_{0}}-1} \prod_{i\neq 1,j_{0}} p_{i}^{\alpha_{i}}=\frac{3}{4}\frac{p_{j_{0}}+1}{p_{j_{0}}} N$ (car les $l_{i}$ sont premiers entre eux puisqu'ils n'ont pas de diviseur impair en commun et sont tous impairs sauf $l_{1}$). Or, $\frac{3}{4}\frac{p_{j_{0}}+1}{p_{j_{0}}}=1$ si et seulement si $p_{j_{0}}=3$. Si $p_{j_{0}}=3$ alors $\alpha_{j_{0}}=1$ (car $l_{1}$ et $l_{j_{0}}$ n'ont pas de diviseur impair en commun) et donc $l_{j_{0}}=2$, ce qui impossible puisqu'on a supposé $l_{j_{0}}$ impair.
\\
\\Ainsi, $l={\rm ppcm}(l_{i},~1 \leq i \leq r)$. S'il existe un $j$ dans $[\![1;r]\!]$ tel que $l_{j}<p_{j}^{\alpha_{j}}$ alors \[l\leq \prod_{i=1}^{r} l_{i} < p_{j}^{\alpha_{j}} \prod_{i \neq j} p_{i}^{\alpha_{i}} =N.\]
\noindent Donc, pour tout $i$, $l_{i}=p_{i}^{\alpha_{i}}$.
\\
\\\uwave{$3^{{\rm \grave{e}me}}$ étape :} On montre que la solution $\overline{k}$-monomiale minimale de $(E_{N})$ est irréductible.
\\
\\Supposons par l'absurde qu'il existe une solution $(\overline{x},\overline{k},\ldots,\overline{k},\overline{y})$ de $(E_{N})$ de taille $3 \leq l' \leq l-1$. Il existe $\epsilon \in \{-1,1\}$ tel que $M_{l'}(\overline{x},\overline{k},\ldots,\overline{k},\overline{y})=\overline{\epsilon} Id$. Pour tout $i$ dans $[\![1;r]\!]$, il existe $\epsilon_{i} \in \{-1,1\}$ tel que $M_{l_{i}}(x+p_{i}^{\alpha_{i}}\mathbb{Z},k+p_{i}^{\alpha_{i}}\mathbb{Z},\ldots,k+p_{i}^{\alpha_{i}}\mathbb{Z},y+p_{i}^{\alpha_{i}}\mathbb{Z})=(\epsilon_{i}+p_{i}^{\alpha_{i}}\mathbb{Z}) Id$. Par le lemme \ref{35}, $\epsilon_{1}=1$ (car $l_{1}$ est pair et $\overline{k} \neq \overline{0}$ puisque $l_{1} \neq 2$).
\\
\\Soit $i$ dans $[\![1;r]\!]$. Par ce qui précède, on a $l_{i}=p_{i}^{\alpha_{i}}$. Ainsi, la solution $(k+p_{i}^{\alpha_{i}}\mathbb{Z})$-monomiale minimale de $(E_{p_{i}^{\alpha_{i}}})$ est irréductible (Théorème \ref{311} et proposition \ref{311b}). Par le lemme \ref{plus}, on a donc $l' \equiv 0 [p_{i}^{\alpha_{i}}]$ et $x+p_{i}^{\alpha_{i}}\mathbb{Z}=y+p_{i}^{\alpha_{i}}\mathbb{Z}=k+p_{i}^{\alpha_{i}}\mathbb{Z}$ ou $l' \equiv 2 [p_{i}^{\alpha_{i}}]$ et $x+p_{i}^{\alpha_{i}}\mathbb{Z}=y+p_{i}^{\alpha_{i}}\mathbb{Z}=0+p_{i}^{\alpha_{i}}\mathbb{Z}$. En particulier, $l' \equiv 0,2 [2^{\alpha_{1}}]$ et donc $l'$ est pair.
\\
\\Notons $I=\{i \in [\![1;r]\!],~l' \equiv 0 [p_{i}^{\alpha_{i}}]\}$ et $J=\{j \in [\![1;r]\!],~l' \equiv 2 [p_{j}^{\alpha_{j}}]\}$. On a $I \cup J=[\![1;r]\!]$ et, comme pour tout $i$ $p_{i}^{\alpha_{i}} \neq 2$, on a $I \cap J=\emptyset$. De plus, $I, J \neq \emptyset$, car sinon on aurait, par le lemme chinois, $l' \equiv 0,2 [N]$ ce qui est impossible. On va distinguer deux cas :
\\
\\i) On suppose que $1 \in I$. Soit $j \in J$. Comme $\frac{l'-2}{l_{j}}$ est pair, on a :
\[M_{l'}(x+p_{1}^{\alpha_{1}}\mathbb{Z},k+p_{1}^{\alpha_{1}}\mathbb{Z},\ldots,k+p_{1}^{\alpha_{1}}\mathbb{Z},y+p_{1}^{\alpha_{1}}\mathbb{Z})=(M_{l_{1}}(k+p_{1}^{\alpha_{1}}\mathbb{Z},\ldots,k+p_{1}^{\alpha_{1}}\mathbb{Z}))^{\frac{l'}{l_{1}}}=(1+p_{1}^{\alpha_{1}}\mathbb{Z}) Id.\]
\begin{eqnarray*}
M &=& M_{l'}(x+p_{j}^{\alpha_{j}}\mathbb{Z},k+p_{j}^{\alpha_{j}}\mathbb{Z},\ldots,k+p_{j}^{\alpha_{j}}\mathbb{Z},y+p_{j}^{\alpha_{j}}\mathbb{Z}) \\
  &=& M_{1}(0+p_{j}^{\alpha_{j}}\mathbb{Z})(M_{l_{j}}(k+p_{j}^{\alpha_{j}}\mathbb{Z},\ldots,k+p_{j}^{\alpha_{j}}\mathbb{Z}))^{\frac{l'-2}{l_{j}}}M_{1}(0+p_{j}^{\alpha_{j}}\mathbb{Z}) \\
	&=& M_{1}(0+p_{j}^{\alpha_{j}}\mathbb{Z})((\epsilon_{j}+p_{j}^{\alpha_{j}}\mathbb{Z})^{\frac{l'-2}{l_{j}}}Id)M_{1}(0+p_{j}^{\alpha_{j}}\mathbb{Z}) \\
	&=& M_{2}(0+p_{j}^{\alpha_{j}}\mathbb{Z},0+p_{j}^{\alpha_{j}}\mathbb{Z}) \\
	&=& (-1+p_{j}^{\alpha_{j}}\mathbb{Z}) Id.
\end{eqnarray*}	

\noindent Si $\epsilon=-1$ alors $\epsilon+p_{1}^{\alpha_{1}}\mathbb{Z}=-1+p_{1}^{\alpha_{1}}\mathbb{Z}$. Or, par ce qui précède, $\epsilon+p_{1}^{\alpha_{1}}\mathbb{Z}=1+p_{1}^{\alpha_{1}}\mathbb{Z} \neq -1+p_{1}^{\alpha_{1}}\mathbb{Z}$ (puisque $p_{1}^{\alpha_{1}} \neq 2$). Si $\epsilon=1$ alors $\epsilon+p_{j}^{\alpha_{j}}\mathbb{Z}=1+p_{j}^{\alpha_{j}}\mathbb{Z}$. Or, par ce qui précède, $\epsilon+p_{j}^{\alpha_{j}}\mathbb{Z}=-1+p_{j}^{\alpha_{j}}\mathbb{Z} \neq 1+p_{j}^{\alpha_{j}}\mathbb{Z}$ (puisque $p_{j}^{\alpha_{j}} \neq 2$).
\\
\\ii) On suppose que $1 \in J$. Soit $i \in I$. Comme $\frac{l'}{l_{i}}$ est pair, on a :
\[M_{l'}(x+p_{i}^{\alpha_{i}}\mathbb{Z},k+p_{i}^{\alpha_{i}}\mathbb{Z},\ldots,k+p_{i}^{\alpha_{i}}\mathbb{Z},y+p_{i}^{\alpha_{i}}\mathbb{Z})=(M_{l_{i}}(k+p_{i}^{\alpha_{i}}\mathbb{Z},\ldots,k+p_{i}^{\alpha_{i}}\mathbb{Z}))^{\frac{l'}{l_{i}}}=(1+p_{i}^{\alpha_{i}}\mathbb{Z}) Id.\]
\begin{eqnarray*}
M &=& M_{l'}(x+p_{1}^{\alpha_{1}}\mathbb{Z},k+p_{1}^{\alpha_{1}}\mathbb{Z},\ldots,k+p_{1}^{\alpha_{1}}\mathbb{Z},y+p_{1}^{\alpha_{1}}\mathbb{Z}) \\
  &=& M_{1}(0+p_{1}^{\alpha_{1}}\mathbb{Z})(M_{l_{1}}(k+p_{1}^{\alpha_{1}}\mathbb{Z},\ldots,k+p_{1}^{\alpha_{1}}\mathbb{Z}))^{\frac{l'-2}{l_{1}}}M_{1}(0+p_{1}^{\alpha_{1}}\mathbb{Z}) \\
	&=& M_{2}(0+p_{1}^{\alpha_{1}}\mathbb{Z},0+p_{1}^{\alpha_{1}}\mathbb{Z}) \\
	&=& (-1+p_{1}^{\alpha_{1}}\mathbb{Z}) Id.
\end{eqnarray*}	

\noindent Si $\epsilon=1$ alors $\epsilon+p_{1}^{\alpha_{1}}\mathbb{Z}=1+p_{1}^{\alpha_{1}}\mathbb{Z}$. Or, par ce qui précède, $\epsilon+p_{1}^{\alpha_{1}}\mathbb{Z}=-1+p_{1}^{\alpha_{1}}\mathbb{Z} \neq 1+p_{1}^{\alpha_{1}}\mathbb{Z}$ (puisque $p_{1}^{\alpha_{1}} \neq 2$). Si $\epsilon=-1$ alors $\epsilon+p_{i}^{\alpha_{i}}\mathbb{Z}=-1+p_{i}^{\alpha_{i}}\mathbb{Z}$. Or, par ce qui précède, $\epsilon+p_{i}^{\alpha_{i}}\mathbb{Z}=1+p_{i}^{\alpha_{i}}\mathbb{Z} \neq -1+p_{i}^{\alpha_{i}}\mathbb{Z}$ (puisque $p_{i}^{\alpha_{i}} \neq 2$).
\\
\\Ainsi, on arrive à une absurdité dans les deux cas. Donc, la solution $\overline{k}$-monomiale minimale de $(E_{N})$ est irréductible.
\\
\\ \uwave{$4^{{\rm \grave{e}me}}$ étape :} On considère le cas $l_{1}=2$.
\\
\\Par la proposition \ref{37}, $l={\rm ppcm}(l_{i},~1 \leq i \leq r)$ ou $l=2\times {\rm ppcm}(l_{i},~1 \leq i \leq r)$. Supposons par l'absurde que $l={\rm ppcm}(l_{i},~1 \leq i \leq r)$. On a
\[l={\rm ppcm}(l_{i},~1 \leq i \leq r) \leq 2 \prod_{i=2}^{r} l_{i} \leq 2 \prod_{i=2}^{r} p_{i}^{\alpha_{i}}=\frac{N}{2^{\alpha_{1}-1}}<N.\]
\noindent Ainsi, $l=2\times {\rm ppcm}(l_{i},~1 \leq i \leq r)$. S'il existe $j$ dans $[\![2;r]\!]$ tel que $l_{j}<p_{j}^{\alpha_{j}}$ alors
\[l \leq 2 \times 2 \times \prod_{i=2}^{r} l_{i} \leq 4 \times l_{j} \times \prod_{i\neq 1,j} p_{i}^{\alpha_{i}} \leq 2^{\alpha_{1}} l_{j} \prod_{i\neq 1,j} p_{i}^{\alpha_{i}}<\prod_{i=1}^{r} p_{i}^{\alpha_{i}}=N.\]
\noindent Ainsi, $l_{1}=2$ et pour tout $i$ dans $[\![2;r]\!]$ $l_{i}=p_{i}^{\alpha_{i}}$ (notons en particulier que $\alpha_{1}=2$). Supposons par l'absurde qu'il existe une solution $(\overline{x},\overline{k},\ldots,\overline{k},\overline{y})$ de $(E_{N})$ de taille $3 \leq l' \leq l-1$. En procédant comme dans l'étape précédente, on a pour tout $i$ dans $[\![2;r]\!]$ $l_{i} \equiv 0,2 [p_{i}^{\alpha_{i}}]$ et par le lemme \ref{plus} $l' \equiv 0 [l_{1}]$, c'est-à-dire $l'$ est pair. Si $r=2$ on a nécessairement $l' \equiv 0,2 [l]$ ce qui est absurde. Donc, $r \geq 3$. On procède comme dans la première étape en définissant les ensembles $I$ et $J$. Comme $l' \not\equiv 0,2 [l]$, $I,J \neq \{1\}$. On choisit alors $u$ dans $I$ et $v$ dans $J$ différents de 1 et on procède comme dans le cas i) de l'étape précédente ($\frac{l'}{l_{u}}$ et $\frac{l'-2}{l_{v}}$ sont pairs).

\end{proof}

\noindent On va maintenant considérer le dernier cas nécessaire à la démonstration du point i) du théorème.

\begin{proposition}
\label{322}

Soit $N=2m$ avec $m \neq 1$ impair non divisible par 3 ou $m=3^{a}$ avec $a \geq 1$. Toutes les solutions monomiales minimales de $(E_{N})$ de taille $N$ sont irréductibles.

\end{proposition}

\begin{proof}

Soit $N=p_{1}^{\alpha_{1}}\ldots p_{r}^{\alpha_{r}}$ ($r \geq 2$). Soit $k$ un entier naturel tel que la solution $\overline{k}$-monomiale minimale de $(E_{N})$ est de taille $N$ (notons que, par le théorème \ref{30}, un tel $k$ existe toujours). Pour tout $i$ dans $[\![1;r]\!]$, on note $l_{i}$ la taille de la solution $(k+p_{i}^{\alpha_{i}}\mathbb{Z})$-monomiale minimale de $(E_{p_{i}^{\alpha_{i}}})$ et $l$ la taille de la solution $\overline{k}$-monomiale minimale de $(E_{N})$. Par hypothèse, $l=N$. Quitte à modifier l'ordre des facteurs premiers, on suppose que $p_{1}=2$. Par hypothèse, $\alpha_{1}=1$. On procède par étapes.
\\
\\\uwave{$1^{{\rm \grave{e}re}}$ étape :} On montre que pour tout $i$, $l_{i}=p_{i}^{\alpha_{i}}$.
\\
\\Supposons que $p_{i} \neq 3$ pour tout $i$. On a $l={\rm ppcm}(l_{i},~1 \leq i \leq r)$ ou $l=2 \times {\rm ppcm}(l_{i},~1 \leq i \leq r)$ (proposition \ref{37}). On a $l_{1}=3$ ou $l_{1}=2$. Si $l_{1}=3$ alors nécessairement 3 divise $l=N$, ce qui contredit l'hypothèse de départ. Donc, $l_{1}=2$. Supposons par l'absurde que $l=2 \times {\rm ppcm}(l_{i},~1 \leq i \leq r)$. 2 divise ${\rm ppcm}(l_{i},~1 \leq i \leq r)$ et donc 4 divise $l=N$, ce qui est absurde. Ainsi, $l={\rm ppcm}(l_{i},~1 \leq i \leq r)$. 
\\
\\S'il existe un $j$ dans $[\![1;r]\!]$ tel que $l_{j}<p_{j}^{\alpha_{j}}$ alors $l\leq \prod_{i=1}^{r} l_{i} < p_{j}^{\alpha_{j}} \prod_{i \neq j} p_{i}^{\alpha_{i}} =N.$ Donc, pour tout $i$, $l_{i}=p_{i}^{\alpha_{i}}$.
\\
\\Supposons que $r=2$ et $p_{2}=3$. Si $\alpha_{2}=1$ alors $N=6$ et la proposition est vraie. On suppose donc $\alpha_{2} \geq 2$. Par la proposition \ref{38}, on a $l={\rm ppcm}(l_{1},l_{2})$.  Comme $3^{\alpha_{2}}$ divise $l$, $3^{\alpha_{2}}$ divise $l_{1}$ ou $l_{2}$. Puisque $\alpha_{2} \geq 2$ et $l_{1} \leq 3$, on a nécessairement que $3^{\alpha_{2}}$ divise $l_{2}$. Or, $l_{2} \leq 3^{\alpha_{2}}$ (lemme \ref{36}). Donc, $l_{2}=3^{\alpha_{2}}$ et nécessairement $l_{1}=2$.
\\
\\\uwave{$2^{{\rm \grave{e}me}}$ étape :} On montre que la solution $\overline{k}$-monomiale minimale de $(E_{N})$ est irréductible.
\\
\\Supposons par l'absurde qu'il existe une solution $(\overline{x},\overline{k},\ldots,\overline{k},\overline{y})$ de $(E_{N})$ de taille $3 \leq l' \leq l-1$. Il existe $\epsilon \in \{-1,1\}$ tel que $M_{l'}(\overline{x},\overline{k},\ldots,\overline{k},\overline{y})=\overline{\epsilon} Id$. Pour tout $i$ dans $[\![1;r]\!]$, il existe $\epsilon_{i} \in \{-1,1\}$ tel que $M_{l_{i}}(x+p_{i}^{\alpha_{i}}\mathbb{Z},k+p_{i}^{\alpha_{i}}\mathbb{Z},\ldots,k+p_{i}^{\alpha_{i}}\mathbb{Z},y+p_{i}^{\alpha_{i}}\mathbb{Z})=(\epsilon_{i}+p_{i}^{\alpha_{i}}\mathbb{Z}) Id$.
\\
\\Soit $i$ dans $[\![2;r]\!]$. Par ce qui précède, on a $l_{i}=p_{i}^{\alpha_{i}}$. Ainsi, la solution $(k+p_{i}^{\alpha_{i}}\mathbb{Z})$-monomiale minimale de $(E_{p_{i}^{\alpha_{i}}})$ est irréductible (Théorème \ref{318}). Par le lemme \ref{plus}, on a donc $l' \equiv 0 [p_{i}^{\alpha_{i}}]$ et $x+p_{i}^{\alpha_{i}}\mathbb{Z}=y+p_{i}^{\alpha_{i}}\mathbb{Z}=k+p_{i}^{\alpha_{i}}\mathbb{Z}$ ou $l' \equiv 2 [p_{i}^{\alpha_{i}}]$ et $x+p_{i}^{\alpha_{i}}\mathbb{Z}=y+p_{i}^{\alpha_{i}}\mathbb{Z}=0+p_{i}^{\alpha_{i}}\mathbb{Z}$. De plus, $k+p_{1}^{\alpha_{1}}\mathbb{Z}=0+2\mathbb{Z}$ et donc $l' \equiv 0 [2]$ (lemme \ref{plus}), c'est-à-dire $l'$ est pair.
\\
\\Si $r=2$ alors on a deux cas possibles :
\begin{itemize}
\item $l_{1} \equiv 0 [2]$ et $l_{2} \equiv 0 [p_{2}^{\alpha_{2}}]$ et donc $l' \equiv 0 [N]$ ce qui est impossible puisque $3 \leq l' \leq N-1$;
\item $l_{1} \equiv 0 [2]$ et $l_{2} \equiv 2 [p_{2}^{\alpha_{2}}]$ et donc $l' \equiv 2 [N]$ ce qui est impossible puisque $3 \leq l' \leq N-1$.
\end{itemize}

\noindent Donc, $r \geq 3$.
\\
\\Notons $I=\{i \in [\![2;r]\!],~l' \equiv 0 [p_{i}^{\alpha_{i}}]\}$ et $J=\{j \in [\![2;r]\!],~l' \equiv 2 [p_{j}^{\alpha_{j}}]\}$. On a $I \cup J=[\![2;r]\!]$ et, comme pour tout $i$ $p_{i}^{\alpha_{i}} \neq 2$, on a $I \cap J=\emptyset$. De plus, $I, J \neq \emptyset$, car sinon on aurait, par le lemme chinois, $l' \equiv 0,2 [N]$ ce qui est impossible.
\\
\\Soient $i \in I$ et $j \in J$. Comme $l_{i}$ est impair et $l'$ est pair, $\frac{l'}{l_{i}}$ est pair et on a :
\begin{eqnarray*}
M &=& M_{l'}(x+p_{i}^{\alpha_{i}}\mathbb{Z},k+p_{i}^{\alpha_{i}}\mathbb{Z},\ldots,k+p_{i}^{\alpha_{i}}\mathbb{Z},y+p_{i}^{\alpha_{i}}\mathbb{Z}) \\
  &=& (M_{l_{i}}(k+p_{i}^{\alpha_{i}}\mathbb{Z},\ldots,k+p_{i}^{\alpha_{i}}\mathbb{Z}))^{\frac{l'}{l_{i}}} \\
	&=& (\epsilon_{i}+p_{i}^{\alpha_{i}}\mathbb{Z})^{\frac{l'}{l_{i}}} Id \\
	&=& (1+p_{i}^{\alpha_{i}}\mathbb{Z}) Id.
\end{eqnarray*}

\noindent Comme $l_{j}$ est impair et $l'$ est pair, $\frac{l'-2}{l_{j}}$ est pair et on a :
\begin{eqnarray*}
N &=& M_{l'}(x+p_{j}^{\alpha_{j}}\mathbb{Z},k+p_{j}^{\alpha_{j}}\mathbb{Z},\ldots,k+p_{j}^{\alpha_{j}}\mathbb{Z},y+p_{j}^{\alpha_{j}}\mathbb{Z}) \\
  &=& M_{1}(0+p_{j}^{\alpha_{j}}\mathbb{Z})(M_{l_{j}}(k+p_{j}^{\alpha_{j}}\mathbb{Z},\ldots,k+p_{j}^{\alpha_{j}}\mathbb{Z}))^{\frac{l'-2}{l_{j}}}M_{1}(0+p_{j}^{\alpha_{j}}\mathbb{Z}) \\
	&=& M_{1}(0+p_{j}^{\alpha_{j}}\mathbb{Z})((\epsilon_{j}+p_{j}^{\alpha_{j}}\mathbb{Z})^{\frac{l'-2}{l_{j}}}Id)M_{1}(0+p_{j}^{\alpha_{j}}\mathbb{Z}) \\
	&=& M_{2}(0+p_{j}^{\alpha_{j}}\mathbb{Z},0+p_{j}^{\alpha_{j}}\mathbb{Z}) \\
	&=& (-1+p_{j}^{\alpha_{j}}\mathbb{Z}) Id.
\end{eqnarray*}	

\noindent Si $\epsilon=-1$ alors $\epsilon+p_{i}^{\alpha_{i}}\mathbb{Z}=-1+p_{i}^{\alpha_{i}}\mathbb{Z}$. Or, par ce qui précède, $\epsilon+p_{i}^{\alpha_{i}}\mathbb{Z}=1+p_{i}^{\alpha_{i}}\mathbb{Z} \neq -1+p_{i}^{\alpha_{i}}\mathbb{Z}$ (puisque $p_{i}^{\alpha_{i}} \neq 2$). Si $\epsilon=1$ alors $\epsilon+p_{j}^{\alpha_{j}}\mathbb{Z}=1+p_{j}^{\alpha_{j}}\mathbb{Z}$. Or, par ce qui précède, $\epsilon+p_{j}^{\alpha_{j}}\mathbb{Z}=-1+p_{j}^{\alpha_{j}}\mathbb{Z} \neq 1+p_{j}^{\alpha_{j}}\mathbb{Z}$ (puisque $p_{j}^{\alpha_{j}} \neq 2$).
\\
\\Ainsi, on arrive à une absurdité dans les deux cas. Donc, la solution $\overline{k}$-monomiale minimale de $(E_{N})$ est irréductible.

\end{proof}

\noindent On va maintenant démontrer le point ii) du théorème \ref{320}. 

\begin{proposition}
\label{323}

Si $N=2 \times 3^{a} \times m$, avec $a \geq 1$ et $m>1$ impair non divisible par 3, alors il existe des solutions monomiales minimales non nulles de $(E_{N})$ de taille $N$ réductibles.

\end{proposition}

\begin{proof}

Soit $k$ l'unique entier de $[\![0;N-1]\!]$ vérifiant $k \equiv 1 [2]$, $k \equiv 2 [3^{a}]$ et $k \equiv -2 [m]$. Par la proposition \ref{37}, le lemme \ref{39} ii) et le théorème \ref{30}, la taille de la solution $\overline{k}$-monoiale minimale de $(E_{N})$ est égale à $2\times {\rm ppcm}(3,3^{a},m)=2\times 3^{a}m=N$.
\\
\\Comme $m$ n'est pas divisible par 3, $3^{a}$ est inversible modulo $m$. Donc, il existe un unique $h$ dans $[\![0;m-1]\!]$ tel que $3^{a}h \equiv 2 [m]$. $h \neq 0$ car $m>2$. Si $h$ est impair, on pose $s=3^{a}h$. Si $h$ est pair, on note $s=3^{a}h+3^{a}m$. De plus, $3 \leq s \leq 3^{a}(m-1)+3^{a}m=2\times 3^{a}m-3^{a}<N$.
\\
\\On a $s \equiv 0 [3]$ et donc :
\[M_{s}(k+2\mathbb{Z},\ldots,k+2\mathbb{Z})=(M_{3}(k+2\mathbb{Z},\ldots,k+2\mathbb{Z}))^{\frac{s}{3}}=(1+2\mathbb{Z})Id.\]
\noindent Comme $s \equiv 0 [3^{a}]$, on a :
\[M_{s}(k+3^{a}\mathbb{Z},\ldots,k+3^{a}\mathbb{Z})=(M_{3^{a}}(k+3^{a}\mathbb{Z},\ldots,k+3^{a}\mathbb{Z})^{\frac{s}{3^{a}}}=(1+3^{a}\mathbb{Z})^{\frac{s}{3^{a}}}Id=(1+3^{a}\mathbb{Z})Id.\]
\noindent $s$ est impair et $s \equiv 2 [m]$. Donc, $\frac{s-2}{m}$ est impair et on a :
\begin{eqnarray*}
M &=& M_{s}(0+m\mathbb{Z},k+m\mathbb{Z},\ldots,k+m\mathbb{Z},0+m\mathbb{Z}) \\
  &=& M_{1}(0+m\mathbb{Z})M_{m}(k+m\mathbb{Z},\ldots,k+m\mathbb{Z})^{\frac{s-2}{m}}M_{1}(0+m\mathbb{Z}) \\
	&=& M_{1}(0+m\mathbb{Z})((-1+m\mathbb{Z})^{\frac{s-2}{m}}Id)M_{1}(0+m\mathbb{Z}) \\
	&=& -M_{2}(0+m\mathbb{Z},0+m\mathbb{Z}) \\
	&=& (1+m\mathbb{Z})Id.
\end{eqnarray*}

\noindent Notons $a$ l'unique entier compris entre 0 et $N-1$ tel que $a \equiv k [2]$, $a \equiv k [3^{a}]$ et $a \equiv 0 [m]$. Par le lemme chinois, on a $M_{s}(\overline{a},\overline{k},\ldots,\overline{k},\overline{a})=Id$. Comme $3 \leq s \leq N-1$, on peut utiliser cette solution pour réduire la solution $\overline{k}$-monomiale minimale de $(E_{N})$. Cette dernière est donc réductible.

\end{proof}

En combinant les propositions \ref{321}, \ref{322} et \ref{323}, on obtient le théorème \ref{320}. Notons également que la proposition \ref{323} et le théorème \ref{30} nous montrent qu'il est possible pour un $N$ et un $l$ fixé d'avoir certaines solutions monomiales minimales de $(E_{N})$ de taille $l$ réductibles et d'autres irréductibles.

\begin{examples}
{\rm On donne ci-dessous quelques applications immédiates des résultats qui précèdent.
\begin{itemize}
\item $N=245=5 \times 7^{2}$. La solution $\overline{37}$-monomiale minimale de $(E_{N})$ est de taille $245=N$ et est donc irréductible.
\item $N=70=2 \times 5 \times 7$. La solution $\overline{12}$-monomiale minimale de $(E_{N})$ est de taille $70=N$ et est donc irréductible.
\item $N=1100=2^{2} \times 5^{2} \times 11$. La solution $\overline{152}$-monomiale minimale de $(E_{N})$ est de taille $1100=N$ et est donc irréductible.
\item $N=90=2 \times 3^{2} \times 5$. La solution $\overline{-7}$-monomiale minimale de $(E_{N})$ est de taille $90=N$. On peut la réduire avec la solution $(\overline{65},\overline{-7},\ldots,\overline{-7},\overline{65}) \in (\mathbb{Z}/N\mathbb{Z})^{27}$.
\end{itemize}
}
\end{examples}

\subsection{Quelques résultats supplémentaires}

On considère dans cette dernière partie un certain nombre d'autres valeurs de taille, dépendantes de $N$, pour lesquels l'irréductibilité des solutions monomiales minimales est assurée.

\begin{proposition}
\label{324}

Soient $N>2$, $p$ un nombre premier et $n \in \mathbb{N}^{*}$. On suppose que $p^{n} \neq 2$. 
\\
\\i) Si $N$ est impair ou si $p$ est impair alors toutes les solutions monomiales minimales non nulles de $(E_{N})$ de taille $p^{n}$ sont irréductibles.
\\
\\ii) On suppose que $N=2^{\alpha_{1}}m$, avec $\alpha_{1} \geq 1$ et $m$ impair, et on suppose que $n=2$ ou que $n \geq \alpha_{1}$. Toutes les solutions monomiales minimales non nulles de $(E_{N})$ de taille $2^{n}$ sont irréductibles.

\end{proposition}

\begin{proof}

Soit $N=p_{1}^{\alpha_{1}} \ldots p_{r}^{\alpha_{r}}$. Supposons qu'il existe un entier naturel $k$ tel que la solution $\overline{k}$-monomiale minimale de $(E_{N})$ est de taille $p^{n}$. Pour tout $i$ dans $[\![1;r]\!]$, on note $l_{i}$ la taille de la solution $(k+p_{i}^{\alpha_{i}}\mathbb{Z})$-monomiale minimale de $(E_{p_{i}^{\alpha_{i}}})$ et $l$ la taille de la solution $\overline{k}$-monomiale minimale de $(E_{N})$. Par hypothèse, $l=p^{n}$. 
\\
\\Si $p$ est impair alors le résultat découle du théorème \ref{318} (si $p=3$ alors soit $l=3$ soit 9 divise $l$). On suppose donc maintenant que $p$ est pair. Si $l=4$ alors la solution $\overline{k}$-monomiale minimale de $(E_{N})$ est irréductible, puisqu'une solution de taille 4 est réductible si et seulement si elle contient $\pm \overline{1}$ et que les solutions $\pm \overline{1}$-monomiale minimale de $(E_{N})$ sont de taille 3. On suppose donc maintenant que $l \geq 8$.
\\
\\Par le lemme \ref{37}, on a $l=2^{n}={\rm ppcm}(l_{i},~1 \leq i \leq r)$ ou $l=2^{n}=2 \times {\rm ppcm}(l_{i},~1 \leq i \leq r)$. Ainsi, ${\rm ppcm}(l_{i},~1 \leq i \leq r)$ est une puissance de 2. Donc, pour tout $1 \leq i \leq r$, $l_{i}$ est une puissance de deux.
\\
\\Supposons par l'absurde que la solution $\overline{k}$-monomiale minimale de $(E_{N})$ est réductible. Il existe une solution $(\overline{x},\overline{k},\ldots,\overline{k},\overline{y})$ de $(E_{N})$ de taille $3 \leq l' \leq l-1$. Il existe $\epsilon \in \{-1,1\}$ tel que $M_{l'}(\overline{x},\overline{k},\ldots,\overline{k},\overline{y})=\overline{\epsilon} Id$. On distingue deux cas.
\\
\\ \uwave{i) $N$ impair :} Pour tout $i$, on a $M_{l_{i}}(k+p_{i}^{\alpha_{i}}\mathbb{Z},\ldots,k+p_{i}^{\alpha_{i}}\mathbb{Z})=(-1+p_{i}^{\alpha_{i}}\mathbb{Z}) Id$, puisque $p_{i}$ est impair et $l_{i}$ pair (lemme \ref{34}). On considère deux cas.
\\
\\A) Si $l={\rm ppcm}(l_{i},~1 \leq i \leq r)$. Il existe nécessairement un $i$ tel que $l_{i}=2^{n}$. Comme $p_{i}^{\alpha_{i}}$ divise $N$, $(x+p_{i}^{\alpha_{i}}\mathbb{Z},k+p_{i}^{\alpha_{i}}\mathbb{Z},\ldots,k+p_{i}^{\alpha_{i}}\mathbb{Z},y+p_{i}^{\alpha_{i}}\mathbb{Z})$ est une solution de $(E_{p_{i}}^{\alpha_{i}})$ de taille $l'$. Donc, la solution $(k+p_{i}^{\alpha_{i}}\mathbb{Z})$-monomiale minimale de $(E_{p_{i}}^{\alpha_{i}})$ est réductible. Or, ceci est absurde car les solutions monomiales minimales réductibles de $(E_{p_{i}^{\alpha_{i}}})$ sont de taille paire non divisible par 4 (théorème \ref{311b} et proposition \ref{311}) et $l$ est divisible par 8.
\\
\\B) Si $l=2 \times {\rm ppcm}(l_{i},~1 \leq i \leq r)$. On a $r \geq 2$ et il existe un $i$ tel que $l_{i}=2^{n-1}$. De plus, il existe nécessairement un $j$ tel que $l_{j}=2^{h}<l_{i}$. En effet, si on avait, pour tout $1 \leq j \leq r$, $l_{j}=2^{n-1}$ alors on aurait $l={\rm ppcm}(l_{i},~1 \leq i \leq r)=2^{n-1}$.
\\
\\Comme $p_{i}^{\alpha_{i}}$ divise $N$, $(x+p_{i}^{\alpha_{i}}\mathbb{Z},k+p_{i}^{\alpha_{i}}\mathbb{Z},\ldots,k+p_{i}^{\alpha_{i}}\mathbb{Z},y+p_{i}^{\alpha_{i}}\mathbb{Z})$ est une solution de $(E_{p_{i}^{\alpha_{i}}})$ de taille $l'$. Or, comme $l_{i}$ est divisible par 4, la solution $(k+p_{i}^{\alpha_{i}}\mathbb{Z})$-monomiale minimale de $(E_{p_{i}^{\alpha_{i}}})$ est irréductible. Donc, par le lemme \ref{plus}, $l' \equiv 0 [l_{i}]$ avec $x+p_{i}^{\alpha_{i}}\mathbb{Z}=y+p_{i}^{\alpha_{i}}\mathbb{Z}=k+p_{i}^{\alpha_{i}}\mathbb{Z}$ ou $l' \equiv 2 [l_{i}]$ avec $x+p_{i}^{\alpha_{i}}\mathbb{Z}=y+p_{i}^{\alpha_{i}}\mathbb{Z}=0+p_{i}^{\alpha_{i}}\mathbb{Z}$. Ainsi, comme $l_{i}=\frac{l}{2}$, la solution $\overline{k}$-monomiale minimale de $(E_{N})$ est nécessairement égale à la somme d'une solution de taille $l_{i}$ avec une solution de taille $l_{i}+2$ ou la somme d'une solution de taille $l_{i}+2$ avec une solution de taille $l_{i}$. Il existe donc une solution $(\overline{x},\overline{k},\ldots,\overline{k},\overline{y})$ de $(E_{N})$ de taille $l'=l_{i}=2^{n-1}$.
\\
\\Comme $p_{j}^{\alpha_{j}}$ divise $N$, $(x+p_{j}^{\alpha_{j}}\mathbb{Z},k+p_{j}^{\alpha_{j}}\mathbb{Z},\ldots,k+p_{j}^{\alpha_{j}}\mathbb{Z},y+p_{j}^{\alpha_{j}}\mathbb{Z})$ est une solution de $(E_{p_{j}^{\alpha_{j}}})$. De plus, $l' \equiv 0 [l_{j}]$. Donc, par le lemme \ref{plus}, $x+p_{j}^{\alpha_{j}}\mathbb{Z}=y+p_{j}^{\alpha_{j}}\mathbb{Z}=k+p_{j}^{\alpha_{j}}\mathbb{Z}$. On a donc :
\[M_{l'}(x+p_{i}^{\alpha_{i}}\mathbb{Z},k+p_{i}^{\alpha_{i}}\mathbb{Z},\ldots,k+p_{i}^{\alpha_{i}}\mathbb{Z},y+p_{i}^{\alpha_{i}}\mathbb{Z})=M_{l_{i}}(k+p_{i}^{\alpha_{i}}\mathbb{Z},\ldots,k+p_{i}^{\alpha_{i}}\mathbb{Z})=(-1+p_{i}^{\alpha_{i}}\mathbb{Z});\]
\[M_{l'}(x+p_{j}^{\alpha_{j}}\mathbb{Z},k+p_{j}^{\alpha_{j}}\mathbb{Z},\ldots,k+p_{j}^{\alpha_{j}}\mathbb{Z},y+p_{j}^{\alpha_{j}}\mathbb{Z})=M_{l_{j}}(k+p_{j}^{\alpha_{j}}\mathbb{Z},\ldots,k+p_{j}^{\alpha_{j}}\mathbb{Z})^{2^{n-1-h}}=(1+p_{j}^{\alpha_{j}}\mathbb{Z}).\]

\noindent Si $\epsilon=1$ alors $\epsilon+p_{i}^{\alpha_{i}}\mathbb{Z}=1+p_{i}^{\alpha_{i}}\mathbb{Z}$. Or, par ce qui précède, $\epsilon+p_{i}^{\alpha_{i}}\mathbb{Z}=-1+p_{i}^{\alpha_{i}}\mathbb{Z} \neq 1+p_{i}^{\alpha_{i}}\mathbb{Z}$ (puisque $p_{i}^{\alpha_{i}} \neq 2$). Si $\epsilon=-1$ alors $\epsilon+p_{j}^{\alpha_{j}}\mathbb{Z}=-1+p_{j}^{\alpha_{j}}\mathbb{Z}$. Or, par ce qui précède, $\epsilon+p_{j}^{\alpha_{j}}\mathbb{Z}=1+p_{j}^{\alpha_{j}}\mathbb{Z} \neq -1+p_{j}^{\alpha_{j}}\mathbb{Z}$ (puisque $p_{j}^{\alpha_{j}} \neq 2$).
\\
\\On arrive donc à une absurdité dans les deux cas. Ainsi, la solution $\overline{k}$-monomiale minimale de $(E_{N})$ est irréductible.
\\
\\ \uwave{ii) $N$ pair :} Quitte à modifier l'ordre des $p_{i}$, on suppose que $p_{1}=2$. Si $r=1$ alors le résultat est vrai puisque les solutions monomiales minimales réductible de $(E_{2^{\alpha_{1}}})$ sont de taille $2^{j}$ avec $j=1$ ou $3 \leq j \leq \alpha_{1}-1$ (théorème \ref{311b} et proposition \ref{311}). On suppose maintenant $r \geq 2$. On suppose également $n \geq \alpha_{1}$.
\\
\\On suppose pour commencer que $\alpha_{1}=1$. On a $l_{1}=2$. On note $s$ la taille de la solution $(k+\frac{N}{2})$-monomiale minimale de $(E_{\frac{N}{2}})$. Par la proposition \ref{38}, $2^{n}={\rm ppcm}(2,s)$. Ainsi, $s=2^{n}$. Comme $m=\frac{N}{2}$ divise $N$, $(x+m\mathbb{Z},k+m\mathbb{Z},\ldots,k+m\mathbb{Z},y+m\mathbb{Z})$ est une solution de $(E_{m})$ de taille $l'$. Donc, la solution $(k+m\mathbb{Z})$-monomiale minimale de $(E_{m})$ est réductible. Or, ceci est absurde par le cas A ($m$ impair). Donc, la solution $\overline{k}$-monomiale minimale de $(E_{N})$ est irréductible.
\\
\\On suppose donc maintenant $\alpha_{1}>1$. On considère deux cas.
\\
\\i) Si $l={\rm ppcm}(l_{i},~1 \leq i \leq r)$. Il existe nécessairement un $i$ tel que $l_{i}=2^{n}$. Comme $p_{i}^{\alpha_{i}}$ divise $N$, $(x+p_{i}^{\alpha_{i}}\mathbb{Z},k+p_{i}^{\alpha_{i}}\mathbb{Z},\ldots,k+p_{i}^{\alpha_{i}}\mathbb{Z},y+p_{i}^{\alpha_{i}}\mathbb{Z})$ est une solution de $(E_{p_{i}^{\alpha_{i}}})$ de taille $l'$. Donc, la solution $(k+p_{i}^{\alpha_{i}}\mathbb{Z})$-monomiale minimale de $(E_{p_{i}^{\alpha_{i}}})$ est réductible. Or, ceci est absurde car, si $i \geq 2$, une solution monomiale minimale de $(E_{p_{i}^{\alpha_{i}}})$ de taille $2^{n}$ est irréductible (puisque $l_{i}$ est divisible par 4, théorème \ref{311b} et proposition \ref{311}) et, si $i=1$, une solution monomiale minimale de $(E_{p_{i}^{\alpha_{i}}})$ de taille $2^{n}$ est irréductible (théorème \ref{311b} et proposition \ref{311} avec $n \geq \alpha_{i}$, et même $n=\alpha_{1}$ avec le lemme \ref{361}).
\\
\\ii) Si $l=2 \times {\rm ppcm}(l_{i},~1 \leq i \leq r)$. Il existe nécessairement un $i$ tel que $l_{i}=2^{n-1}$. Supposons pour commencer que $i \neq 1$. Si $l_{1}<2^{n-1}$ alors on peut faire la même chose que dans le point B) du cas $N$ impair en prenant $j=1$. Si $l_{1}=2^{n-1}$ alors on peut également procéder de la même façon que dans le point B) du cas $N$ impair en utilisant en plus que $M_{l_{1}}(k+p_{1}^{\alpha_{1}}\mathbb{Z},\ldots,k+p_{1}^{\alpha_{i}}\mathbb{Z})=(1+p_{1}^{\alpha_{1}}\mathbb{Z})Id$, puisque $l_{1} \neq 2$ (lemme \ref{35}). On suppose donc maintenant que $i=1$.
\\
\\Comme $M_{l_{1}}(k+p_{1}^{\alpha_{1}}\mathbb{Z},\ldots,k+p_{1}^{\alpha_{1}}\mathbb{Z})=(1+p_{1}^{\alpha_{1}}\mathbb{Z})Id$, il existe nécessairement un $j$ compris entre 2 et $r$ tel que $l_{j}=2^{n-1}$. En effet dans le cas contraire, on aurait $l={\rm ppcm}(l_{i},~1 \leq i \leq r)=2^{n-1}<2^{n}$. On peut donc procéder de la même façon qu'au dessus.
\\
\\On arrive donc à une absurdité dans les deux cas. Ainsi, la solution $\overline{k}$-monomiale minimale de $(E_{N})$ est irréductible.

\end{proof}

\begin{corollary}
\label{325}

Soient $N>2$, $n \geq 2$. On suppose que 16 ne divise pas $N$. Toutes les solutions monomiales minimales non nulles de $(E_{N})$ de taille $2^{n}$ sont irréductibles.

\end{corollary}

\begin{proof}

On a deux cas, soit $N$ est impair et alors le résultat découle du théorème précédent, soit $N$ est pair non divisible par 16. Dans ce dernier cas, on a donc (en reprenant les notations du théorème précédent) $1 \leq \alpha_{1} \leq 3$ et donc $n=2$ ou $n \geq 3 \geq \alpha_{1}$. Le résultat découle alors du théorème précédent.

\end{proof}

\begin{examples}
{\rm On donne ci-dessous quelques applications immédiates des résultats qui précèdent.
\begin{itemize}
\item $N=51=3 \times 17$. La solution $\overline{6}$-monomiale minimale de $(E_{N})$ est de taille $8$ et est donc irréductible.
\item $N=62=2 \times 31$. La solution $\overline{40}$-monomiale minimale de $(E_{N})$ est de taille $16$ et est donc irréductible.
\item $N=56=2^{3} \times 7$. La solution $\overline{14}$-monomiale minimale de $(E_{N})$ est de taille $8$ et est donc irréductible.
\item $N=80=2^{4} \times 5$. La solution $\overline{50}$-monomiale minimale de $(E_{80})$ est de taille 16 et est donc irréductible.
\end{itemize}
}
\end{examples}

\begin{remark}
{\rm Le résultat du corollaire \ref{325} n'est plus vrai si $N$ est divisible par 16. En effet, si $N$ est divisible par 16 alors les solutions monomiales minimales de $(E_{N})$ dont la taille est une puissance de deux peuvent être réductibles ou irréductibles. On a par exemple :
\begin{itemize}
\item la solution $\overline{50}$-monomiale minimale de $(E_{80})$ est irréductible de taille 16;
\item la solution $\overline{20}$-monomiale minimale de $(E_{80})$ est réductible de taille 8 (voir proposition \ref{316}).
\end{itemize}
}
\end{remark}

\begin{proposition}
\label{325bis}

Soient $N \geq 2$ non divisible par 3. Toutes les solutions monomiales minimales de $(E_{N})$ de taille $6$ sont irréductibles.

\end{proposition}

\begin{proof}

Supposons qu'il existe un entier naturel $k$ tel que la solution $\overline{k}$-monomiale minimale de $(E_{N})$ est de taille $l=6$. On suppose par l'absurde que celle-ci est réductible. Comme $\overline{k} \neq \pm \overline{1}$ (puisque $l \neq 3$), la solution $\overline{k}$-monomiale minimale de $(E_{N})$ est la somme de deux solutions de taille 4, ce qui implique $\overline{k}^{2}=\overline{0}$ ou $\overline{k}^{2}=\overline{2}$. Si $\overline{k}^{2}=\overline{2}$ alors $l=4$. Donc, $\overline{k}^{2}=\overline{0}$. Ainsi, on a 
\[M_{6}(\overline{k},\ldots,\overline{k})=\begin{pmatrix}
   \overline{k^{6}}-5\overline{k}^{4}+6\overline{k}^{2}-\overline{1} & -\overline{k^{5}}+4\overline{k}^{3}-3\overline{k} \\
    \overline{k^{5}}-4\overline{k}^{3}+3\overline{k}    & -\overline{k^{4}}+3\overline{k}^{2}-\overline{1} 
   \end{pmatrix}=\begin{pmatrix}
    -\overline{1} & -3\overline{k} \\
    3\overline{k}    & -\overline{1} 
   \end{pmatrix}.\]

\noindent On a donc $\overline{3k}=\overline{0}$, ce qui implique $\overline{k}=\overline{0}$ (puisque $\overline{3}$ est inversible modulo $N$). On a donc $l=2$ ce qui est absurde. Donc, la solution $\overline{k}$-monomiale minimale de $(E_{N})$ est irréductible.

\end{proof}

\begin{proposition}
\label{326}

Soient $N \geq 2$ et $n \in \mathbb{N}^{*}$. Soit $p$ un nombre premier impair premier avec $N$. Toutes les solutions monomiales minimales non nulles de $(E_{N})$ de taille $2p^{n}$ sont irréductibles.

\end{proposition}

\begin{proof}

Soit $N=p_{1}^{\alpha_{1}} \ldots p_{r}^{\alpha_{r}}$. Supposons qu'il existe un entier naturel $k$ tel que la solution $\overline{k}$-monomiale minimale de $(E_{N})$ est de taille $2p^{n}$. Si $p=3$ et $n=1$ alors le résultat découle de la proposition précédente. On suppose donc maintenant $p^{n}>3$. Pour tout $i$ dans $[\![1;r]\!]$, on note $l_{i}$ la taille de la solution $(k+p_{i}^{\alpha_{i}}\mathbb{Z})$-monomiale minimale de $(E_{p_{i}^{\alpha_{i}}})$ et $l$ la taille de la solution $\overline{k}$-monomiale minimale de $(E_{N})$. Par hypothèse, $l=2p^{n}$. 
\\
\\Par la proposition \ref{37}, on a $l=2p^{n}={\rm ppcm}(l_{i},~1 \leq i \leq r)$ ou à $l=2p^{n}=2 \times {\rm ppcm}(l_{i},~1 \leq i \leq r)$.
\\
\\Supposons par l'absurde que la solution $\overline{k}$-monomiale minimale de $(E_{N})$ est réductible. Il existe une solution $(\overline{x},\overline{k},\ldots,\overline{k},\overline{y})$ de $(E_{N})$ de taille $3 \leq l' \leq l-1$. Il existe $\epsilon \in \{-1,1\}$ tel que $M_{l'}(\overline{x},\overline{k},\ldots,\overline{k},\overline{y})=\overline{\epsilon} Id$. On distingue deux cas.
\\
\\ \uwave{i) $l=2p^{n}={\rm ppcm}(l_{i},~1 \leq i \leq r)$ :} Supposons qu'il existe $i$ dans $[\![1;r]\!]$ tel que $l_{i}=2p^{n}$. Comme $p$ est impair et ne divise pas $N$, la solution $(k+p_{i}^{\alpha_{i}}\mathbb{Z})$-monomiale minimale de $(E_{p_{i}^{\alpha_{i}}})$ est irréductible (théorème \ref{311b} et proposition \ref{311}). Or, $(x+p_{i}^{\alpha_{i}}\mathbb{Z},k+p_{i}^{\alpha_{i}}\mathbb{Z},\ldots,k+p_{i}^{\alpha_{i}}\mathbb{Z},y+p_{i}^{\alpha_{i}}\mathbb{Z})$ est une solution de $(E_{p_{i}^{\alpha_{i}}})$ de taille $3 \leq l' \leq l-1$. On peut donc l'utiliser pour réduire la solution $(k+p_{i}^{\alpha_{i}}\mathbb{Z})$-monomiale minimale de $(E_{p_{i}^{\alpha_{i}}})$, ce qui est une contradiction.
\\
\\Ainsi, $r \geq 2$ et, pour tout $h$ dans $[\![1;r]\!]$, $l_{h}=2p^{\beta_{h}}$ avec $0 \leq \beta_{h} \leq n-1$ ou $l_{h}$ est une puissance de $p$. De plus, il existe $i$ et $j$ dans $[\![1;r]\!]$ tel que $l_{i}=p^{n}$ et $l_{j}$ est pair. 
\\
\\$(x+p_{j}^{\alpha_{j}}\mathbb{Z},k+p_{j}^{\alpha_{j}}\mathbb{Z},\ldots,k+p_{j}^{\alpha_{j}}\mathbb{Z},y+p_{j}^{\alpha_{j}}\mathbb{Z})$ est une solution de $(E_{p_{j}^{\alpha_{j}}})$. Si $l_{j}=2$ alors, par le lemme \ref{plus}, $l'$ est pair. Si $l_{j}$ est divisible par $p$ alors la solution $(k+p_{j}^{\alpha_{j}}\mathbb{Z})$-monomiale minimale de $(E_{p_{j}^{\alpha_{j}}})$ est irréductible (théorème \ref{311b} et proposition \ref{311}). Ainsi, par le lemme \ref{plus}, $l' \equiv 0,2 [l_{j}]$. En particulier, $l'$ est pair.
\\
\\Comme $l_{i}$ est impair, la solution $(k+p_{i}^{\alpha_{i}}\mathbb{Z})$-monomiale minimale de $(E_{p_{i}^{\alpha_{i}}})$ est irréductible. Comme $(x+p_{i}^{\alpha_{i}}\mathbb{Z},k+p_{i}^{\alpha_{j}}\mathbb{Z},\ldots,k+p_{i}^{\alpha_{j}}\mathbb{Z},y+p_{i}^{\alpha_{i}}\mathbb{Z})$ est une solution de $(E_{p_{i}^{\alpha_{i}}})$, on a, par le lemme \ref{plus}, $l' \equiv 0 [p^{n}]$ ou $l' \equiv 2 [p^{n}]$, c'est-à-dire $l'=p^{n}$ ou $l'=p^{n}+2$ (puisque $l'<2p^{n}$). On a donc $l'$ impair. On aboutit ainsi à une contradiction.
\\
\\\uwave{ii) $l=2p^{n}=2 \times {\rm ppcm}(l_{i},~1 \leq i \leq r)$ :} On a $r \geq 2$ et ${\rm ppcm}(l_{i},~1 \leq i \leq r)=p^{n}$. Ainsi, pour tout $h$ dans $[\![1;r]\!]$, $l_{h}$ est une puissance de $p$ et $p \leq l_{h} \leq p^{n}$. De plus, il existe $i$ dans $[\![1;r]\!]$ tel que $l_{i}=p^{n}$. En particulier, $p_{i}^{\alpha_{i}} \neq 2$ car $l_{i} \neq 2,3$. Comme $l_{i}$ est impair, la solution $(k+p_{i}^{\alpha_{i}}\mathbb{Z})$-monomiale minimale de $(E_{p_{i}^{\alpha_{i}}})$ est irréductible. Puisque $(x+p_{i}^{\alpha_{i}}\mathbb{Z},k+p_{i}^{\alpha_{i}}\mathbb{Z},\ldots,k+p_{i}^{\alpha_{i}}\mathbb{Z},y+p_{i}^{\alpha_{i}}\mathbb{Z})$ est une solution de $(E_{p_{i}^{\alpha_{i}}})$, on a, par le lemme \ref{plus}, $l' \equiv 0 [p^{n}]$ ou $l' \equiv 2 [p^{n}]$, c'est-à-dire $l'=p^{n}$ ou $l'=p^{n}+2$ (puisque $l'<2p^{n}$). Donc, la solution $\overline{k}$-monomiale minimale de $(E_{N})$ est nécessairement égale à la somme d'une solution de taille $p^{n}$ avec une solution de taille $p^{n}+2$ ou à la somme d'une solution de taille $p^{n}+2$ avec une solution de taille $p^{n}$. On suppose que $l'=l_{i}=p^{n}$. Par le lemme \ref{plus}, $x+p_{i}^{\alpha_{i}}\mathbb{Z}=y+p_{i}^{\alpha_{i}}\mathbb{Z}=k+p_{i}^{\alpha_{i}}\mathbb{Z}$. Il existe $\alpha \in \{-1,1\}$ tel que $M_{l_{i}}(k+p_{i}^{\alpha_{i}}\mathbb{Z},\ldots,k+p_{i}^{\alpha_{i}}\mathbb{Z})=(\alpha+p_{i}^{\alpha_{i}}\mathbb{Z})Id$.
\\
\\De plus, comme tous les $l_{h}$ sont impairs, il existe nécessairement un $j \neq i$ dans $[\![1;r]\!]$ tel que $p_{j}^{\alpha_{j}} \neq 2$ et $M_{l_{j}}(k+p_{j}^{\alpha_{j}}\mathbb{Z},\ldots,k+p_{j}^{\alpha_{j}}\mathbb{Z})=(-\alpha+p_{j}^{\alpha_{j}}\mathbb{Z})Id$ (car sinon on aurait $l={\rm ppcm}(l_{h},~1 \leq h \leq r)$ par le lemme \ref{39}). On note $l_{j}=p^{\beta_{j}}$. Puisque $(x+p_{j}^{\alpha_{j}}\mathbb{Z},k+p_{j}^{\alpha_{j}}\mathbb{Z},\ldots,k+p_{j}^{\alpha_{j}}\mathbb{Z},y+p_{j}^{\alpha_{j}}\mathbb{Z})$ est une solution de $(E_{p_{j}^{\alpha_{j}}})$ de taille $l' \equiv 0 [p^{\beta_{j}}]$, on a, par le lemme \ref{plus}, $x+p_{j}^{\alpha_{j}}\mathbb{Z}=y+p_{j}^{\alpha_{j}}\mathbb{Z}=k+p_{j}^{\alpha_{j}}\mathbb{Z}$. On a alors 
\begin{eqnarray*}
M &=& M_{l'}(x+p_{j}^{\alpha_{j}}\mathbb{Z},k+p_{j}^{\alpha_{j}}\mathbb{Z},\ldots,k+p_{j}^{\alpha_{j}}\mathbb{Z},y+p_{j}^{\alpha_{j}}\mathbb{Z}) \\
  &=& M_{l_{j}}(k+p_{j}^{\alpha_{j}}\mathbb{Z},\ldots,k+p_{j}^{\alpha_{j}}\mathbb{Z})^{p^{n-\beta_{j}}} \\
	&=& (-\alpha+p_{j}^{\alpha_{j}}\mathbb{Z})^{p^{n-\beta_{j}}}Id \\
	&=& (-\alpha+p_{j}^{\alpha_{j}}\mathbb{Z})Id~~~~(p~{\rm impair}).
\end{eqnarray*}

\noindent Si $\epsilon=-\alpha$ alors $\epsilon+p_{i}^{\alpha_{i}}\mathbb{Z}=-\alpha+p_{i}^{\alpha_{i}}\mathbb{Z}$. Or, par ce qui précède, $\epsilon+p_{i}^{\alpha_{i}}\mathbb{Z}=\alpha+p_{i}^{\alpha_{i}}\mathbb{Z} \neq -\alpha+p_{i}^{\alpha_{i}}\mathbb{Z}$ (puisque $p_{i}^{\alpha_{i}} \neq 2$). Si $\epsilon=\alpha$ alors $\epsilon+p_{j}^{\alpha_{h}}\mathbb{Z}=\alpha+p_{j}^{\alpha_{j}}\mathbb{Z}$. Or, par ce qui précède, $\epsilon+p_{j}^{\alpha_{j}}\mathbb{Z}=-\alpha+p_{j}^{\alpha_{j}}\mathbb{Z} \neq \alpha+p_{j}^{\alpha_{j}}\mathbb{Z}$ (puisque $p_{j}^{\alpha_{j}} \neq 2$). On arrive ainsi à une absurdité.
\\
\\On aboutit donc à une absurdité dans tous les cas. Ainsi, la solution $\overline{k}$-monomiale minimale de $(E_{N})$ est irréductible.

\end{proof}

\begin{examples}
{\rm On donne ci-dessous quelques applications immédiates des résultats qui précèdent.
\begin{itemize}
\item $N=34=2 \times 17$. La solution $\overline{20}$-monomiale minimale de $(E_{N})$ est de taille $18=2\times 3^{2}$ et est donc irréductible.
\item $N=46=2 \times 23$. La solution $\overline{34}$-monomiale minimale de $(E_{N})$ est de taille $22=2 \times 11$ et est donc irréductible.
\item $N=185=5 \times 37$. La solution $\overline{20}$-monomiale minimale de $(E_{N})$ est de taille $18=2\times 3^{2}$ et est donc irréductible.
\item $N=5757=3 \times 19 \times 101$. La solution $\overline{30}$-monomiale minimale de $(E_{N})$ est de taille $50=2\times 5^{2}$ et est donc irréductible.
\end{itemize}
}
\end{examples}

\begin{remarks}
{\rm i) Le résultat n'est plus vrai en général si on considère des solutions de taille $2p^{n}$ avec $p$ qui divise $N$. En effet, prenons par exemple $N=25$ et $k=5$. La solution $\overline{5}$-monomiale minimale de $(E_{N})$ est de taille $l=10$ (proposition \ref{316}) et on peut la réduire avec la solution $(-\overline{5},\overline{5},\overline{5},-\overline{5})$. Cela dit, on peut également trouver des solutions de taille $2p^{n}$ avec $p$ qui divise $N$ qui sont irréductibles. En effet, prenons par exemple $N=30$ et $k=15$. La solution $\overline{15}$-monomiale minimale de $(E_{N})$ est de taille $l=6$ et pourtant celle-ci est irréductible (Théorème \ref{32}).
\\
\\ii) Le résultat n'est plus vrai en général si on considère des solutions de taille $2^{m}p^{n}$, et cela même si $p$ ne divise pas $N$. En effet, prenons par exemple $N=77=7 \times 11$ et $k=3$. La solution $\overline{k}$-monomiale minimale de $(E_{N})$ est de taille $l=40=2^{3} \times 5$ et on peut la réduire avec la solution $(\overline{66},\overline{3},\overline{3},\overline{3},\overline{3},\overline{3},\overline{3},\overline{3},\overline{3},\overline{3},\overline{3},\overline{66})$.
\\
\\iii) Le résultat n'est plus vrai en général si on considère des solutions de taille $2p^{n}q^{m}$. En effet, prenons par exemple $N=38$ et $k=11$. La solution $\overline{k}$-monomiale minimale de $(E_{N})$ est de taille $l=30=2\times 3 \times 5$ et on peut la réduire avec la solution $(\overline{19},\overline{11},\overline{11},\overline{11},\overline{11},\overline{11},\overline{11},\overline{11},\overline{11},\overline{11},\overline{11},\overline{19})$.
}
\end{remarks}

\noindent On termine avec le résultat ci-dessous :

\begin{proposition}
\label{327}

Soient $N \geq 2$ impair et $n \in \mathbb{N}^{*}$. Soit $p$ un nombre premier impair premier avec $N$. Toutes les solutions monomiales minimales non nulles de $(E_{N})$ de taille $4p^{n}$ sont irréductibles.

\end{proposition}

\begin{proof}

Soit $N=p_{1}^{\alpha_{1}} \ldots p_{r}^{\alpha_{r}}$ impair. Supposons qu'il existe un entier naturel $k$ tel que la solution $\overline{k}$-monomiale minimale de $(E_{N})$ est de taille $4p^{n}$. Pour tout $i$ dans $[\![1;r]\!]$, on note $l_{i}$ la taille de la solution $(k+p_{i}^{\alpha_{i}}\mathbb{Z})$-monomiale minimale de $(E_{p_{i}^{\alpha_{i}}})$ et $l$ la taille de la solution $\overline{k}$-monomiale minimale de $(E_{N})$. Par hypothèse, $l=4p^{n}$. Il existe $\alpha \in \{-1,1\}$ tel que $M_{l}(\overline{k},\ldots,\overline{k})=\overline{\alpha} Id$.
\\
\\Par la proposition \ref{37}, on a $l={\rm ppcm}(l_{i},~1 \leq i \leq r)$ ou à $l=2 \times {\rm ppcm}(l_{i},~1 \leq i \leq r)$.
\\
\\Supposons par l'absurde que la solution $\overline{k}$-monomiale minimale de $(E_{N})$ est réductible. Il existe une solution $(\overline{x},\overline{k},\ldots,\overline{k},\overline{y})$ de $(E_{N})$ de taille $3 \leq l' \leq l-1$. Il existe $\epsilon \in \{-1,1\}$ tel que $M_{l'}(\overline{x},\overline{k},\ldots,\overline{k},\overline{y})=\overline{\epsilon} Id$. On distingue deux cas.
\\
\\ \uwave{i) $l=4p^{n}={\rm ppcm}(l_{i},~1 \leq i \leq r)$ :} Pour tout $h$ dans $[\![1;r]\!]$, $l_{h}$ est de la forme $p^{a}$, $2p^{a}$ou $4p^{a}$.
\\
\\Supposons qu'il existe $i$ dans $[\![1;r]\!]$ tel que $l_{i}=4p^{n}$. Comme $p \neq p_{i}$, la solution $(k+p_{i}^{\alpha_{i}}\mathbb{Z})$-monomiale minimale de $(E_{p_{i}^{\alpha_{i}}})$ est irréductible. Or, $(x+p_{i}^{\alpha_{i}}\mathbb{Z},k+p_{i}^{\alpha_{i}}\mathbb{Z},\ldots,k+p_{i}^{\alpha_{i}}\mathbb{Z},y+p_{i}^{\alpha_{i}}\mathbb{Z})$ est une solution de $(E_{p_{i}^{\alpha_{i}}})$ de taille $3 \leq l' \leq l-1$. On peut donc l'utiliser pour réduire la solution $(k+p_{i}^{\alpha_{i}}\mathbb{Z})$-monomiale minimale de $(E_{p_{i}^{\alpha_{i}}})$, ce qui est une contradiction.
\\
\\Ainsi, $r \geq 2$ et pour tout $h$ dans $[\![1;r]\!]$ on a $l_{h} \neq 4p^{n}$. Par conséquent, il existe nécessairement deux entiers $i$ et $j$ tels que $l_{i}=4p^{a}$ avec $a<n$ et $l_{j}=2^{b}p^{n} \in \{p^{n},2p^{n}\}$.
\\
\\Par le lemme \ref{34}, $M_{l_{i}}(k+p_{i}^{\alpha_{i}}\mathbb{Z},\ldots,k+p_{i}^{\alpha_{i}}\mathbb{Z})=(-1+p_{i}^{\alpha_{i}}\mathbb{Z}) Id$, puisque $p_{i}$ est impair. Donc, 
\[M_{l}(k+p_{i}^{\alpha_{i}}\mathbb{Z},\ldots,k+p_{i}^{\alpha_{i}}\mathbb{Z})=M_{l_{i}}(k+p_{i}^{\alpha_{i}}\mathbb{Z},\ldots,k+p_{i}^{\alpha_{i}}\mathbb{Z})^{p^{n-a}}=(-1+p_{i}^{\alpha_{i}}\mathbb{Z}) Id~(p~{\rm impair)}.\]
\noindent En particulier, $\alpha+p_{i}^{\alpha_{i}}\mathbb{Z}=-1+p_{i}^{\alpha_{i}}\mathbb{Z}$. De plus, 
\[M_{l}(k+p_{j}^{\alpha_{j}}\mathbb{Z},\ldots,k+p_{j}^{\alpha_{j}}\mathbb{Z})=M_{l_{j}}(k+p_{j}^{\alpha_{j}}\mathbb{Z},\ldots,k+p_{j}^{\alpha_{j}}\mathbb{Z})^{2^{2-b}}=(1+p_{j}^{\alpha_{j}}\mathbb{Z}) Id~(2-b \geq 1).\]
\noindent En particulier, $\alpha+p_{j}^{\alpha_{j}}\mathbb{Z}=1+p_{j}^{\alpha_{j}}\mathbb{Z}$. Comme $p_{i}^{\alpha_{i}}, p_{j}^{\alpha_{j}} \neq 2$, on arrive à une absurdité.
\\
\\\uwave{ii) $l=4p^{n}=2 \times {\rm ppcm}(l_{i},~1 \leq i \leq r)$ :} On a $r \geq 2$ et ${\rm ppcm}(l_{i},~1 \leq i \leq r)=2p^{n}$. Ainsi, pour tout $h$ dans $[\![1;r]\!]$, $l_{h}$ est de la forme $p^{a}$ ou $2p^{a}$.
\\
\\Supposons qu'il existe un $i$ dans $[\![1;r]\!]$ tel que $l_{i}=2p^{n}$. 
\\
\\Supposons par l'absurde que  pour tout $h$ dans $[\![1;r]\!]$ $l_{h}=2p^{a_{h}}$. Par le lemme \ref{34}, on a pour tout $h$ dans $[\![1;r]\!]$, 
\[M_{2p^{n}}(k+p_{h}^{\alpha_{h}}\mathbb{Z},\ldots,k+p_{h}^{\alpha_{h}}\mathbb{Z})=(M_{l_{h}}(k+p_{h}^{\alpha_{h}}\mathbb{Z},\ldots,k+p_{h}^{\alpha_{h}}\mathbb{Z}))^{p^{n-a_{h}}}=(-1+p_{h}^{\alpha_{h}}\mathbb{Z})^{p^{n-a_{h}}} Id=(-1+p_{h}^{\alpha_{h}}\mathbb{Z}) Id.\]
\noindent Par le lemme chinois, on a $l$ divise $2p^{n}$ ce qui est absurde. Donc, il existe $j$ dans $[\![1;r]\!]$ tel que $l_{j}=p^{a_{j}}$.
\\
\\Comme $p \neq p_{i}$, la solution $(k+p_{i}^{\alpha_{i}}\mathbb{Z})$-monomiale minimale de $(E_{p_{i}^{\alpha_{i}}})$ est irréductible. Puisque le $l'$-uplet $(x+p_{i}^{\alpha_{i}}\mathbb{Z},k+p_{i}^{\alpha_{i}}\mathbb{Z},\ldots,k+p_{i}^{\alpha_{i}}\mathbb{Z},y+p_{i}^{\alpha_{i}}\mathbb{Z})$ est une solution de $(E_{p_{i}^{\alpha_{i}}})$, on a, par le lemme \ref{plus}, $l' \equiv 0 [2p^{n}]$ ou $l' \equiv 2 [2p^{n}]$, c'est-à-dire $l'=2p^{n}$ ou $l'=2p^{n}+2$ (puisque $l'<4p^{n}$). Donc, la solution $\overline{k}$-monomiale minimale de $(E_{N})$ est nécessairement égale à la somme d'une solution de taille $2p^{n}$ avec une solution de taille $2p^{n}+2$ ou à la somme d'une solution de taille $2p^{n}+2$ avec une solution de taille $2p^{n}$. On peut donc supposer que $l'=2p^{n}$. Par le lemme \ref{plus}, $x+p_{i}^{\alpha_{i}}\mathbb{Z}=y+p_{i}^{\alpha_{i}}\mathbb{Z}=k+p_{i}^{\alpha_{i}}\mathbb{Z}$. Ainsi, par le lemme \ref{34}, $M_{l'}(x+p_{i}^{\alpha_{i}}\mathbb{Z},k+p_{i}^{\alpha_{j}}\mathbb{Z},\ldots,k+p_{i}^{\alpha_{i}}\mathbb{Z},y+p_{i}^{\alpha_{i}}\mathbb{Z})=(-1+p_{i}^{\alpha_{i}}\mathbb{Z})Id$.
\\
\\Puisque $(x+p_{j}^{\alpha_{j}}\mathbb{Z},k+p_{j}^{\alpha_{j}}\mathbb{Z},\ldots,k+p_{j}^{\alpha_{j}}\mathbb{Z},y+p_{j}^{\alpha_{j}}\mathbb{Z})$ est une solution de $(E_{p_{j}^{\alpha_{j}}})$ et $l' \equiv 0 [l_{j}]$, on a, par le lemme \ref{plus}, $x+p_{j}^{\alpha_{j}}\mathbb{Z}=y+p_{j}^{\alpha_{j}}\mathbb{Z}=k+p_{j}^{\alpha_{j}}\mathbb{Z}$. Ainsi,  puisque $p_{j}$ est impair, on a, par le lemme \ref{34}, 
\[M_{l'}(x+p_{i}^{\alpha_{i}}\mathbb{Z},k+p_{i}^{\alpha_{i}}\mathbb{Z},\ldots,k+p_{i}^{\alpha_{i}}\mathbb{Z},y+p_{i}^{\alpha_{i}}\mathbb{Z})=M_{l_{j}}(k+p_{j}^{\alpha_{j}}\mathbb{Z},\ldots,k+p_{j}^{\alpha_{j}}\mathbb{Z})^{2p^{n-a_{j}}}=(1+p_{j}^{\alpha_{j}}\mathbb{Z})Id.\]

\noindent Comme $p_{i}^{\alpha_{i}}, p_{j}^{\alpha_{j}} \neq 2$, on arrive à une absurdité.
\\
\\Supposons maintenant que pour tout $h$ dans $[\![1;r]\!]$ $l_{h} \neq 2p^{n}$. Il existe $i,j$ dans $[\![1;r]\!]$ tels que $l_{i}=p^{n}$ et $l_{j}=2p^{a_{j}}$ ($a_{j}<n$). Comme $l_{i}$ est impair, la solution $(k+p_{i}^{\alpha_{i}}\mathbb{Z})$-monomiale minimale de $(E_{p_{i}^{\alpha_{i}}})$ est irréductible. Puisque $(x+p_{i}^{\alpha_{i}}\mathbb{Z},k+p_{i}^{\alpha_{i}}\mathbb{Z},\ldots,k+p_{i}^{\alpha_{i}}\mathbb{Z},y+p_{i}^{\alpha_{i}}\mathbb{Z})$ est une solution de $(E_{p_{i}^{\alpha_{i}}})$, on a, par le lemme \ref{plus}, $l' \equiv 0 [p^{n}]$ ou $l' \equiv 2 [p^{n}]$.
\\
\\De plus, comme $p \neq p_{j}$, la solution $(k+p_{j}^{\alpha_{j}}\mathbb{Z})$-monomiale minimale de $(E_{p_{j}^{\alpha_{j}}})$ est irréductible. Puisque $(x+p_{j}^{\alpha_{j}}\mathbb{Z},k+p_{j}^{\alpha_{j}}\mathbb{Z},\ldots,k+p_{j}^{\alpha_{j}}\mathbb{Z},y+p_{j}^{\alpha_{j}}\mathbb{Z})$ est une solution de $(E_{p_{j}^{\alpha_{j}}})$, on a, par le lemme \ref{plus}, $l' \equiv 0 [l_{j}]$ ou $l' \equiv 2 [l_{j}]$. Ainsi, $l'$ est pair et donc $l' \in \{2p^{n},2p^{n}+2\}$ (puisque $p$ est impair). Donc, la solution $\overline{k}$-monomiale minimale de $(E_{N})$ est nécessairement égale à la somme d'une solution de taille $2p^{n}$ avec une solution de taille $2p^{n}+2$ ou à la somme d'une solution de taille $2p^{n}+2$ avec une solution de taille $2p^{n}$. On peut donc supposer que $l'=2p^{n}$.
\\
\\Par le lemme \ref{plus}, $x+p_{i}^{\alpha_{i}}\mathbb{Z}=y+p_{i}^{\alpha_{i}}\mathbb{Z}=k+p_{i}^{\alpha_{i}}\mathbb{Z}$. Ainsi, par le lemme \ref{34}, 
\[M_{l'}(x+p_{i}^{\alpha_{i}}\mathbb{Z},k+p_{i}^{\alpha_{i}}\mathbb{Z},\ldots,k+p_{i}^{\alpha_{i}}\mathbb{Z},y+p_{i}^{\alpha_{i}}\mathbb{Z})=M_{p^{n}}(k+p_{i}^{\alpha_{j}}\mathbb{Z},\ldots,k+p_{i}^{\alpha_{i}}\mathbb{Z})^{2}=(1+p_{i}^{\alpha_{i}}\mathbb{Z})Id.\] \noindent Comme $l' \equiv 0 [l_{j}]$, on a, par le lemme \ref{plus}, $x+p_{j}^{\alpha_{j}}\mathbb{Z}=y+p_{j}^{\alpha_{j}}\mathbb{Z}=k+p_{j}^{\alpha_{j}}\mathbb{Z}$. Ainsi, puisque $p_{j}$ est impair, on a, par le lemme \ref{34} : 
\begin{eqnarray*}
M &=& M_{l'}(x+p_{j}^{\alpha_{j}}\mathbb{Z},k+p_{j}^{\alpha_{j}}\mathbb{Z},\ldots,k+p_{j}^{\alpha_{j}}\mathbb{Z},y+p_{j}^{\alpha_{j}}\mathbb{Z}) \\
  &=& M_{l_{j}}(k+p_{j}^{\alpha_{j}}\mathbb{Z},\ldots,k+p_{j}^{\alpha_{j}}\mathbb{Z})^{p^{n-a_{j}}} \\
	&=& (-1+p_{j}^{\alpha_{j}}\mathbb{Z})^{p^{n-a_{j}}}Id \\
	&=&(-1+p_{j}^{\alpha_{j}}\mathbb{Z}) Id.
\end{eqnarray*}

\noindent Comme $p_{i}^{\alpha_{i}}, p_{j}^{\alpha_{j}} \neq 2$, on arrive à une absurdité.
\\
\\On aboutit donc à une absurdité dans tous les cas. Ainsi, la solution $\overline{k}$-monomiale minimale de $(E_{N})$ est irréductible.

\end{proof}

\begin{examples}
{\rm On donne ci-dessous quelques applications immédiates des résultats qui précèdent.
\begin{itemize}
\item $N=69=3 \times 23$. La solution $\overline{12}$-monomiale minimale de $(E_{N})$ est de taille $44=4 \times 11$ et est donc irréductible.
\item $N=117=3^{2} \times 13$. La solution $\overline{18}$-monomiale minimale de $(E_{N})$ est de taille $28=4 \times 7$ et est donc irréductible.
\end{itemize}
}
\end{examples}

\begin{remarks}
{\rm i) Le résultat n'est plus vrai en général si on considère des entiers $N$ pairs. En effet, prenons par exemple $N=14=2 \times 7$ et $k=3$. La solution $\overline{3}$-monomiale minimale de $(E_{N})$ est de taille $l=12=4 \times 3$ et on peut la réduire avec la solution $(\overline{7},\overline{3},\overline{3},\overline{3},\overline{3},\overline{7})$.
\\
\\ii) Le résultat n'est plus vrai en général si on considère des solutions de taille $2^{m}p^{n}$ avec $m \geq 3$. En effet, considérons $N=35=5 \times 7$ et $k=4$. La solution $\overline{k}$-monomiale minimale de $(E_{N})$ est de taille $l=24=2^{3} \times 3$ et on peut la réduire avec la solution $(\overline{14},\overline{4},\overline{4},\overline{4},\overline{4},\overline{14})$.
}
\end{remarks}

\end{document}